\documentclass[a4paper,11pt,a4wide]{amsart}
\makeatletter
\usepackage{amsfonts,delarray,amssymb,amsmath,amsthm,a4,color}
\usepackage{amsmath, euscript}
\usepackage[ansinew]{inputenc}

\newtheorem{theo}{Theorem}
\newtheorem{pro}{Proposition}[section]
\newtheorem{lem}[pro]{Lemma}
\newtheorem{coro}[pro]{Corollary}
\newtheorem{remark}[pro]{Remark}

\def\XXint#1#2#3{{\setbox0=\hbox{$#1{#2#3}{\int}$}
     \vcenter{\hbox{$#2#3$}}\kern-.5\wd0}}

\def\({\left(}
\def\){\right)}
\def\1{\mathbf{1}}
\def\la{\langle}
\def\ra{\rangle}
\def\hal{\frac{1}{2}}
\def\a{\alpha}
\def\b{\beta}
\def\ro{\rho}
\def\g{\gamma}
\def\p{\partial}
\def\dist{\text{dist}\ }
\def\div{\, \mathrm{div} \, }
\def\curl{\, \mathrm{curl} \, }
\def\D{\displaystyle}
\def\ep{\varepsilon}
\def\lep{|\mathrm{log } \, \ep|}
\def\llep{\mathrm{log} \, \lep}

\def\R{\mathbb{R}}
\def\nab{\nabla}
\def\np{\nab^{\perp}}
\def\N{N_\ep}
\def\l{\lambda}

\def\v{\mathrm{\sc{v}}}
\def\pp{\mathrm{\sc{p}}}
\def\tV{\tilde V_\ep}
\def\d{d}
\def\om{\omega}
\def\h{\mathrm{h}}

\def\hci{H_{c_1}}

\def\T{\tilde{S}_\ep} 
\def\E{\mathcal{E}_\ep}

\def\t{\mathrm{T}}
\def\te{T_\ep}
\def\i{\int_{\R^2}}
\def\supp{\mathrm{Supp} \,} 

 \numberwithin{equation}{section}
\title[Mean Field Limits of Ginzburg-Landau]{Mean Field Limits of the Gross-Pitaevskii and  parabolic  Ginzburg-Landau equations}
\author[Sylvia Serfaty]{Sylvia Serfaty}
\address[Sylvia Serfaty]{Sorbonne Universit\'es, UPMC Univ. Paris 06, CNRS, UMR 7598, Laboratoire Jacques-Louis Lions, 4, place Jussieu 75005, Paris, France.
 \newline \& Institut Universitaire de France \newline \& Courant Institute, New York University, 251 Mercer st, NY NY 10012, USA.} 
\email{serfaty@ann.jussieu.fr}

\date{June 26, 2016}
\begin{document}
\maketitle

\begin{abstract}We prove that in a certain asymptotic regime, solutions of the Gross-Pitaevskii equation converge to solutions of the incompressible Euler equation, and solutions to the parabolic  Ginzburg-Landau equations converge to solutions of a 
limiting equation which we identify. 

We work in the setting of the whole plane (and possibly the whole three-dimensional space in the Gross-Pitaevskii case), in the asymptotic limit where $\ep$,  the characteristic lengthscale of the vortices, tends to $0$, and in a situation where the number of vortices  $\N$ blows up as $\ep \to 0$. The requirements are that $\N$ should blow up faster than $\lep$ in the Gross-Pitaevskii case, and at most like $\lep$ in the parabolic  case.  Both results  assume a well-prepared initial condition and regularity of the limiting initial data, and use the regularity of the solution to the limiting equations. 
  
In the case of the parabolic Ginzburg-Landau equation,   the limiting mean-field dynamical law that we identify coincides with the  one proposed by Chapman-Rubinstein-Schatzman and E  in the regime $\N\ll \lep$, but not if $\N$ grows faster.\end{abstract}

\noindent
{\bf keywords: } Ginzburg-Landau, Gross-Pitaevskii, vortices, vortex liquids, mean-field limit, hydrodynamic limit, Euler equation.\\
{\bf MSC: } 35Q56,35K55,35Q55,35Q31,35Q35 

\section{Introduction}
\subsection{Problem and background}
We are interested in the Gross-Pitaevskii equation 
\begin{equation}\label{gp}
i\p_t u= \Delta u + \frac{u}{\ep^2} (1-|u|^2) \quad \text{in} \ \R^2\end{equation} and the parabolic Ginzburg-Landau equation 
\begin{equation}
\label{pgl}
\p_t u = \Delta u +  \frac{u}{\ep^2} (1-|u|^2) \quad \text{in} \ \R^2\end{equation} in the plane, 
and also the three-dimensional version of the Gross-Pitaevskii equation 
 \begin{equation}\label{gp3}
i\p_t u= \Delta u + \frac{u}{\ep^2} (1-|u|^2) \quad \text{in} \ \R^3,\end{equation} 
all in the asymptotic limit $\ep \to 0$. 
 
These are famous equations of mathematical  physics, which arise in the modeling of superfluidity, superconductivity, nonlinear optics, etc.  The Gross-Pitaevskii equation is an important instance of a nonlinear Schr\"odinger equation.   These equations also come in a version with gauge, more suitable for the modeling of superconductivity, but whose essential mathematical features are similar to these, and which we will discuss briefly below. There is also interest in the ``mixed flow" case,  sometimes called complex Ginzburg-Landau equation
\begin{equation}\label{mixed}
(a+ib) \p_t u= \Delta u +  \frac{u}{\ep^2} (1-|u|^2) \quad \text{in} \ \R^2.\end{equation}
For further reference on these models,  one can see e.g.~\cite{t,tt,ak,livre}.

In these equations, the unknown function $u$ is complex-valued, and it can exhibit vortices, which are zeroes of $u$ with non-zero topological degree, and  a core  size on the order of $\ep$. In the plane, these vortices are  points, whereas in the three-dimensional space they are lines. We are interested in  one of the main open problems on Ginzburg-Landau dynamics, which is to understand the dynamics of vortices in the regime in which their number $\N$ blows up as $\ep \to 0$.   The only available results until now are due to Kurzke and Spirn \cite{ks} in the case of \eqref{pgl} and Jerrard and Spirn \cite{js2} in the case of \eqref{gp}. The latter concern a very dilute limit in which $\N$ grows slower than a power of $\log \lep$ as $\ep \to 0$ (more details are given below).

The  dynamics of vortices in \eqref{gp} and \eqref{pgl} was first studied in the case where their number  $N$ is bounded  as $\ep \to 0$ (hence can be assumed to be independent of $\ep$.
It was proven, either in the setting of the whole plane or that of a bounded domain, that, for ``well-prepared" initial data,  after suitable time rescaling,  their limiting positions obey the law
\begin{equation}\label{limW}
\frac{d a_i}{dt}=  (a+ \mathbb{J} b) \nab_i W (a_1, \dots , a_N) 
\end{equation}
where $\mathbb{J}$ is the rotation by $\pi/2$ in the plane, and 
$W$ is in the setting of the plane the so-called Kirchhoff-Onsager energy 
\begin{equation}\label{ko}
W(a_1, \dots, a_N) =- \pi \sum_{i \neq j} d_i d_j \log |a_i-a_j|\end{equation}
where the $d_i $'s are the degrees of the vortices and  are assumed to be initially in $\{1, -1\}$. In the setting of a bounded domain and prescribed Dirichlet boundary data, $W$  is \eqref{ko} with some additional terms accounting for boundary effects. It was introduced and derived in  that context  in the book of Bethuel-Brezis-H\'elein  \cite{bbh} where it was called the ``renormalized energy". 

In other words, the vortices move according to the corresponding flow (gradient, Hamiltonian, or mixed) of their limiting interaction energy $W$.
After some formal results based on matched asymptotics by Pismen-Rubinstein and E in \cite{pr,E}, these results were proven in the setting of a bounded domain  by Lin \cite{li} and Jerrard-Soner \cite{js} in the parabolic case, Colliander-Jerrard \cite{cj1,cj2}  and Lin-Xin \cite{lx2} with later improvements by Jerrard-Spirn \cite{jsp1} in the Schr\"odinger case, and  Kurzke-Melcher-Moser-Spirn \cite{kmms} in the mixed flow case.  In the setting of the whole plane, the analogous results were obtained by Lin-Xin \cite{lx}  in the parabolic case,   Bethuel-Jerrard-Smets \cite{bjs} in the Schr\"odinger case, and Miot \cite{miot}  in the mixed flow case.
A proof based on the idea of relating gradient flows and $\Gamma$-convergence was also given in \cite{ss}, it was the initial motivation for  the abstract scheme of ``$\Gamma$-convergence of gradient flows"  introduced there. 
Generalizations to the case with gauge, pinning terms and applied electric field terms were also studied  \cite{sp1,sp2,ks0,tice,stice}.

All these results  hold for  well-prepared data and  for as long as the points evolving under the dynamical law \eqref{limW} do not collide. 
In the parabolic case, Bethuel-Orlandi-Smets showed in the series of papers \cite{bos1,bos2,bos3} how to lift the well-prepared condition and  handle the difficult issue of collisions and extend the dynamical law \eqref{limW} beyond them.  Results of a similar nature were also obtained in \cite{s2}.

The expected limiting dynamics of  three-dimensional vortex lines under \eqref{gp3} is the binormal flow of a curve, but in contrast to the two-dimensional case there are only partial results towards establishing this rigorously  \cite{jerrard3d}.

\medskip

When the number of points $\N$ blows up as $\ep \to 0$, then it is expected that the limiting system of ODEs \eqref{limW} should be replaced by a mean-field evolution for the density of vortices, or vorticity.
More precisely, 
for a family of functions $u_\ep$, one introduces the supercurrent  $j_\ep$ and the vorticity (or Jacobian) $\mu_\ep$  of the map $u_\ep$ which are defined via 
\begin{equation}\label{defj}
j_\ep:= \la iu_\ep, \nab u_\ep\ra  \qquad \mu_\ep:=\curl j_\ep,
\end{equation}where $\la x,y\ra$ stands for the scalar product in $\mathbb{C}$ as identified with $\R^2$  via $\la x,y\ra = \hal (\bar x y+ \bar y x).$
The vorticity  $\mu_\ep$ plays the same role as the vorticity in classical fluids, the only  difference being that it is essentially quantized at the $\ep$ level,  as can be seen from the asymptotic estimate  $\mu_\ep\simeq 2\pi \sum_i d_i \delta_{a_i}$ as $\ep \to 0$, with $\{a_i\}$ the vortices of $u_\ep$ and $d_i\in \mathbb{Z}$ their degrees (these are the so-called Jacobian estimates, which we will recall in the course of the paper).

The mean-field evolution for $\mu=\lim_{\ep \to 0} \mu_\ep/\N$ can be guessed to be the mean-field limit of \eqref{limW} as $N \to \infty$. Proving this essentially amounts to showing that the limits $\ep \to 0$ and $N\to \infty$ can be interchanged, which is a delicate question.

In the case of the Gross-Pitaevskii equation \eqref{gp}-\eqref{gp3}, it is well-known that the Madelung transform  formally yields that the limiting evolution equation should be the incompressible Euler equation (for this and related questions,  see for instance the survey \cite{danchin}).
In the case of the parabolic Ginzburg-Landau equation, it was proposed, based on heuristic considerations by Chapman-Rubinstein-Schatzman \cite{crs}  and E \cite{E2}, that the limiting equation should be 
\begin{equation}\label{crs}
\p_t \mu-    \div ( \mu\nab h )=0 \qquad h= -\Delta^{-1} \mu.
\end{equation} where $\mu$ is the limit of the vortex density, assumed to be nonnegative.  In fact, both papers   really derived the equation for possibly signed densities, \cite{crs} did it   for  the very similar  model with gauge in a bounded domain, in which case the coupling $h=-\Delta^{-1}\mu$ is replaced by $h=(-\Delta +I)^{-1}\mu$ (see also Section \ref{sec:gauge} below), and \cite{E2}  treated  both situations with and without gauge, also for signed densities, without discussing the domain boundary. 
 
 After this model was proposed, existence, weak notions and  properties of solutions to \eqref{crs}  were studied in \cite{lz,duzhang,sv} (see also \cite{mz} for some related models).  They depend greatly on the regularity of the initial data $\mu$. For $\mu$ a general probability measure, the product $\mu \nab h$ does not make sense, and a weak formulation \`a la Delort \cite{delort} must be used; also uniqueness of solutions can fail, although there is always existence of a unique solution which becomes instantaneously $L^\infty$.  
It also turns out that \eqref{crs} can be interpreted as   the gradient flow (as in \cite{otto,ags})  in the space of probability measures equipped with the $2$-Wasserstein distance,  of the energy  functional 
\begin{equation}\label{mfe}
\Phi(\mu)= \int |\nab h |^2\qquad  h = -\Delta^{-1}\mu,\end{equation}
which is also the mean field limit of the usual Ginzburg-Landau energy. 
The  equation \eqref{crs} was also studied with that point of view in \cite{as} in the bounded domain setting (where the possible entrance and exit of mass creates difficulties). 
This energy point of view  allows one  to envision a possible (and so far unsuccessful)  energetic proof of the convergence of \eqref{pgl}  based on the scheme of Gamma-convergence of gradient flows, as described in \cite{serfaty}.
\medskip

A proof of the  convergence of \eqref{gp}, \eqref{pgl} or \eqref{gp3} to these limiting equations (Euler or \eqref{crs}) has remained elusive for a long time.

The time-independent analogue of this result, i.e.~the convergence of solutions to the static Ginzburg-Landau equations to the time-independent version of \eqref{crs} -- a suitable weak formulation of $\mu \nab h=0$ -- in the regime $\N \gg 1$ (the notation means $\N \to +\infty$ as $\ep \to 0$), and without any assumption on the sign of the vorticity,  was obtained in \cite{ss4}. As is standard, due to its translation-invariance, the  stationary Ginzburg-Landau  equation can be rewritten as a conservation law, namely  that the so-called stress-energy tensor is divergence free. The method of the proof then 
consists in passing to  the limit in that relation,  taking advantage of a good control of the size of the set occupied by vortices. This approach 
seems to fail to extend to the dynamical setting for lack of extension of this good control.

  On the other hand, the proof of convergence of the dynamic equations in  the case of bounded number of vortices is usually  based on examining the expression for the  time-derivative of the energy density and identifying limits for each term (for a  quick description, one may refer to the introduction of \cite{stice}).
This proof also seems very hard to extend to the situation of $\N \gg 1$ for the following reasons~:
\begin{itemize}\item It relies on  estimates on the evolution of the ``energy-excess" $D$, where terms controlled  in $\sqrt{D}$ instead of $D$ arise (and these are not amenable to the use of Gronwall's lemma).
\item Understanding the evolution of the energy density or excess seems to require controlling the speed of each individual vortex, which is difficult when their number gets large while  only averaged information is available.
\item The proof works until the first time of possible collision between vortices, which can in principle  occur in very short time once the number of vortices blow up, so it seems that one needs good control and understanding of the vortices' mutual distances. 
\end{itemize}

Recently Spirn-Kurzke \cite{ks} and Jerrard-Spirn \cite{js} were however able to make the first breakthrough by  accomplishing this program  for \eqref{pgl} and \eqref{gp} respectively, in the case where $\N$ grows very slowly: $\N \ll (\log \log \lep)^{1/4}$ in the parabolic Ginzburg-Landau case, and $\N \ll (\log \lep)^{1/2}$ in the two-dimensional  Gross-Pitaevskii case, assuming some  specific well-preparedness conditions on the initial data.  Relying on their previous work \cite{ks0,jsp1},  they showed that the method of proof for finite  number  of vortices can be made more quantitative and  pushed beyond bounded $\N$, controlling the vortex distances and proving  that their positions remain close to those of the  $\N$ points solving the ODE system \eqref{limW}, and then finally passing to the limit for that system by applying classical  ``point vortex methods", in the manner of, say, Schochet \cite{scho}. 
There is however  very little hope for extending such an approach  to larger values of $\N$. 
\smallskip

By a different proof method based on the evolution of a ``modulated energy", we will establish 
\begin{itemize}
\item the mean-field limit evolution of \eqref{gp}--\eqref{gp3}  in the regime $\lep \ll \N \ll 1/\ep$ 
\item the mean-field limit evolution of  \eqref{pgl} in the regime $1\ll \N \le O( \lep)$ 
\end{itemize} 
both at the level of convergence of the supercurrents, and not just the vorticity. We note that the condition  $\N \le O(\lep)$ allows to reach a physically very relevant regime: in the case of the equation with gauge, $\lep$ is the order of the number of vortices that are present just after one reached the so-called first critical field $\hci$, itself of the order of $\lep$ (cf. \cite{livre}).
 
Our method relies on the assumed regularity of the solution to the limiting equation, thus is restricted to limiting vorticities which are sufficiently regular. In contrast, although they concern very dilute limits, the results of \cite{ks,js} allow for general (possibly irregular) limiting vorticities.

\subsection{Our setting and results}

\subsubsection{Scaling of the equation}
In order to obtain a nontrivial limiting evolution,  the appropriate scaling of the equation  to consider is  
\begin{equation}\label{eq}
\left\{\begin{array}{ll}
\N(\frac{\a}{\lep}+ i \b) \p_t u = \Delta u + \frac{1}{\ep^2}u (1-|u|^2)& \quad \text{in} \ \R^2\\ [2mm]
u(\cdot, 0)= u^0.\end{array}\right.
\end{equation} 
Here we have put both the 2D Gross-Pitaevskii and parabolic equations in the same framework.  To obtain Gross-Pitaevskii, one should set $\a=0$ and $\b=1$ and to obtain the  parabolic Ginzburg-Landau equation, one should set $\a=1 $ and $\b=0$. We will also comment later on the  mixed case where one would have $\a$ and $\b$ nonzero, and say, $\a^2+\b^2=1$, and we will describe the adaptation to the 3D Gross-Pitaevskii case as we go. In all the paper, when we write $\R^n$, we mean the whole Euclidean space with $n$ being either $2 $ or $3$.

\subsubsection{Limiting equation}

Throughout the paper, for vector fields $X$ in the plane,  we use the notation
$$X^\perp= (-X_2, X_1) \qquad \np=(-\p_2, \p_1)  $$
and this way $$\curl X= \p_1 X_2 -\p_2 X_1= - \div X^\perp.$$
In the 2D Gross-Pitaevskii case, or in the parabolic case  with $\N\ll \lep$, then the limiting equation  will be the incompressible equation
\begin{equation}\label{lime}
\boxed{\left\{
\begin{array}{ll}\p_t \v = 2\b \v^\perp \curl \v - 2\a \v \curl \v  +\nab \pp& \quad \text{in} \ \R^2\\ [2mm]
 \div \v =0 & \quad \text{in} \ \R^2,\end{array}\right.}\end{equation} with $\pp$ the pressure associated to the divergence-free condition.
 In the parabolic case with  $\N $  comparable to $ \lep$, then without loss of generality, we may assume that 
\begin{equation}\label{defl}
\lambda:= \lim_{\ep \to 0} \frac{\lep}{\N} \end{equation}exists and is positive and finite. 
In that case, the limiting equation will be
\begin{equation}
\label{limg}
\boxed{\p_t \v = \lambda \nab \div \v - 2 \v \curl \v    \quad \text{in} \  \R^2.}\end{equation}
This is a particular case of the mixed flow equation
\begin{equation}
\label{limgm}
\p_t \v = \frac{\lambda}{\a } \nab \div \v+  2\b \v^\perp \curl \v - 2\a \v \curl \v   \quad \text{in} \  \R^2.\end{equation}
The incompressible Euler equation is typically written as
\begin{equation}\label{ince}
\left\{
\begin{array}{ll}\p_t \v +\v \cdot \nab \v +\nab \pp=0
   & \quad \text{in} \ \R^n\\ [2mm]
 \div \v =0 & \quad \text{in} \ \R^n,\end{array}\right.\end{equation} where $\pp$ is the pressure.
 But when $\div \v =0$ and $n=2$, one has the identity 
 \begin{equation}\label{vdivv}\v\cdot \nab \v = \div (\v \otimes \v) = \v^\perp\curl \v + \hal \nab |\v|^2,\end{equation}
 so when $\a=0$, \eqref{lime}  is exactly  the 2D incompressible  Euler equation (up to a time rescaling by a factor of 2), with the new pressure equal to the old one plus $ |\v|^2$.  Existence, uniqueness, and regularity of its solutions  are  well-known since Volibner and  Yudovich (one can refer to textbooks such as \cite{bertom,chemin} and references therein). 
 
 In the three-dimensional Gross-Pitaevskii case, our limiting equation will be the time-rescaled incompressible Euler equation rewritten again as  \begin{equation}\label{incel}
\boxed{
\left\{
\begin{array}{ll}\p_t \v = 2 \div \( \v \otimes \v - \hal |\v|^2 I\) +\nab \pp& \quad \text{in} \ \R^3\\ [2mm]
 \div \v =0 & \quad \text{in} \ \R^3.\end{array}
 \right.
 }\end{equation}

Taking the curl of \eqref{lime}, one obtains 
 $$\p_t \omega= 2 \div (( \a \v^\perp +\b \v)  \omega ) \qquad \omega = \curl \v , \  \div \v =0, $$ 
 so in the case $\a=0$, $\b=1$, we of course retrieve  the vorticity form of the Euler equation, and in the case $\a =1$,  $\b=0$ and $\N \ll \lep$, we see that  we retrieve  the Chapman-Rubinstein-Schatzman-E equation  \eqref{crs}. 
 But in the case $0<\lambda <+\infty$ the curl of \eqref{limg}  is {\it not} \eqref{crs}.  In fact the divergence of $\v$ is not zero, even if we assume it vanishes  at initial time, and this affects the dynamics of the vortices. This phenomenon can be seen as a drift on the vortices created by the phase of $u$, a ``phase-vortex interaction", as  was also observed in \cite{bos1}.
 
 The existence and uniqueness of global regular  solutions   to \eqref{lime} can be obtained exactly as in the case of the Euler equation, see for example \cite{chemin}.  It suffices to write the equation in vorticity  form  as
$\p_t \omega= A(\omega) \cdot \nab \omega$ where $A(\omega)$ is a Fourier operator of order $-1$ (we do not give details here). In contrast, 
 the equation \eqref{limg} is new in the literature, and for the sake of completeness we present in an appendix (cf.~Theorem \ref{th-exi}) a result of local existence and uniqueness of $C^{1,\gamma}$ solutions to the general equation \eqref{limgm}, which suffices for our purposes. More results, including global  existence  of such regular solutions are   established in   \cite{du}.

 We note that the transition from one limiting equation to the other  happens  in the parabolic case, in the regime where $\N$ is proportional to $\lep$. This is a natural regime, which corresponds to the situation where the number of points is of the same order as the ``self-interaction" energy of each vortex.\subsubsection{Method of proof}
Our method of proof bypasses  issues such as taking  limits in the energy evolution or vorticity evolution relations and  controlling vortex distances. Instead, it takes advantage of the regularity and stability  of the solution to the limiting equation. More precisely, we introduce what can be called
a ``modulated Ginzburg-Landau energy"
\begin{equation}\label{defE}
\E(u,t)  = \hal \int_{\R^n} |\nab u -iu \N \v(t)|^2 + \frac{(1-|u|^2)^2}{2\ep^2} +\N^2 (1-|u|^2) \psi(t) ,\end{equation}
where  $\v(t)$ is a regular solution to the desired limiting equation, and 
\begin{equation}\label{defpsi}
\psi(t)= \begin{cases} \pp(t)-|\v(t)|^2 & \text{in the Gross-Pitaevskii case} \\
 -|\v(t)|^2 &  \text{in  the parabolic case} .\end{cases}
\end{equation}
This quantity is modelled on the Ginzburg-Landau energy, and controls the $L^2$ distance between the supercurrent $j_\ep = \la iu_\ep, \nab u_\ep\ra$  normalized by $\N$ and the  limiting  velocity field $\v$, because
in the regimes we consider, the term $\int_{\R^n}\N^2 (1-|u|^2)\psi(t)$ is a small perturbation (that term will however play a role in algebraic cancellations).  

One of the difficulties in the proof is that the convergence of $j_\ep/\N$ to $\v$ is not strong in $L^2$, in general, but rather weak in $L^2$, due to a concentration of  an amount  $\pi \lep$ of  energy at each of the vortex points (this energy concentration can be seen as  a defect measure in the convergence of $j_\ep/\N$ to $\v$).  In order to take that  concentration into account, we need to subtract off of $\E$ the constant quantity $\pi \N \lep$. In the regime where $\N \gg \lep$, then $\pi \N \lep=o(\N^2)$ and this quantity  (or the concentration) happens to  become negligible, which is what will  make the proof in the Gross-Pitaevskii case much simpler and applicable to the three-dimensional setting as well, but of course restricted to that regime. 
 
The main point of the proof consists in differentiating $\E$ in time, and showing that 
$\frac{d}{dt} \E(u_\ep(t),t) \le C (\E(u_\ep(t), t)-\pi \N \lep)$, which allows us to apply Gronwall's lemma, and conclude that if $\E-\pi \N \lep$ is initially small, it remains so.  The  difficulty for this is to show that a control  in $C(\E-\pi \N \lep)$ instead of $C \sqrt{\E}$, is possible, even though the terms involved initially appear to be of order $\sqrt{\E}$. This is made possible by a series of algebraic simplifications in the computations  that reveal only quadratic terms.  An important insight leading to these simplifications  is that one should work, by analogy with the gauged Ginzburg-Landau model, with gauged quantities, where $\N \v$ plays the role of a space gauge vector-field, and $\N \phi$ plays the role of a temporal gauge, $\phi$ being defined by 
  \begin{equation}\label{defphi}
\phi(t)= \begin{cases} \pp(t) & \text{in cases leading to } \ \eqref{lime}, \eqref{incel}\\
{\lambda} \div \v(t) &  \text{in cases leading to } \ \eqref{limg}.\end{cases}
\end{equation}

The idea of  proving convergence via a Gronwall argument on the  modulated energy, while  assuming and using the regularity of the limiting solution is similar to the relative entropy method for establishing  (the stability of) hydrodynamic limits,  first introduced in   \cite{yau} and used for quantum many body problems, mean-field theory 
and semiclassical limits, one example of the latter arising precisely  for the limit of the Gross-Pitaevskii equation in
 \cite{lz2}; or Brenier's modulated entropy method to establish   kinetic to fluid limits such as   the derivation of the Euler equation  from the Boltzmann or Vlasov  equations (see for instance \cite{laure} and references therein).

 In the point of view of that method, 
  $\int_{\R^n} |\nab u -iu \N \v|^2$ is the modulated energy, while  $\int_{\R^n} |\p_t u -iu \N \phi|^2$ is the modulated energy-dissipation.
\subsubsection{Well-posedness of the Cauchy problem}\label{wp}
The equations \eqref{gp} and \eqref{gp3} are shown in \cite{gerard,gallo} to be globally well-posed in the natural energy space  
\begin{equation}\label{nrjspace}
E=\{u \in H^1_{loc} (\R^n), \nab u \in L^2(\R^n), |u|^2-1\in L^2(\R^n)\}.\end{equation}
 This is the setting we will consider in dimension 3 and corresponds to solutions which have zero total degree at infinity. But in general this is too restrictive for our purposes: 
the problem is that, when working in the whole space,  the natural energy is infinite as soon as the total degree of $u$ at infinity is not zero. It thus needs to be renormalized by substracting off the energy of some fixed map $U_{D_\ep}$ which behaves at infinity like $u_\ep$, i.e.~typically like $e^{i\N\theta}$ for example.  To be more specific, in the two-dimensional case
we consider as in \cite{bs,miot}, for each integer $D$,  a reference map  $ U_D$, which is smooth in $\R^2$ and such that 
\begin{equation}\label{UN}
 U_D=\(\frac{z}{|z|}\)^D\quad \text{ outside of $B(0,1)$}.\end{equation}

The well-posedness of the Cauchy problem in that context was established in \cite{bs} in the Gross-Pitaevskii case and \cite{miot} in the mixed flow (hence possibly parabolic) case : they show  that given  $u_\ep^0\in U_{D_\ep} +H^1(\R^2)$ (in fact they even consider a slightly wider class than this), there exists a unique global solution to \eqref{eq} such that $u_\ep(t)-U_{D_\ep}\in C^0(\R, H^1(\R^2))$, and satisfying  other 
properties that will be listed in Section \ref{sec:refu}. This is the set-up that we will use, as is also done in 
 \cite{jsp2}. Without loss of generality, we may assume that $D_\ep \ge 0$. 
In the Gross-Pitaevskii case, we will allow for $D_\ep$ (the total degree) to be possibly different from $\N$ (the total number of vortices), which corresponds to a vorticity which does not have a distinguished sign. For simplicity, we will then assume that  $D_\ep/\N = d\le 1$, a number independent of $\ep$.
In the parabolic case, we need to have a distinguished sign and we will  take $D_\ep=\N$.

Let us point out  that the use of the modulated energy will simplify the proofs in that respect, in the sense that it naturally provides a finite energy and thus replaces having to ``renormalize" the infinite energy  as in \cite{bjs,miot,jsp2}.

\subsubsection{Main results}
We may now state our main results, starting with the  Gross-Pitaevskii case. In all the paper, we  denote by $C^\gamma$ the space of functions which are bounded and H\"older continuous with exponent $\gamma$, and by $C^{1, \gamma}$ the space of functions which are bounded and whose derivative is bounded and $C^\gamma$.   We use the standard notation $a\ll b$ to mean $\lim a/b=0$, and the standard $o$ notation, all asymptotics being taken as $\ep \to 0$. 

\begin{theo}[Gross-Pitaevskii case] \label{th1}Assume $\N$ satisfies 
\begin{equation}\label{condN1}
\lep \ll \N \ll \frac{1}{\ep}\quad \text{as} \ \ep \to 0.\end{equation} Let $\{u_\ep\}_{\ep>0}$ be solutions to 
\begin{equation}\label{gpr}
\left\{\begin{array}{ll}
i\N \p_t u_\ep = \Delta u_\ep + \frac{u_\ep}{\ep^2} (1-|u_\ep|^2)\quad \text{in} \ \R^n\\ [2mm]
u_\ep(\cdot, 0)=u_\ep^0.\end{array}\right.\end{equation}

If $n=2$ we assume  $u_\ep^0 \in U_{D_\ep}+ H^{2}(\R^2)$ where  we take $0\le D_\ep\le \N$ with $D_\ep/\N=d$, 
 $ U_{D_\ep}$ is as in \eqref{UN}, and  $\v$ is a solution to \eqref{lime}, such that 
 $\v(0) - d \la \nab U_1, i U_1\ra   \in L^2(\R^2)$,     $\v(t)\in L^\infty (\R_+ , C^{0,1}(\R^2))$, and $\curl \v(t)\in L^\infty(\R_+, L^1(\R^2))$.
\\
If $n=3$ we assume $u_\ep^0\in E$ as in \eqref{nrjspace} with $\Delta u_\ep^0\in L^2(\R^3)$, and $\v$ is a solution to \eqref{incel} such that  $ \v(t)\in L^\infty ([0,T] , C^{0,1}(\R^3)\cap L^2(\R^3))$, $\p_t \v \in L^\infty ([0,T] , L^\infty(\R^3)\cap L^2(\R^3))$ and  $\curl \v(t)\in L^\infty([0,T], L^1(\R^3))$.
 
  Letting $\E$ be defined from $\v(t)$ via \eqref{defE},   assume that  
  \begin{equation}\label{eu0}
  \E(u_\ep^0,0)\le o(\N^2).\end{equation} Then,  for every $t \ge 0$ (resp. $t\le T$), we have
  $\E(u_\ep(t), t) \le o(\N^2)$, and  in particular  we have 
\begin{equation}\frac{j_\ep}{\N}:= \frac{\la \nab u_\ep, i u_\ep\ra}{\N} \to  \v \quad  \text{strongly in} \  L^1_{{loc}}(\R^n). \end{equation} 
If we know in addition that $u_\ep $ is bounded in $L^\infty_{{loc}}(\R_+, L^\infty(\R^n))$ uniformly in $\ep$, then the convergence is strong in $L^2(\R^n)$.
\end{theo}
The restriction $\N \gg \lep$ is a technical obstruction caused by  the difficulty in controlling the velocity of the individual vortices because of the lack of control of $\int_{\R^n} |\p_t u_\ep|^2$. On the other hand, the restriction $\N \ll \frac{1}{\ep}$ seems quite natural, since when $\N$ is larger, the modulus of $u$ should enter the limiting equation, giving rise to compressible Euler equations.  On that aspect we refer to the survey \cite{bdgs} and results quoted therein. 
\smallskip 

In the parabolic  case, we have the following result 
\begin{theo}[Parabolic  case]\label{th2}
Assume $\N$ satisfies 
\begin{equation}\label{condN2}
1\ll  \N \le O( \lep) \quad \text{as} \ \ep \to 0,\end{equation} and let $\lambda$ be as in \eqref{defl}. \\
Let $\{u_\ep\}_{\ep>0}$ be solutions to \eqref{eq} with $\b=0$ and $\a=1$, associated to initial conditions $u_\ep^0 \in U_{\N}+ H^1(\R^2)$ where $ U_{\N}$ is as in \eqref{UN}. Assume $\v$ is a solution to \eqref{lime}  if $\N \ll \lep$, and to \eqref{limg} otherwise,  such that 
 $\v(0) - \la \nab U_1, i  U_1\ra\in L^2(\R^2)$, $\curl \v(0)\ge 0$,  and belonging to   $L^\infty ([0,\t] , C^{1, \gamma}(\R^2))$ for some $\gamma>0$ and some $\t >0$ (possibly infinite).   
  Letting $\E$ be defined from $\v(t)$ via \eqref{defE},   assume that  
  \begin{equation}\label{eu02}
  \E(u_\ep^0,0)\le \pi \N \lep+ o(\N^2).\end{equation} Then,  for every $t \in [0,\t]$,  we have $\E(u_\ep(t), t) \le \pi \N \lep + o(\N^2)$ and
\begin{multline}\label{123b}\text{ 
 $\D \frac{\nab u_\ep-iu_\ep\N \v}{\N}  \to 0 $   strongly in $L^p_{loc}(\R^2)$  for $p<2$,} \\ \text{ and weakly in  $L^2(\R^2) $ 
  if in addition $ \lambda<\infty,$}\end{multline}
   \begin{equation}\label{123c}
   \text{$\D \frac{j_\ep}{\N} : = \frac{\la \nab u_\ep, iu_\ep\ra}{\N} \to \v $ strongly in $  L^1_{loc}(\R^2).$}\end{equation}
  If we know in addition that $u_\ep $ is bounded in $L^\infty_{{loc}}(\R_+, L^\infty(\R^2))$ uniformly in $\ep$, then the convergence of $j_\ep/ \N$ is in the same sense as in \eqref{123b}.\end{theo}
We note that in both theorems, we obtain the convergence of the solutions of \eqref{eq} at the level of their supercurrents $j_\ep$, which is obviously stronger than  the convergence of the vorticity $\mu_\ep/\N = \curl j_\ep /\N$ to $\curl \v$, which it implies. 

The additional condition on the uniform boundedness of $u_\ep$ is easy to verify in the parabolic case: 
for example if the initial data satisfies  $|u_\ep^0|\le 1$, then this is preserved along the flow of  \eqref{pgl} by the maximum principle.

The conditions  \eqref{eu0}   and \eqref{eu02} are well-preparedness conditions. It is fairly standard to check that one can build configurations $u_\ep^0$ that satisfy them, for example proceeding as in \cite[Section 7.3]{livre}.  

In the parabolic case, for $u_\ep^0\in U_{\N}+H^1(\R^2)$ and \eqref{eu02} to hold,  the initial configuration should have most of its vortices of degree $1$, and  thus $\curl \v(0)$ must be nonnegative (it is automatically of mass $2\pi$ by the condition $\v(0) - \la \nab U_1, i  U_1\ra\in L^2(\R^2)$ so the assumption $\curl \v(0)\ge 0$ is redundant). 
 We take advantage of these well-preparedness conditions as well as of the regularity of the solutions to the limiting equations in crucial ways. Since regularity propagates in time in these limiting equations, then the regularity assumptions  we have placed really amount to just another assumption on the initial data. 
  It is of course  significantly more challenging and an open problem to prove convergence without such assumptions, in particular without knowing in the parabolic case that the initial limiting vorticity has a sign. 

The reason for the restriction  $\N \le O(\lep)$  will become clear in the course of the proof:  the factor of growth of the modulated energy in Gronwall's lemma is bounded by $C\N/\lep$ and thus becomes too large  otherwise.  We are not even sure whether the formal analogue of \eqref{limg}, i.e. the equation with $\lambda=0$ (shown to be locally well-posed in \cite{du}),  is  the correct limiting equation.


\subsection{Other settings}
\subsubsection{The mixed flow case}
With our method of proof, we can prove that if $\a>0$, $\b>0$, and $\a^2+\b^2=1$, the same results as Theorem \ref{th2} hold, with convergence to the limiting equation \eqref{lime}, respectively \eqref{limgm}, under the additional condition $\N \gg \llep$, cf. Remark \ref{remark410}. For a   sketch  of the proof and quantities to use, one can refer to Appendix C, setting the gauge fields to~$0$. 

\subsubsection{The torus}

We have chosen to work in the whole plane or space, but one can easily check that the proof works with no change in the case of a torus. The proof  is of course  easier since there is no need for controlling infinity, and there are no boundary terms in integrations by parts.  However, this gives rise to a   nontrivial situation  only in the Gross-Pitaevskii case, since in the parabolic case we have to assume that the vorticity has a  distinguished sign, while at the same time its integral over the torus vanishes. In the parabolic case, to have a  nontrivial situation one should instead  consider  the case with gauge as described just below  in Section \ref{sec:gauge}, where the total vorticity over a period does not have to vanish.

\subsubsection{Bounded domains}
On the contrary, working in a bounded domain entails significant difficulty in the parabolic  case : one basically needs to control the  change in energy due to the entrance and exit of vortices (one can see the occurence of this difficulty at the limiting equation level in \cite{as}), and the necessary tools are  not yet available. 
  
\subsubsection{The case with gauge}\label{sec:gauge}
Our proof adapts well to the cause with gauge, which is the true physical model for superconductors, again in the setting of the infinite plane or the torus.  The  corresponding evolution equations are then the so-called Gorkov-Eliashberg equations 
  \begin{equation}\label{glg}
  \left\{\begin{array}{ll}
 & \N  \(\frac{\a}{\lep}+ i\b\) (\p_t u-iu\Phi)= (\nab_A)^2 u+\frac{u}{\ep^2}(1-|u|^2)  \quad \text{in} \ \R^2\\ [2mm]
&  \sigma (\p_t A-\nab \Phi)= \np \curl A + \la iu,\nab_A u\ra \quad \text{in} \ \R^2.
\end{array}\right.
\end{equation}
These were first derived by Schmid \cite{schmid} and Gorkov-Eliashberg \cite{ge}, and the mixed flow equation was also suggested as a good model for the classical Hall effect  by Dorsey \cite{dorsey} and Kopnin  et al. \cite{kik}.
In this system the unknown functions are $u$, the complex-valued order parameter, $A$ the gauge of the magnetic field (a real-valued vector field), and $\Phi$ the gauge of the electric field (a real-valued function). The notation $\nab_A$ denotes the covariant derivative $\nab -i A$.  The magnetic field in the sample is deduced from $A$ by $h=\curl A$, and the electric field by $ 
E= -\p_t A+\nab \Phi$. Finally, $\sigma\ge 0$ is a real parameter, the characteristic relaxation time of the magnetic field, which may be taken to be $0$. 
The dynamics of a finite number of vortices under such flows was established rigorously in \cite{sp1,sp2,ss,ks0,tice,stice}, in a manner analogous to that described in the case without gauge. A dynamical law for the limit of the vorticity was formally proposed in \cite{crs,E2}, the analogue of \eqref{crs} mentioned above. 

Natural physical quantities associated to this model are the gauge-invariant supercurrent 
  $$j_\ep = \la iu_\ep, \nab_{A_\ep} u_\ep\ra, $$
  the gauge-invariant vorticity 
  $$\mu_\ep= \curl (j_\ep+ A_\ep)$$
  and the electric field 
   $$E_\ep= - \p_t A_\ep +\nab \Phi_\ep.$$

In Appendix \ref{gauge} we explain how to adapt the computations made in the planar case without gauge to the case with gauge, in order to derive the following limiting equations: 
 if $\a=0$ or $\N \ll \lep$, 
\begin{equation}\label{le1}
\left\{\begin{array}{ll} & 
 \p_t \v - \mathrm{E}= (-2\a\v+2\b\v^\perp) (\curl \v+ \h) + \nab \pp  \\ [2mm]
& \div \v =0\\ [2mm]
& - \sigma \mathrm{E}= \v + \np \h \\ [2mm]
&\p_t \h = - \curl \mathrm{E},\end{array}\right.\end{equation}
 and if $\a\neq 0$ and $\lim_{\ep \to 0} \frac{\lep}{\N}=\lambda$ is positive and finite, 
\begin{equation}\label{le2}
\left\{\begin{array}{ll} &  \p_t \v - \mathrm{E}= (-2\a\v+2\b\v^\perp) (\curl \v+ \h) +\frac{\lambda}{\a} \nab \div \v \\ [2mm]
& - \sigma \mathrm{E}= \v + \np \h \\ [2mm]
&\p_t \h = - \curl \mathrm{E}.\end{array}\right.
\end{equation}
The corresponding results to Theorems \ref{th1} and \ref{th2} are then the convergences
\begin{equation}\label{cvg}
 \frac{j_\ep}{\N} \to \v, \quad \frac{\mu_\ep}{\N}\to \mu:=\curl \v + \h,\quad  \frac{\curl A_\ep}{\N} \to \h,  \quad \frac{E_\ep}{\N} \to \mathrm{E}\end{equation}
in the case $\a=0$ and $ \lep \ll \N \ll \frac{1}{\ep}$ to \eqref{le1}, and in the case $\b=0$ and $1\ll \N \le O(\lep)$ to \eqref{le1} or \eqref{le2} according to the situation. 

We note that if $\sigma \neq 0$,  \eqref{le1}, resp. \eqref{le2}, can be rewritten as a system of equations on only two unknowns $\v$ and $\h$~:
\begin{equation}\label{le11}
\left\{\begin{array}{ll} &  \p_t \v + \frac{1}{\sigma} (\v + \np \h) = (-2 \a\v+2\b \v^\perp)  (\curl \v+ \h) +\nab p \\ [2mm]
& \div \v =0\\ [2mm]
& \sigma \p_t \h =\curl \v + \Delta \h,  \end{array}\right.
\end{equation}
respectively
\begin{equation}\label{le11b}
\left\{\begin{array}{ll} &  \p_t \v + \frac{1}{\sigma} (\v + \np \h) = (-2 \a\v+2\b \v^\perp)  (\curl \v+ \h) + \frac{\lambda}{\a} \nab \div \v \\ [2mm]
&
 \sigma \p_t \h =  \curl \v + \Delta \h.  \end{array}\right.
\end{equation}
We can also note  that taking the curl of \eqref{le1} or \eqref{le2}  gives  (with $\mu$ the limiting vorticity  as in \eqref{cvg}) 
\begin{equation*}
\left\{\begin{array}{ll} &  \p_t \mu   = \div ( (2 \a\v^\perp +2\b \v )  \mu)   \\ [2mm]
&
\sigma \p_t \h - \Delta \h + \h = \mu  \\ [2mm]
 & \div \v =0 \ \text{or} \ \p_t \div \v = -\frac{1}{\sigma} \div \v+ \div ( -2\a \v + 2\b \v^\perp)  \mu) - \frac{\lambda}{\a}\Delta \div \v. \end{array}\right.
\end{equation*}
This is a transport equation for $\mu$, coupled with a linear heat equation for $\h$. As before, when choosing $\beta=0$, this equation  coincides with the model of \cite{crs,E2} when  $\N \ll \lep$, but not if $\N $ is larger and $\lambda $ is finite. 
\\

The rest of the paper is organized as follows: We start with some preliminaries in which we recall some properties of the solutions and a priori bounds, introduce the  basic quantities like the stress-energy tensor, the velocity and the modulated energy, and present explicit computations on them. 

Then we present the proofs   in increasing order of complexity: we start with the   proof of Theorem \ref{th1}  which is remarkably short.   We then present the proof of Theorem \ref{th2}, which requires using all the by now standard tools of vortex analysis techniques : vortex-balls constructions, Jacobian estimates, and product estimate.  
An appendix is devoted to the proof of the short-time existence and uniqueness  of solutions to \eqref{limg}, and another one to the computations in the gauge case. 
\\

{\it Acknowledgments: } I would like to warmly thank Jean-Yves Chemin for  his guidance with the proof of Theorem \ref{th-exi}, and to thank Didier Smets for pointing out that the proof for the Gross-Pitaevskii case should work in dimension 3. I also thank Matthias Kurzke, Mitia Duerinckx, and the anonymous referees for their careful reading and very useful comments.

\section{Preliminaries} \label{sec2}
In these preliminaries, we will work in the general setting of \eqref{eq} (or its three-dimensional version) with $\a$, $\b$, nonnegative satisfying $\a^2+\b^2=1$, which allows  us to treat the Gross-Pitaevskii and parabolic (and mixed) cases at once.  We note that with the choice \eqref{defphi}, the limiting  equations \eqref{lime}, \eqref{limg} can always be rewritten 
\begin{equation}\label{eqglob}
\p_t \v=\nab \phi +  2\b \v^\perp \curl \v - 2 \a \v \curl \v.\end{equation}
Here and in all the paper, $C$ will denote some positive constant independent of $\ep$, but which may depend on the various bounds on $\v$.

Also $C^{-1+\gamma}(\R^n)$ denotes functions that are  a sum of derivatives of $C^{\gamma}(\R^n)$ functions and $C^\infty(\R^n)$ bounded functions, and $C^{1+\gamma}(\R^n)$ is the same as $C^{1,\gamma}(\R^n)$, i.e.~functions which are bounded and whose derivative is bounded and $C^\gamma$.

\subsection{A priori bounds}\label{sec:refu}
\subsubsection{A priori estimates on $\v$}
We first gather a few   facts about the solutions $\v$ to \eqref{lime}, \eqref{limg} and \eqref{incel} that we consider.
\begin{lem}\label{lemv}
Let  $\v$ be a  solution to \eqref{lime} in $ L^\infty([0,\infty], C^{0,1}(\R^2))$ such that $\v(0)-d\la \nab U_1, iU_1\ra \in L^2(\R^2)$ and $\curl \v\in L^\infty(\R_+,L^1(\R^2))$. There exists a $\pp$ such that \eqref{lime} holds and such that  for any $0<\t<\infty$,
$\v- d\la \nab U_1, iU_1\ra \in L^\infty([0,\t], L^2(\R^2))$, 
 $\v \in L^\infty([0,\t], L^4(\R^2))$, $\p_t  \v\in L^\infty([0,\t], L^2(\R^2)\cap L^\infty(\R^2))$, $\pp\in L^\infty([0,\t], C^{0,1}(\R^2)  \cap L^2(\R^2)) $, and $ \p_t \pp, \nab \pp \in L^\infty([0,\t], L^2(\R^2))$. 
 
 Let $\v$ be a solution to \eqref{incel} such that $ \v\in L^\infty ([0,T] , C^{0,1}(\R^3)\cap L^2(\R^3))$, $\p_t \v \in L^\infty ([0,T] , L^\infty(\R^3)\cap L^2(\R^3))$ and  $\curl \v(t)\in L^\infty([0,T], L^1(\R^3))$.
There exists a $\pp$ such that \eqref{incel} holds and such that  for any $0<\t<\infty$,  $\pp\in L^\infty([0,\t], C^{0,1}(\R^3)   \cap L^2(\R^3)) $ and $ \p_t \pp, \nab \pp \in L^\infty([0,\t], L^2(\R^3))$. 
 
Let $\v $ be a solution to \eqref{limg} in $L^\infty([0,\t],  C^{1,\gamma}(\R^2))$, $0<\t<\infty$, such that $\v(0)-\la \nab U_1, iU_1\ra\in L^2(\R^2)$. 
 We have $\v- \la \nab U_1, iU_1\ra \in L^\infty([0,\t], L^2(\R^2))$, 
 $\v \in L^\infty([0,\t], L^4(\R^2))$, $\div \v\in L^\infty([0,\t],  L^2(\R^2)\cap L^\infty(\R^2))$ and  $\p_t \v, \nab \div \v \in L^2([0,\t], L^2(\R^2)\cap  L^\infty(\R^2))$. Also for any $t \in [0, \t]$, $\frac{1}{2\pi}\curl \v(t)$ is a probability measure. 
 \end{lem}
 \begin{proof}
Let us start with the case of \eqref{lime}. We first observe that such a solution exists, and that it corresponds to the solution belonging to the space $E_{2\pi d}$ in the notation of \cite[Definition 1.3.3]{chemin}. It is also known (cf. \cite{chemin}) that the solution remains such that $\v(t)- d\frac{x^\perp}{|x|^2}  \in L^\infty([0,\t], L^2(\R^2\backslash B(0,1)))$ and thus $\v(t)-\la \nab U_1, iU_1\ra \in L^\infty([0,\t], L^2(\R^2)$ (the case of general $\a $ and $\b$ works similarly). 

Since   $\la \nab U_1,i U_1\ra $ decays like $1/|x|$, by boundedness of $\v$ and the $L^2$ character of $\v- d \la \nab U_1, iU_1\ra$, we deduce that $\v\in L^\infty(\R_+,  L^4(\R^2))$. 

The integrability of $\pp$ is deduced from that of $|\v|^2$ by using the formula 
\begin{equation}\label{forme}
\Delta \pp = 2\b \( \div \div (\v\otimes \v) -\hal \Delta |\v|^2 \) +2\a  \div ( \div (\v\otimes \v))^\perp \end{equation}which means that $\Delta \pp$ is a second order derivative of $\v\otimes \v$. 
Since $\v\in  L^4(\R^2)\cap C^{0,1}(\R^2)$, this allows us to pick a pressure $\pp$ such that $\pp\in L^\infty\cap L^2$ (cf.~\cite[p.13]{chemin}) and it is also $C^{0,1}$ by the assumed regularity of $\v$.  
We also note that $\v \curl \v\in L^2(\R^2)\cap L^\infty(\R^2)$ uniformly in time by boundedness of $\v, \curl \v$, and the fact that $\curl \v$ is  integrable uniformly in time.  On the other hand, we may write 
$$ \p_t \v= \nab^\perp \Delta^{-1}\p_t \curl \v= \nab^\perp \Delta^{-1}\div (-2\b \v\curl \v + 2\a \v^\perp \curl \v)$$ where $\Delta^{-1}$ is the convolution with $-\frac{1}{2\pi} \log$,
and with the above remark, we may deduce that $\p_t \v$ remains in $L^2(\R^2)\cap L^\infty(\R^2)$ uniformly in time, and the result for $\nab \pp$ follows by using the equation. The same result follows for $\p_t \pp$ by applying $\p_t \Delta^{-1}$ to   \eqref{forme}. 

The same arguments  apply to prove the results stated for  \eqref{incel}.

Let us now turn to \eqref{limg}. It is proved in 
Theorem \ref{th-exi} that $\v(t)-\v(0)\in L^\infty([0,\t],L^2(\R^2))$,  $\p_t \v\in L^2([0,\t],L^2(\R^2))$,  and $\div \v\in L^\infty([0,\t],L^2(\R^2)).$ It also  follows immediately that $\v(t)- \la \nab U_1, iU_1\ra \in L^\infty([0,\t], L^2(\R^2))$ and the uniform $L^4 $ character of $\v$ follows just as in the Euler case above.
 The fact that  $\nab \div \v\in L^2([0,\t], L^2(\R^2))$ follows in view of  \eqref{limg}.  The fact that $\frac{1}{2\pi}\curl \v$ remains a probability measure is standard by integrating the equation. 

Next, differentiating \eqref{limg} we find that 
\begin{equation}\label{heatdu}\p_t (\nab \div \v) = \lambda \Delta (\nab \div \v) -2 \nab \div (\v \curl \v) .\end{equation}
It can be found in \cite[Proposition 2.1]{chemin2} that if $u$ solves on $\R^2$ 
the equation $\p_t u- \Delta u= f$ on $[0,T]$ with  $f\in L^\infty([0,T], C^{-2,\gamma}(\R^2))$ and initial data $u_0\in C^{0, \gamma}(\R^2)$ then 
\begin{equation}\label{chemin}
\|u\|_{L^\infty([0,T], C^\gamma(\R^2))} \le C_T \( \|u_0\|_{C^{\gamma}(\R^2)}+ \|f\|_{L^\infty([0,T], C^{-2+\gamma}(\R^2))}\)
\end{equation}
(it suffices to apply the result there with $\rho=p=\infty$ and $s=-2+\gamma$ and notice that the Besov space $B_\infty^s$ is the same as $C^s$ or  $C^{0,s}$).
Applying this to \eqref{heatdu},  since the right-hand side is $L^\infty([0,\t], C^{-2+\gamma}(\R^2))$ we obtain that $\nab \div \v \in L^\infty([0,\t],C^\gamma(\R^2))$ hence $L^\infty([0,\t], L^\infty(\R^2))$. Inserting into \eqref{limg} and using that $\curl \v \in L^1\cap L^\infty$  yields that $\p_t \v \in L^\infty([0,\t], L^\infty(\R^2))$.

\end{proof}
\subsubsection{Estimates on the solutions to \eqref{eq}}




\begin{lem}\label{apriori}
Assume $u_\ep$  and $\v(t)$ satisfy the assumptions of Theorem \ref{th1} or \ref{th2}.  Let $U_{D_\ep} =1$ in the three-dimensional case of Theorem \ref{th1} and $D_\ep=\N$ in the case of Theorem \ref{th2}. 
 Then for any $\t>0$ we have $\p_t u_\ep \in L^1 ([0,\t], L^2(\R^n))$ and for any $t \in [0, \t] $,
$$ \nab (u_\ep(t)-  U_{D_\ep}),  1-|u_\ep(t)|^2,  \nab u_\ep(t)-iu_\ep(t) \N \v(t)$$ all belong to  $L^2(\R^n)$, $\E(u_\ep(t),t)$ is  finite, and 
$$j_\ep(t)-\N \v(t)\in (L^1+L^2)(\R^n).$$
\end{lem}
\begin{proof}
First let us justify that $\p_t u_\ep(t) \in L^1([0,\t],L^2(\R^n))$.
In the two-dimensional Gross-Pitaevskii case, since we assume $u_\ep^0 \in U_{D_\ep} + H^2(\R^2)$, then studying  the equation for $w_\ep:= u_\ep- U_{D_\ep}$, we find in \cite[Prop. 3, Lemma 3]{bs} that $w_\ep$   remains in $H^2(\R^2)$ (by semi-group theory) and thus $\p_t u_\ep=\p_t w_\ep \in L^\infty_{loc}(\R, L^2(\R^2))$ by the equation.

For the two-dimensional parabolic (or mixed flow case), we assume $u_\ep^0 \in U_{D_\ep}+ H^1(\R^2)$ and (by decay of the energy for $w_\ep:= u_\ep- U_{D_\ep}$)  it shown in \cite[Theorem 1]{miot} that $\p_t u_\ep \in L^1([0,\t],L^2(\R^2))$.

In the case of the three-dimensional solutions to \eqref{gp3}, it is shown in \cite[Sec. 3.3]{gerard} that when  $\Delta u_\ep^0  \in L^2 (\R^3)$ then the solution belongs for all time to $X^2:=\{u\in L^\infty(\R^3), D^2u\in L^2(\R^3)\} $ and thus in view of the equation $\p_t u_\ep \in L^\infty_{loc}(\R, L^2(\R^3))$.

Let us turn to the two-dimensional cases. 
Following exactly \cite[Proposition 3]{bs} or \cite[Lemma A.6]{miot}, we get that 
$$ \hal \int_{\R^2} |\nab (u_\ep(t)- U_{D_\ep})|^2 +|u_\ep(t)-U_{D_\ep} |^2+(1-|u_\ep(t)|^2)^2 \le C(\ep,t)$$
where $C(\ep, t)$ is finite and  depends on  $\ep$, $t$ and $U_{D_\ep}$. We thus obtain   the $L^2$ character of the first three items. 

By Lemma \ref{lemv} and  $D_\ep=d \N$ (resp. $D_\ep= \N$), we  have that $\N\v(t) -\la  \nab  U_{D_\ep}, iU_{D_\ep} \ra \in L^2(\R^2)$. 
We may  then write $$|\nab u_\ep-iu_\ep\N \v|\le |\nab (u_\ep - U_{D_\ep})|+ \N|\v| |u_\ep-U_{D_\ep}|+ |\nab  U_{D_\ep} -i U_{D_\ep} \N\v|.$$  The first two  expressions in the right-hand side are in $L^2(\R^2)$ by what precedes and the boundedness of $\v$. For the third quantity, we have that $|\nab  U_{D_\ep}- iU_{D_\ep} \N \v| = |\la \nab  U_{D_\ep},iU_{D_\ep} \ra -  \N \v | $ outside of $B(0,1)$ and is bounded in $B(0,1)$, by definition of $U_D$, hence is in $L^2$ by Lemma \ref{lemv}.  We conclude by Lemma~\ref{lemv} that this term is also in $L^2(\R^2)$. 
The finiteness of $\E(u_\ep(t),t)$ is then an immediate consequence of what precedes, the fact that $\v\in L^4$ from Lemma~\ref{lemv}, and the Cauchy-Schwarz inequality.

Lastly, for  $j_\ep -\N\v$ we write $$|j_\ep -\N \v|\le  |\la \nab U_{D_\ep}, i U_{D_\ep}\ra - \N \v|+  |\nab ( U_{D_\ep} - u_\ep)| |u_\ep| + |\nab  U_{D_\ep}| |u_\ep- U_{D_\ep}| $$  and we conclude by the previous observations, writing $|u_\ep|=1+ (|u_\ep|-1)$ and using that $ (1-|u_\ep|)^2 \le (1-|u_\ep|^2)^2$, that $j_\ep-\N \v \in L^1+L^2(\R^n)$.

\end{proof}

\subsubsection{Coerciveness of the modulated energy}
We check that the modulated energy $\E$ does control the quantities we are interested in. 
We have 
\begin{lem} \label{contrmod}
The functional $\E$ being as in \eqref{defE} and $\psi$ as in \eqref{defpsi},  we have for any $u_\ep$, $R \ge 1$, $1<p<\infty$,
\begin{equation}\label{contmod}\hal \int_{\R^n} |\nab u_\ep-iu_\ep\N \v|^2 + \frac{(1-|u_\ep|^2)^2}{4\ep^2} \le \E(u_\ep) +\ep^2 \N^4 \|\psi\|_{L^2}^2 ,\end{equation}
\begin{multline}\label{contjl1}
\int_{B_R}  |j_\ep -\N \v|\le  C_{R,p}  \|\nab u_\ep -iu_\ep \N \v \|_{L^p(B_R)}  \\ +C_{R,p} (  \ep^2 \N^2  \|\psi\|_{L^2}   + \N \ep \|\v\|_{L^\infty} )  \E(u_\ep)^{\hal} + C\ep \E(u_\ep)+C_R \ep^2 \N^3\|\v\|_{L^\infty}\|\psi\|_{L^2}, 
\end{multline}
and
\begin{equation}\label{contj}
 \int_{\R^n} |j_\ep - \N \v|^2 \le C ( \|u_\ep\|_{L^\infty}^2 +  \ep^2 \N^2 \|\v\|_{L^\infty}^2 ) (\E(u_\ep) + \ep^2 \N^4 \|\psi \|_{L^2}^2) ,\end{equation}where $C$ is universal, $C_R$ depends only on $R$ and $C_{R,p}$ on $R$ and $p$.
\end{lem}
In view of assumptions \eqref{condN1} and \eqref{condN2} and Lemma \ref{lemv}, the second term on the right-hand sides of \eqref{contmod} and \eqref{contj} will always be bounded by $o(\N^2)$.  Also, if $\|u_\ep\|_{L^\infty}\le C$, the upper bound in \eqref{contj} is by $C \E(u_\ep) + o(1) $ and that in \eqref{contjl1} is by $ C_{R,p} \E(u_\ep)^{\hal} + C \ep \E(u_\ep)$.
\begin{proof} We observe that 
\begin{equation*}\frac{(1-|u_\ep|^2)^2}{2\ep^2}+  \N^2 \psi  (1-|u_\ep|^2 )
= \frac{(1-|u_\ep|^2 +\ep^2 \N^2 \psi)^2}{2\ep^2} -\ep^2\N^4  \psi^2\ge -\ep^2\N^4  \psi^2.\end{equation*}
Thus, using that $\psi \in L^2$ by Lemma \ref{lemv}, we have
$$\hal \int_{\R^n}  |\nab u_\ep-iu_\ep\N \v(t)|^2+ \frac{(1-|u_\ep|^2)^2}{4\ep^2} \le \E(u_\ep)  +  \hal\int_{\R^n}\ep^2 \N^4 \psi^2 .$$ 
For \eqref{contj}, we write that 
\begin{multline}\label{pjj}
|j_\ep - \N \v|\le |j_\ep - |u_\ep|^2\N \v|+\N |1-|u_\ep|^2|  |\v|=  |\langle i u_\ep, \nab u_\ep -iu_\ep \N \v\rangle|+\N |1-|u_\ep|^2|  |\v|
\\ 
\le |u_\ep| |\nab u_\ep - iu_\ep \N \v|+ \N |1-|u_\ep|^2|  |\v|\\
= |\nab u_\ep - iu_\ep \N \v|+ (|u_\ep|-1) |\nab u_\ep -iu_\ep \N\v|+ \N |1-|u_\ep|^2|  |\v|.\end{multline}
For \eqref{contjl1} we integrate this relation over $B_R$ and use H\"older's  inequality to get that for any $1<p<\infty$,
\begin{multline*}
\int_{B_R} |j_\ep -\N\v|\le C_{R,p} \( \int_{B_R} |\nab u_\ep-iu_\ep \N \v|^p\)^{\frac{1}{p}}
\\ +\( \int_{B_R} (1-|u_\ep|)^2\)^{\hal} \(\int_{B_R}  |\nab u_\ep -iu_\ep \N \v|^2 \)^\hal + C_R \|\v\|_{L^\infty}  \N \(\int_{B_R} (1-|u_\ep|^2)^2\)^{\hal}\end{multline*} and using that $(1-|u|)^2 \le (1-|u|^2)^2$ we are led to 
\begin{multline*}\int_{B_R} |j_\ep -\N\v|
\le C_{R,p}\|\nab u_\ep -iu_\ep \N \v \|_{L^p(B_R)}  \\+ C\ep (\E(u_\ep)+\N)\( \E(u_\ep)    +\ep^2\N^4\|\psi\|_{L^2}^2\)^\hal
\end{multline*}hence \eqref{contjl1} follows.
The proof of \eqref{contj} is a straightforward consequence of \eqref{pjj}. 
\end{proof}

\subsection{Identities} In this section, we  present important standard and less standard identities that will be used throughout the paper. In all that follows $\v$ is a vector field, which   implicitly depends on time and  solves one of our limiting equations.
\subsubsection{Current and velocity}
We recall that for a family $\{u_\ep\}_\ep$,  the supercurrent and vorticity (or Jacobian)  are defined as 
$$j_\ep = \la \nab u_\ep,iu_\ep\ra\qquad
\mu_\ep= \curl j_\ep.$$
 Following \cite{ss6}, we also define the  velocity 
\begin{equation}\label{defV}
V_\ep:= 2 \la i \p_t u_\ep,  \nab u_\ep\ra
\end{equation} and we have the identity
\begin{equation}\label{dtj}
\p_t j_\ep= \nab \la iu_\ep, \p_t u_\ep \ra +V_\ep.\end{equation}
Taking the curl of this relation yields that $\p_t \mu_\ep=\curl V_\ep$. (In dimension 2, this  means that   the vorticity is transported by $V_\ep^\perp$, hence the name velocity).
We also define  the modulated vorticity 
\begin{equation}\label{tildemu}
\tilde \mu_\ep := \curl \(\la  \nab u_\ep -iu_\ep \N \v,iu_\ep\ra + \N \v\),
\end{equation} and the modulated velocity 
\begin{equation}\label{tV}
\tV:= 2 \la i(\p_t u_\ep -iu_\ep \N \phi) , \nab u_\ep -iu_\ep\N \v \ra = V_\ep -\N \v \p_t |u_\ep|^2  +\N \phi \nab |u_\ep|^2 ,\end{equation}
with $\phi$ as in \eqref{defphi}.

We will  use the fact that for $u_\ep$ solution of \eqref{eq} (resp. \eqref{gpr}) we have the relation 
\begin{equation}\label{divji}
\div j_\ep= \N \la (\frac{\a}{\lep}+i\b)\p_t u_\ep, iu_\ep\ra,
\end{equation}
(resp. with $\alpha=0$ and $\beta =1$) 
which is obtained by taking the inner product of \eqref{eq} or \eqref{gpr} with $iu_\ep$.

\subsubsection{Stress-energy tensor}
We next introduce the stress-energy tensor associated to a function  $u$: it is the $n\times n$ tensor defined by 
\begin{equation}\label{defT}
(S_\ep(u))_{kl}:= \la \p_k u, \p_l u \ra - \hal \( |\nab u|^2 + \frac{1}{2\ep^2} (1-|u|^2)^2\)\delta_{kl}.\end{equation}
A direct computation shows that if $u$ is sufficiently regular, we have
$$\div S_\ep(u) :=\sum_{l=1}^n \p_l (S_\ep(u))_{kl}= \la \nab u, \Delta u + \frac{1}{\ep^2} u (1-|u|^2)\ra$$ so
 if $u_\ep$ solves \eqref{eq} or \eqref{gpr}, we have 
\begin{equation}\label{divT}
\div S_\ep(u_\ep)=\N \la   (\frac{\a}{\lep} + i \b) \p_t u_\ep, \nab u_\ep \rangle .\end{equation} 
We next introduce the modulated stress-energy tensor~: 
\begin{multline}
(\T (u))_{kl}:= \la \p_k u -iu \N \v_k, \p_l u -iu \N \v_l\ra  + \N^2 (1-|u|^2 ) \v_k \v_l 
\\- \hal \( |\nab u-iu\N\v |^2 + (1-|u|^2) \N^2|\v|^2+ \frac{1}{2\ep^2} (1-|u|^2)^2\)\delta_{kl},\end{multline}where $\delta_{kl}$ is $1$ if $k=l$ and $0$ otherwise.
One can observe that 
\begin{equation}\label{borneS}
|\T(u)|\le |\nab u-iu\N \v|^2+\frac{1}{4\ep^2} (1-|u|^2)^2+\N^2|1-|u|^2| |\v|^2\end{equation}and thus with Lemma \ref{contrmod} and the Cauchy-Schwarz inequality, we may write
 \begin{equation}\label{bornints}
\int_{\R^n} |\T(u)|\le 2 \E(u) +\N^2 \|1-|u|^2\|_{L^2}\|\v\|_{L^4}^2+ 2\ep^2 \N^4 \|\psi\|^2_{L^2}.\end{equation}
For simplicity, we will also denote $S_\ep$ for $S_\ep(u_\ep)$ and $\T$ for $\T(u_\ep)$, as well as
\begin{equation}\label{Sv}
S_\v:= \v \otimes \v - \hal |\v|^2 I.\end{equation}

\begin{lem}\label{divst} Let  $u_\ep$ solve \eqref{eq} or \eqref{gpr} and $\v$  and $\phi$ be as above. Then we have
\begin{multline}\label{219}
\div \T(u_\ep)=  \frac{\N\a}{\lep} \la \p_t u_\ep -iu_\ep \N \phi , \nab u_\ep -iu_\ep \N \v\ra +   \frac{\beta}{2}\N V_\ep-\beta \N^2\v \la \p_t u_\ep,u_\ep\ra\\
- \N^2 \frac{\N\a}{\lep}|u_\ep|^2 \v \phi + \N j_\ep ( \frac{\N\a}{\lep} \phi  -\div \v)
\\
+ \N^2 \div S_\v- \N\((\v \cdot \nab) j_\ep + (j_\ep \cdot \nab ) \v -\nab (j_\ep \cdot \v)\),
\end{multline}
which in dimension $n=2$ can be rewritten
\begin{align}\label{divtt2}
\div \T(u_\ep) = & \frac{\N\a}{\lep}  \la \p_t u_\ep -iu_\ep \N \phi , \nab u_\ep -iu_\ep \N \v\ra  +\frac{\b}{2}\N V_\ep   -\b \N^2 \v  \la \p_t u_\ep, u_\ep \ra  \\
\nonumber &  + \N^2\v^\perp\curl \v -\N j_\ep^\perp \curl \v - \N \v^\perp \mu_\ep
\\ \nonumber &+\N^2  \v (\div \v - \frac{\N\a}{\lep}|u_\ep|^2 \phi )+ \N j_\ep (\frac{\N\a}{\lep} \phi-\div \v).\end{align} \end{lem}
\begin{proof}
First, a direct computation yields
$$\T(u_\ep)= S_\ep(u_\ep) + \N^2 S_\v -\N ( \v\otimes j_\ep + j_\ep \otimes \v -  (j_\ep \cdot \v) I).$$
Since we  have the following relations for general vector fields $\v $ and $j$:
\begin{equation}\label{divvj}
\div (\v \otimes j ) = j \div \v + (\v \cdot \nab ) j ,\end{equation}
we deduce that 
\begin{multline}\label{divstilde}\div \T(u_\ep)= \div S_\ep(u_\ep) + \N^2 \div S_\v\\
 - \N \(j_\ep \div \v + \v \div j_\ep + (\v \cdot \nab) j_\ep + (j_\ep \cdot \nab ) \v - \nab (j_\ep \cdot \v)\).\end{multline}
On the other hand,
writing $\p_t u_\ep = \p_t u_\ep- iu_\ep \N \phi
+ iu_\ep \N \phi$  and $\nab u_\ep= \nab u_\ep - iu_\ep \N \v+ iu_\ep \N \v$
 yields 
 \begin{multline*}
 \la \p_t u_\ep , \nab u_\ep\ra = \la \p_t u_\ep -iu_\ep \N \phi , \nab u_\ep -iu_\ep \N \v\ra + \N j_\ep \phi \\+ \N \v \la \p_t u_\ep, iu_\ep \ra  -\N^2 |u_\ep|^2 \v \phi  
\end{multline*}
and combining with \eqref{defV}, \eqref{divji} and 
  \eqref{divT}, we find
\begin{multline*}\div S_\ep(u_\ep)=  \frac{\N\a}{\lep}\Big( \la \p_t u_\ep -iu_\ep \N \phi , \nab u_\ep -iu_\ep \N \v\ra + \N j_\ep \phi  -\N^2 |u_\ep|^2 \v \phi  \Big)\\+ \N \v \div j_\ep
-\beta \N^2\v \la \p_t u_\ep,u_\ep\ra
+ \N \frac{\beta}{2}V_\ep.\end{multline*}
Inserting into \eqref{divstilde} yields \eqref{219}.
In the two-dimensional case, we notice that we have the identities 
$$\div S_\v= \v \div \v + \v^\perp \curl \v$$ and 
$$ (\v \cdot \nab) j + (j \cdot \nab ) \v - \nab (j \cdot \v)= j^\perp\curl \v + \v^\perp \curl j, $$ so \eqref{divtt2} follows.



\end{proof}

\subsubsection{Time derivative of the energy}

 Given a Lipschitz compactly supported function $\chi(x)$,  and a sufficiently regular function $\psi(x,t)$  let us  define 
\begin{equation}\label{defER}
\hat\E(u, t)= \hal \int_{\R^n} \chi \(|\nab u-iu\N \v(t)|^2 +\frac{(1-|u|^2)^2}{2\ep^2}+ \N^2  (1-|u|^2 )\psi(x,t) \).
\end{equation}
For simplicity we  will most often write $\hat\E(u_\ep)$  for $\hat \E(u_\ep(t), t)$. 

\begin{lem}\label{lemdte} Let $u_\ep$ solve \eqref{eq} or \eqref{gpr} and $\v$ satisfy  the results of Lemma \ref{lemv}. Then we have 
\begin{align}\label{eqlemdte}
  \frac{d}{dt} \hat\E(u_\ep)
 =&  -\int_{\R^n} \chi  \frac{\N\a}{\lep} |\p_t u_\ep|^2   + \nab \chi  \cdot  \la \nab u_\ep-iu_\ep \N \v , \p_t u_\ep\ra   
\\  \nonumber & + \int_{\R^n} \chi  \( \N^2 \v\cdot \p_t \v  -  \N j_\ep \cdot \p_t \v +\N \la \p_t u_\ep, iu_\ep\ra \div \v- \N V_\ep  \cdot \v\) 
\\  \nonumber &+   \hal \int_{\R^n}   \chi \N^2   \p_t \((1-|u_\ep|^2) (\psi -|\v|^2 ) \).
\end{align}
\end{lem}
\begin{proof}
Since the solution $u_\ep$ is smooth and $\chi$ is compactly supported, expanding the square, we may first  rewrite $\hat\E$ as  
\begin{multline}\label{dever}
\hat\E(u_\ep)= \hal \int_{\R^n} \chi  \( |\nab u_\ep|^2 + \frac{(1-|u_\ep|^2)^2}{2\ep^2} \) \\+ \hal \int_{\R^n}\chi\(      \N^2|\v|^2  + \N^2(1-|u_\ep|^2) ( \psi-   |\v|^2)   \)  -\int_{\R^n} \chi  \N j_\ep\cdot \v.\end{multline}
We then  differentiate in time and obtain 
\begin{multline*}
 \frac{d}{dt} \hat\E(u_\ep)=-  \int_{\R^n} \chi  (\la \p_t u_\ep, \Delta u_\ep +\frac{1}{\ep^2} u_\ep(1-|u_\ep|^2) \ra+\nab \chi  \cdot \la \nab u_\ep, \p_t u_\ep\ra \\ 
+\int_{\R^n} \chi (\N^2 \v \cdot \p_t \v - \N j_\ep \cdot \p_t \v -\N \v \cdot \p_t j_\ep ) 
   +  \hal \int_{\R^n}   \chi \N^2   \p_t \((1-|u_\ep|^2) (\psi -|\v|^2 ) \) .\end{multline*} 
 Inserting \eqref{eq} and \eqref{dtj} and also writing $\nab u_\ep =\nab u_\ep-iu_\ep \N \v+ iu_\ep \N \v$, this becomes
\begin{align*}
  \p_t \hat\E(u_\ep)
 =&  -\int_{\R^n} \chi  \frac{\N\a}{\lep} |\p_t u_\ep|^2   + \nab \chi  \cdot\(  \la \nab u_\ep-iu_\ep \N \v , \p_t u_\ep\ra  +
 \N \v \la \p_t u_\ep, iu_\ep\ra \)  
\\ & + \int_{\R^n} \chi  \( \N^2 \v\cdot \p_t \v  -  \N j_\ep \cdot \p_t \v -\N \v \cdot \nab \la \p_t u_\ep, iu_\ep\ra - \N V_\ep  \cdot \v\) 
\\ &+  \hal \int_{\R^n}   \chi \N^2   \p_t \((1-|u_\ep|^2) (\psi -|\v|^2 ) \).
\end{align*}
With   an integration by parts,  we find that two terms simplify and  we obtain the result.
\end{proof}

\section{The Gross-Pitaevskii case : proof of Theorem \ref{th1}}

In this section, we consider the Gross-Pitaevskii cases, in which $\v$ solves \eqref{lime} with $\alpha=0$ and $\beta=1$ or \eqref{incel},  and for which $\div \v=0$, $\phi=\pp$ and $\psi=\pp-|\v|^2$. Below, we apply the result of Lemma \ref{lemdte} with these choices.
First, we insert the equation solved by $\v$ and the result of Lemma~\ref{divst} to obtain the crucial step in our proof, where all the algebra combines. 
We note that in view of \eqref{vdivv}, 
\eqref{lime} and \eqref{incel} can both be written as 
\begin{equation}
\label{inceln} 
\p_t \v= 2\div S_\v + \nab \pp\end{equation}with the notation \eqref{Sv}.
\begin{lem}\label{dte2}Let $u_\ep$ solve \eqref{gpr}  and $\v$ solve \eqref{lime} or \eqref{incel} according to the dimension.  Then we have 
\begin{align}
\label{compen}
  & \p_t \hat\E(u_\ep)= \int_{\R^n} \chi (2   \T(u_\ep): \nab \v    -\N(1-|u_\ep|^2)\p_t (|\v|^2-\frac{\pp}{2})
 \\  \nonumber
&   -\int_{\R^n} \nab \chi  \cdot\Big(   \la \nab u_\ep-iu_\ep \N \v , \p_t u_\ep \ra +  \N(\N \v - j_\ep)  \pp     -2\T  \v \Big) ,
\end{align}
where for two $2\times 2$ matrices $A $ and $B$, $A:B$ denotes $\sum_{kl} A_{kl} B_{kl}$. 
\end{lem}
\begin{proof}
Starting from the result of Lemma \ref{lemdte}, 
the first  step is to insert \eqref{inceln}, which yields
\begin{equation}\label{31}
 \int_{\R^n} \chi  \( \N^2 \v\cdot \p_t \v  -  \N j_\ep \cdot \p_t \v \)= 
\int_{\R^n}\chi \N (\N  \v -j_\ep) \cdot ( 2\div S_\v + \nab \pp ).\end{equation}
In  order to transform the (first order in the error) term  $(\N \v-j_\ep)\cdot 
\div S_\v$ into a quadratic term, we 
 multiply the result of Lemma~\ref{divst} by $2\chi  \v $, which yields 
\begin{multline}\label{neV}
 - \int_{\R^n} \chi \N V_\ep \cdot \v 
  =-2\int_{\R^n} \chi   \div \T(u_\ep)\cdot  \v
   +2 \int_{\R^n} \chi \N^2 \div S_\v\cdot \v \\-  2 \N \int_{\R^n} \chi \((\v\cdot \nab ) j_\ep +(j_\ep\cdot \nab) \v -\nab (j_\ep \cdot \v)\) \cdot \v     
    - \int_{\R^n}\chi \N^2 |\v|^2  \p_t  |u_\ep|^2. 
\end{multline}But by direct computation, one may check that  for any vector fields $\v $ and $j$ we have 
\begin{equation}\label{vdj}
\((\v\cdot \nab ) j +(j\cdot \nab) \v -\nab (j \cdot \v)\) \cdot \v  = - \div S_\v\cdot j\end{equation}
and  applying to $\v $ and $\v$, we easily deduce that 
$$\div S_\v \cdot \v =0.$$
Therefore, inserting  \eqref{vdj} applied with $\v$ and $j_\ep$,  \eqref{31} and \eqref{neV}  into the result of Lemma \ref{lemdte}, and noticing several cancellations, we obtain 
\begin{align*}
  \frac{d}{dt} \hat\E(u_\ep)
 =&  -\int_{\R^n} \nab \chi  \cdot  \la \nab u_\ep-iu_\ep \N \v , \p_t u_\ep\ra   \\
 & + \int_{\R^n}\chi  (\N(\N \v-j_\ep) \cdot \nab \pp   - 2   \div \T(u_\ep)\cdot  \v)  
  \\   &+ \int_{\R^n} \chi\(  \N^2(  |\v|^2  \p_t (1-|u_\ep|^2)   +\hal \N^2 \p_t \( (1-|u_\ep|^2) (\psi- |\v|^2 )\) \right) . 
\end{align*}
Integrating by parts and using \eqref{divji}, we have 
\begin{equation*}
 \int_{\R^n} \chi \N (\N \v-j_\ep) \cdot \nab \pp= - \int_{\R^n} \N \nab \chi  \cdot (\N \v - j_\ep)  \pp +\int_{\R^n} \N^2\chi   \la \p_t u_\ep , u_\ep\ra \pp.\end{equation*}
Inserting into the previous relation and collecting terms, we are led to 
\begin{align*}
  &\frac{d}{dt}\hat\E(u_\ep)
 =  -\int_{\R^n} \nab \chi  \cdot\(   \la \nab u_\ep-iu_\ep \N \v , \p_t u_\ep\ra +  \N(\N \v - j_\ep)  \pp\)
\\  & -2\int_{\R^n} \chi  \div \T(u_\ep)\cdot  \v \\
  &+ \int_{\R^n} \chi  \N^2\(    \p_t (1-|u_\ep|^2) \( |\v|^2 -  \hal  \pp\)+\hal \p_t \((1-|u_\ep|^2) (\psi- |\v|^2 ) \)\) . 
\end{align*}
Since we have chosen $\psi= \pp-|\v|^2$ we see that the terms involving $\p_t (1-|u_\ep|^2)$ cancel and we get the conclusion.
\end{proof}
We may check that all the terms in factor of $\chi$ or $\nab \chi $ in \eqref{compen}  are in $L^1([0,t]\times\R^n)$, thanks to  \eqref{borneS}, Lemmas  \ref{lemv} and \ref{apriori}. 
Inserting  for $\chi$ in  \eqref{compen} a sequence $\{\chi_k\}_k$ of functions bounded in $C^1(\R^n)$ such that $\chi_k \to 1$ and $\|\nab \chi_k\|_{L^\infty(\R^n)} \to 0$, we  may thus integrate \eqref{compen} in time and, using the finiteness of $\E(u_\ep(t),t)$ given by Lemma \ref{apriori},    let $k \to \infty$ to  obtain by dominated convergence, 
\begin{multline*}
   \E(u_\ep(t),t)- \E(u_\ep^0,0) 
     = \int_0^t  \int_{\R^n}   2 \T(u_\ep): \nab \v -  \N^2 (1-|u_\ep|^2) \p_t ( |\v|^2 -  \frac{  \pp}{2}).
\end{multline*}
Inserting \eqref{bornints}, the bounds given by Lemma \ref{lemv}, and using the Cauchy-Schwarz inequality to control $(1-|u_\ep|^2)$ by $\E$ in view of \eqref{contmod}, we arrive at 
\begin{align*}
\lefteqn{
\E(u_\ep(t),t)-\E(u_\ep^0,0)  
} \qquad & \\
& \le C\int_0^t \E(u_\ep) + C \N^2 \int_0^t \( \N^2 \ep^2+ \ep \sqrt{\E(u_\ep)+C\ep^2 \N^4}\) \\
& \le C\int_0^t \E(u_\ep) + C t \ep^2 \N^4,
\end{align*}where $C$ depends only on the bounds on $\v$.
Applying Gronwall's lemma and using \eqref{condN1}, we finally obtain that 
$$\E(u_\ep(t),t) \le C_t\( \E(u_\ep^0,0) +o(\N^2)\)$$
which is $o(\N^2)$ with the initial data assumption \eqref{eu0}. The conclusions of the theorem follow in view of either \eqref{contjl1} or \eqref{contj}.

\section{The parabolic  case}

In this section, we turn to the proof of Theorem~\ref{th2}, which only concerns the dimension $n=2$.  Throughout we assume that $\a=1$, $\b=0$ and \eqref{condN2} holds.

In the rest of the paper,  we  let $\chi_R$ be 
\begin{equation}\label{defchir}
\chi_R(x) =\begin{cases} & \frac{\log |x|-\log R^2}{\log R- \log R^2} \text{ for }  R\le |x|\le R^2\\
& \chi_R(x)=1  \text{ for }  |x|\le R\\
& \chi_R (x)=0  \text{ for } |x|\ge R^2.\end{cases}\end{equation}
With this choice we have 
\begin{equation}\label{condchir}
 \|\nab \chi_R\|_{L^\infty(\R^2)} \to 0  \qquad \|\nab \chi_R\|_{L^2(\R^2) } \to 0 \quad \text{as } \ R \to \infty.\end{equation}

\subsection{A priori bound on the velocity}
We define  $\te$ as the maximum time $t\le \min(1,\t)$ (where $\t>0$ is the time of existence of the solution to the limiting equation) such that 
\begin{equation}\label{bap}
\E(u_\ep(t)) \le \pi \N \lep + \N^2 \quad \text{ for all } \  t \le \te.\end{equation}
Our goal is  to  show that $\E(u_\ep(t))\le \pi \N \lep +o(\N^2)$ for all $t\le \te$, which will imply that $\te=\min(1,\t)$. 

Let us start with a crude a priori bound on the time derivative of $u_\ep$. 
\begin{lem}\label{aprioridtu}
 Assume  \eqref{condN2} and $\E(u_\ep^0) \le \pi \N \lep + o(\N^2)$. Assume $\v$ satisfies the results of Lemma \ref{lemv}. Then 
$$\int_0^{\te} \i|\p_t u_\ep|^2 \le   C \N^3\lep^3    ,$$ where $C$ depends only on the bounds on $\v$.
\end{lem}
\begin{proof}
Let us return to the result of Lemma \ref{lemdte}  with  the choice $\psi=- |\v|^2$  
and $\chi=\chi_R$  as in \eqref{defchir}.  By \eqref{defV}, we have $V_\ep= 2\la i \p_t u_\ep, \nab u_\ep\ra = 2\la i\p_t u_\ep, \nab u_\ep -iu_\ep \N \v\ra + 2\N \v \la \p_t u_\ep, u_\ep\ra$, and inserting 
  into the result of Lemma \ref{lemdte}, we obtain  
\begin{multline*}
\p_t \hat \E (u_\ep) = 
-\i \chi_R  \frac{\N}{\lep} |\p_t u_\ep|^2 - \i \nab \chi_R  \cdot \la \nab u_\ep-iu_\ep \N\v, \p_t u_\ep\ra \\+ \i \chi_R\(-\N^2 (1-|u_\ep|^2) \p_t |\v|^2 + \N\la  \p_t u_\ep ,iu_\ep\ra \div \v \right. \\ \left. +\N(\N\v-j_\ep )\cdot \p_t \v
+ 2 \N  \la  i \p_t u_\ep, \nab u_\ep -iu_\ep \N \v \ra  \cdot  \v\).
\end{multline*}We next observe that again all the terms in factor of $\chi_R $ or $\nab \chi_R$ are in $L^1(\R^2)$ thanks to Lemmas \ref{lemv} and \ref{apriori}. We may then integrate in time and let $R \to \infty$ to obtain 
\begin{align*}
\lefteqn{ 
\int_0^{\te} \i  \frac{\N}{\lep} |\p_t u_\ep|^2 
} \qquad & \\
& = \E(u_\ep^0)- \E(u_\ep(\te)) 
+ \int_0^{\te} \i \N \la  \p_t u_\ep,iu_\ep \ra \div \v + \N (\N \v-j_\ep) \cdot  \p_t \v \\
& \qquad + 2\N \la i \p_t u_\ep, \nab u_\ep -iu_\ep \N \v \ra  \cdot  \v+o(1),
\end{align*}
where we also used \eqref{condN2} and \eqref{bap} to control all the terms containing $(1-|u_\ep|^2)$.  Next we  insert \eqref{pjj}  to obtain 
\begin{align*}
& \int_0^{\te} \i  \frac{\N}{\lep} |\p_t u_\ep|^2 
 \le  \E(u_\ep^0) \\ & 
+ \int_0^{\te} \i\Big( \N  | \p_t u_\ep | |\div \v|  + |1-|u_\ep|| |\p_t u_\ep| |\div \v|+  \N |\nab u_\ep -iu_\ep \N \v|| \p_t \v |\\ & \qquad + \N |1-|u_\ep||  |\nab u_\ep- iu_\ep \N \v||\p_t \v| + \N^2 |1-|u_\ep|^2||\v| |\p_t \v|\\ & \qquad + 2\N |\p_t u_\ep| | \nab u_\ep -iu_\ep \N \v ||  \v|\Big)+o(1).
\end{align*}
Using  $|1-|u_\ep||\le |1-|u_\ep|^2|$, the Cauchy-Schwarz inequality,    \eqref{eu02}, \eqref{bap}, the $L^\infty \cap L^2 $ character of $\div \v $ and $\p_t \v$ given by  Lemma \ref{lemv}, the boundedness of $\v$ and  Lemma   \ref{contrmod}, we deduce that 
\begin{multline*}
\int_0^{\te} \i \frac{\N}{\lep} |\p_t u_\ep|^2 \\ \le   \pi \N \lep + C\N \int_0^{\te}(1+ \|\p_t u_\ep\|_{L^2} )(1+ \E(u_\ep ))\, dt +o(\N^2).\end{multline*}
Using \eqref{bap}, bounding $\te$ by $1$ and using \eqref{condN2}, we easily deduce the result. 
\end{proof}

\subsection{Preliminaries: ball construction and product estimate}
The proof in the parabolic case is more involved than in the Gross-Pitaevskii case, and in particular it requires all the machinery to study vortices which has been developed over the years. Indeed, as explained in the introduction, we will need to subtract off the (now leading order) contribution of the vortices to the energy.  This will be done via the ball-construction method (introduced in \cite{sa,j}) coupled with the ``Jacobian estimates" \cite{js} (with precursors in  \cite{br,ss1}). Here we will need a lower bound with smaller errors, as in \cite[Theorems 4.1 and 6.1]{livre}, coupled to an improvement due to \cite{stice1}.

\subsubsection{Vorticity to modulated vorticity} 
Before doing so, we will need the following result which connects the vorticity to the modulated vorticity.
\begin{lem}\label{mvor}Assume $\v $ satisfies the results of Lemma \ref{lemv} and $u_\ep$ is such that $\E(u_\ep)<\infty$. 
Let $\mu_\ep = \curl j_\ep$ and $\tilde \mu_\ep= \curl (\la \nab u_\ep -iu_\ep \N \v,iu_\ep\ra+ \N \v)$ as in \eqref{tildemu}.
We have that $\tilde \mu_\ep \in L^1(\R^2)$ with 
\begin{equation}\label{tmu}
|\tilde \mu_\ep |\le 2 |\nab u_\ep -iu_\ep \N \v|^2+ |1-|u_\ep|^2| |\curl \v|\end{equation} 
and \begin{equation}\label{intmun}
\i \tilde\mu_\ep =2\pi \N.\end{equation}
Moreover, for any $\xi\in H^1(\R^2)$, we have 
\begin{equation}\label{mumu}
\left|\i \xi(\mu_\ep - \tilde \mu_\ep) \right|\le C \|\nab \xi\|_{L^2(\R^2)} \N(\ep \sqrt{\E(u_\ep)}+ C \ep^2 \N^2).
\end{equation}
\end{lem}
\begin{proof} First, a direct computation gives that \eqref{tmu} holds, and it follows immediately with Lemmas \ref{lemv} and  \ref{apriori} that $\tilde \mu_\ep \in L^1(\R^2)$.  We may then write, with $\chi_R$ as in \eqref{defchir}  
\begin{multline*}\i \tilde \mu_\ep = \lim_{R \to \infty} \i \chi_R \tilde \mu_\ep\\= -\lim_{R\to \infty} \i \np \chi_R\cdot (\la \nab u_\ep -iu_\ep \N \v, iu_\ep\ra +  \N (\v- \la \nab U_1, iU_1\ra+ \la \nab U_1, iU_1\ra) .\end{multline*}  
Since  $\nab u_\ep-iu_\ep \N \v$ and $\v -\la \nab U_1, iU_1\ra$ are in $L^2$ by Lemmas \ref{lemv} and \ref{apriori}, the corresponding terms tend to $0$ as $R \to \infty$. There remains 
\begin{multline*}\i \tilde \mu_\ep =- \lim_{R \to \infty}  \i \N \np \chi_R \cdot  \la \nab U_1, iU_1\ra\\ = \lim_{R \to \infty} \i \N \chi_R \curl \la \nab U_1, iU_1 \ra= 2\pi \N\end{multline*} by choice of $U_1$. 
For \eqref{mumu} we observe that by a direct computation, we have
$\tilde \mu_\ep - \mu_\ep = \curl ((1-|u_\ep|^2) \N \v)$. Thus,  for any $\xi \in H^1(\R^2)$ we have
\begin{align*}
\left|\i \xi(\mu_\ep - \tilde \mu_\ep) \right| 
& = \left |\N \i (1-|u_\ep|^2) \np \xi \cdot \v\right| \\
& \le   \N\|\nab \xi\|_{L^2(\R^2)} (C\ep \sqrt{\E(u_\ep)+  \ep^2 \N^4}),
\end{align*} 
where we used Lemma \ref{contrmod}.\end{proof}

\subsubsection{Jacobian estimate for unbounded domains}
 The next lemma is an infinite domain version of the estimate of \cite[Theorem 6.1]{livre}. For this lemma, we temporarily use the  notation $\mu_\ep$ with a slightly different meaning. 
 \begin{lem}\label{lemjac}
 Let $\Omega $ be an open subset of $\R^2$, and let $\Omega^\ep= \{x \in \Omega, \dist(x, \p \Omega) >\ep\}$.   Let $u_\ep: \Omega \to \mathbb{C}$ and $A_\ep : \Omega \to \R^2$. 
   Assume that $\{B_i\}_i $ is a finite  collection of disjoint closed balls of centers $a_i$ and radii $r_i$  covering $\{| |u_\ep|-1|\ge\hal \}\cap \Omega^\ep$, and let $d_i = \deg(u_\ep, \p B_i) $ if $B_i \subset \Omega^\ep$ and $d_i=0$ otherwise.
  Then, setting 
  $$\mu_\ep = \curl (\la \nab u_\ep -iA_\ep u_\ep, iu_\ep\ra + A_\ep), $$
   we have, for any $\xi \in C^{0,1}_c(\Omega)$, 
  \begin{multline}\label{bornJ}
\left|  \int_\Omega \xi (\mu_\ep - 2\pi \sum_i d_i \delta_{a_i}  ) \right|\\
\le C \Big( \sum_i r_i+\ep\Big)    \|\xi\|_{C^{0,1}(\Omega)} \int_{\Omega}  2 |\nab u_\ep -iu_\ep A_\ep|^2+|1-|u_\ep|^2| |\curl A_\ep|  + \frac{(1-|u_\ep|^2)^2}{2\ep^2},
\end{multline} where $C$ is universal.
\end{lem}
\begin{proof} As in \cite[Chap. 6]{livre}, we set 
  $\chi:\R_+\to \R_+$ to be defined by 
$$\left\{\begin{array}{ll}\chi(x)=2x & \text{if} \ x\in [0,\hal]\\
\chi(x)=1 & \text{if} \ x\in [\hal,\frac{3}{2}]\\
\chi(x)=1+2(x-3/2)& \text{if} \ x\in [\frac{3}{2},2]\\
\chi(x)=x& \text{if} \ x\ge 2.\end{array}\right.
$$ We then let $v_\ep(x)= \frac{\chi(|u_\ep|)}{|u_\ep|} u_\ep$. This is clearly well-defined, with $|v_\ep|=1$ outside of $\cup_i B_i$ and 
from  \cite[Chap. 6]{livre} or direct calculations,  we know that 
\begin{multline}\label{diffmd0}
\| \la \nab  u_\ep -iA_\ep u_\ep ,iu_\ep\ra -\la \nab  v_\ep-i A_\ep v_\ep,iv_\ep\ra\|_{L^1(\Omega)}\\ \le 3\|1-|u_\ep|\|_{L^2(\Omega)} \|\nab u_\ep-i A_\ep u_\ep\|_{L^2(\Omega)},\end{multline}  
and 
\begin{equation}\label{diffmd20}
|1-|v_\ep||\le |1-|u_\ep||, \qquad |\nab v_\ep - iA_\ep v_\ep |\le 2|\nab u_\ep -iA_\ep u_\ep|.
\end{equation}
Letting $\hat \mu_\ep = \curl (\la \nab v_\ep-i A_\ep v_\ep,  iv_\ep\ra + A_\ep)$, we have 
\begin{multline}\label{JJ}
\left|\int_{\Omega} \xi ( \mu_\ep - \hat \mu_\ep )\right|= \left|\int_\Omega \np \xi 
\cdot (\la \nab u_\ep -iu_\ep A_\ep, iu_\ep  \ra- \la \nab v_\ep-i  A_\ep v_\ep,  iv_\ep\ra )\right|\\
\le  3\|\nab \xi\|_{L^\infty(\Omega)}  \|1-|u_\ep|^2\|_{L^2(\Omega)} \|\nab u_\ep -i A_\ep u_\ep \|_{L^2(\Omega)}.\end{multline}
Next, we note that $\hat \mu_\ep$ vanishes wherever $|v_\ep|=1$, and thus as soon as $||u_\ep|-1|\le \hal$. Thus, by property  of the balls, we have $\supp \hat \mu_\ep \cap \Omega^\ep \subset \cup_i B_i$ (recall the definition of $\Omega^\ep$ in the statement of the lemma).  We also have that  whenever $B_i \subset \Omega^\ep$, it holds that  $\int_{B_i} \hat \mu_\ep = 2\pi d_i$ (see \cite[Lemma 6.3]{livre}).  Writing $\xi$ as $\xi(a_i)+ O(r_i) \|\xi\|_{C^{0,\1}}$ in each $B_i$, we conclude, exactly as in the proof of \cite[Theorem 6.1]{livre}. The only point that is a bit different is we need an analogue of  \cite[Lemma 6.4]{livre} to bound $\sum_i r_i \int_{B_i} | \mu_\ep|$  which works on an unbounded domain. For that, we may check by direct computation that  \begin{equation}\label{calmu}
\mu_\ep= 2\la \nab u_\ep -iA_\ep u_\ep,i(\nab u_\ep - i A_\ep u_\ep)\ra+ (1-|u_\ep|^2 ) \curl A_\ep  ,\end{equation} so
$$\int_\Omega |\mu_\ep |\le  2\int_\Omega 2 |\nab u_\ep -i A_\ep u_\ep|^2+ |1-|u_\ep|^2||\curl A_\ep|$$ and the same holds for $\hat \mu_\ep$ with a factor $4$ in front, in view of \eqref{diffmd20}.  We thus obtain that 
\begin{multline*}
\left|  \int_\Omega \xi (\mu_\ep - 2\pi \sum_i d_i \delta_{a_i}  ) \right|
\\ \le  C (\sum_i r_i  +\ep) \|\xi\|_{C^{0,1}} \int_{\cup_i B_i\cup (\Omega\backslash \Omega^\ep) }  2 |\nab u_\ep -i A_\ep u_\ep|^2+|1-|u_\ep|^2| |\curl A_\ep|  
\\+ C \|\xi\|_{C^{0,1}}\|1-|u_\ep|^2\|_{L^2(\R^2)} \|\nab u_\ep -i  A_\ep u_\ep\|_{L^2(\R^2)}, 
\end{multline*} where $C$ is  universal, and we easily deduce the result.
\end{proof}

\subsubsection{Ball construction lower bound, sharp version} In the result below, we use the Lorentz space $L^{2, \infty}$ as in \cite{stice}, which can be defined by 
\begin{equation}\label{l2infini}
\|f\|_{L^{2,\infty}(\R^2)} = \sup_{|E|<\infty} \int_E |f|\end{equation}  where $|E|$ denotes the Lebesgue measure of $E$.
\begin{pro}\label{lem43}
Assume $\E(u_\ep)\le \pi \N \lep+ \N^2$ where  $\N$ satisfies \eqref{condN2}, and $\v\in C^{1,\gamma}(\R^2)$ with $\curl \v\in L^2(\R^2)$. Then there exists $\ep_0$  such that for any $\ep<\ep_0$, the following holds.
There exists a finite collection of disjoint closed balls $\{ B_i= B(a_i, r_i)\}_i$ such that, letting $d_i= \deg(u_\ep, \p B_i)$, the following holds
\begin{enumerate}
\item $\sum_i r_i  \le e^{-\sqrt{\N}}$.
\item $\{ x, ||u_\ep|-1| \ge \hal  \} \subset\cup_{i} B(a_i, r_i).$
\item $$\hal \int_{\cup_i B_i} |\nab u_\ep-iu_\ep \N \v|^2 +\frac{(1-|u_\ep|^2)^2}{2\ep^2} \ge \pi \sum_i |d_i| \lep - o(\N^2).$$   
\item
$$\|\nab u_\ep -iu_\ep\N\v\|_{L^{2,\infty}(\R^2)}^2 \le
 C  \(\E(u_\ep) -    \pi \sum_i |d_i| \lep     + \sum_i |d_i|^2\)+o(\N^2).
$$
\item  Letting $\tilde \mu_\ep$ be as in \eqref{tildemu},  for any $\xi \in C^{0,\gamma}_c(\R^2)$, we have
\begin{equation}
\left|\int \xi (2\pi \sum_i d_i \delta_{a_i}- \tilde \mu_\ep) \right|\le o(1) \|\xi\|_{C^{0,\gamma}} .
\end{equation}
\end{enumerate}
\end{pro}
\begin{proof} The result of  \cite[Theorem 4.1]{livre} applied to $u_\ep$, $A_\ep= \N \v$ and $\alpha=3/4$, provides for $\ep $ small enough, for  any $ \ep^{1/4} \le r <1$     a collection of disjoint closed balls $\mathcal B(r)
$   covering $\{ x, ||u_\ep|-1| \ge \ep^{3/16}  \}$, such that the sum of the radii of the balls in the collection is $r$, and such that  denoting $D:=\sum_{B\in \mathcal B(r)} |d_B|$, we have
$$\hal \int_{\cup_{B\in \mathcal B(r)} }  |\nab u_\ep-iu_\ep \N \v|^2 +r^2 |\N\curl\v|^2+ \frac{(1-|u_\ep|^2)^2}{2\ep^2} \ge \pi D \( \log \frac{  r }{D\ep}-C\).$$
 One notes that the fact that we are in an unbounded domain does not create any problem. Indeed, since $\E(u_\ep)<\infty$, this implies (cf. e.g. \cite[Lemma 2.3]{j} or \cite[Lemma 3.5]{hanli})that there exists a radius $R_\ep$ such that $|u_\ep|\ge \hal$ outside of $B(0,R_\ep)$ and $$\int_{B(0,R_\ep)^c} |\nab u_\ep -iu_\ep \N \v|^2 + \frac{(1-|u_\ep|^2)^2}{2\ep^2} +\N^2 (1-|u_\ep|^2) \psi <1$$
 i.e. the remaining energy is smaller than the desired error. This way the construction can be applied in $B(0,R_\ep+1)$ only, yielding always a finite collection of balls covering $\{|u_\ep|\le \hal\}$.

We first apply this result with the choice  $r=r'= \ep^{1/4}$ to obtain a collection  of balls $\{B_j'\}$ 
with centers $a_j'$, radii $r_j'$ and degrees $d_j'$, covering $\{ x, ||u_\ep|-1| \ge \hal  \}$ and satisfying $\sum_j r_j'\le \ep^{1/4}$. The  estimates of \cite{livre} also yield that 
$$\sum_j |d_j'| \le \frac{C}{\lep}\int_{\cup_j B_j'} |\nab u_\ep-iu_\ep \N \v|^2 +\ep^{1/2}\N^2 |\curl \v|^2+\frac{(1-|u_\ep|^2)^2}{2\ep^2}.$$
Using that $\curl \v$ is bounded, the fact that $\sum r_i'\le \ep^{1/4}$, the upper bound on $\E(u_\ep)$ and \eqref{condN2}, we deduce from this relation that $\sum_j|d_j'| \le C\N$. 

 We next apply the above result with the choice  $r= e^{-\sqrt{\N}}$.
This gives a  collection of balls $\{B_i\}= \{B(a_i,r_i)\}$ of degrees $d_i$,  satisfying  items 1 and 2 of the proposition and 
\begin{multline*}
\hal \int_{\cup_i B_i} |\nab u_\ep-iu_\ep \N \v|^2 +e^{-2\sqrt{\N}} |\N \curl \v|^2+\frac{(1-|u_\ep|^2)^2}{2\ep^2} \\ \ge \pi \sum_i|d_i|\( \log \frac{   e^{-\sqrt{\N}}   }{\sum_i |d_i|\ep}-C\).\end{multline*}
It is part of the statements of \cite[Theorem 4.1]{ss} that the family $\mathcal{B}(r)$ is increasing in $r$, i.e. here that  the $B_i$'s cover the $B_j' $'s. By additivity of the  degree, we  thus have  $\sum_i |d_i|\le \sum_j |d_j'|\le C \N$, and using  $\sum_i r_i\le e^{-\sqrt{\N}}$ and the boundedness of $\curl \v$,  the 
desired estimate of item 3 follows.

For item 4,  we use \cite[Corollary 1.2]{stice1}  which yields that  \begin{multline}
 \|\nab u_\ep -iu_\ep \N \v\|_{L^{2,\infty} (\R^2) }^2 \\ \le  C \(\E(u_\ep)- \pi \sum_i| d_i|\(\log \frac{r}{\ep \sum_i|d_i|}-C\)    +\pi \sum_i |d_i|^2\)+o(\N^2).
 \end{multline}
This is  essentially a strengthened version of the result of item 3, in which the difference between the two sides of the inequality is shown to bound from above $\|\nab u_\ep -iu_\ep \N \v\|_{L^{2,\infty} (\R^2) }^2$.

Let us now turn to item 5, which is an adaptation to an infinite setting of the Jacobian estimates, for instance as in \cite[Theorem 6.1]{livre}. The reason we needed a two-step construction above is that the total  radius of the second set of  balls, $e^{-\sqrt{\N}}$,  (which have to be chosen this large so that they contain enough energy)  is not very small compared to  $\lep$ when $\N$ is not very large, and thus the Jacobian estimate applied directly on these large balls would give too large of an error.
 
Letting $\xi $ be a  smooth test-function, we may write
$$\i\Big( \sum_i  d_i \delta_{a_i}- \sum_j d_j' \delta_{a_j'}\Big) \xi= \sum_i \Big( d_i \xi(a_i)- \sum_{j, B_j' \subset B_i}  d_j' \xi(a_j')\Big).$$Since $\sum_{j, B_j' \subset B_i } d_j'=d_i$ and $\sum_j |d_j'|\le C \N$,  we  may write 
 \begin{multline}\label{m1}\left|\i\Big( \sum_i  d_i \delta_{a_i}- \sum_j d_j' \delta_{a_j'}\Big) \xi\right|
 \le \|\xi\|_{C^{0,1}}\sum_i  r_i^\gamma(\sum_{j, B_j' \subset B_i}  |d_j'|)\\ \le\|\xi\|_{C^{0,1}} (\sum_i r_i)\sum_i \sum_{j, B_j'\subset B_i}  |d_j'| \le C  \|\xi\|_{C^{0,1}} e^{-\gamma\sqrt{\N}} \N =o(1) \|\xi\|_{C^{0,1}}.\end{multline} 
Applying Lemma \ref{lemjac}  with $A_\ep = \N \v$ on $\R^2$, and it yields that 
 \begin{multline*}
\left| \i \xi (2\pi\sum_j d_j' \delta_{a_j'}-\tilde \mu_\ep) \right| \\ \le  C \ep^{1/4} \|\xi\|_{C^{0,1}} \i |\nab u_\ep -iu_\ep \N\v|^2+\N |1-|u_\ep|^2| |\curl \v|  + \frac{(1-|u_\ep|^2)^2}{2\ep^2}. \end{multline*} Using that $\curl \v\in L^1(\R^2)$, the upper bound on  $\E(u_\ep)$, and  combining with \eqref{m1}, we obtain the desired result for $\gamma=1$. The result for $\gamma<1$ follows by interpolation as in \cite{js}, using that $\sum_j |d_j'|\le C \N$   and $\i |\tilde \mu_\ep| \le C \N \lep $ hence  $\|2\pi \sum_j d_j' \delta_{a_j'}-\tilde \mu_\ep\|_{(C^0)^*} \le C \N$.
 \end{proof}
 \subsubsection{Ball construction lower bound, localized version}
We will also  need a less precise but localizable version of the ball construction. This can be borrowed directly from \cite{j,sa,ss1}, and combined with the Jacobian estimate of Lemma \ref{lemjac}, so we omit the proof. 
\begin{lem}\label{blocal}Under the assumptions of Proposition \ref{lem43},  there exists $\ep_0$ such that for all $\ep<\ep_0$, there exists a finite collection of disjoint closed balls $\{B_i\}_i=\{B(a_i,r_i)\}_i$ such that, letting $d_i= \deg(u_\ep, \p B_i)$, the following holds
\begin{enumerate}
\item $\sum_i r_i \le e^{-\sqrt{\lep}}$.
\item 
$$\forall i, \quad \hal \int_{B_i} |\nab u_\ep-iu_\ep \N \v|^2+ \frac{(1-|u_\ep|^2)^2}{2\ep^2} \ge \pi |d_i| \lep(1-o(1)).$$
\item For any $0<\gamma \le 1$ and any $\xi \in C^{0,\gamma}_c(\R^2)$, we have
$$\left|\i \xi \( 2\pi \sum_i d_i \delta_{a_i}- \tilde \mu_\ep\) \right|\le o(1)\|\xi\|_{C^{0,\gamma}} . $$ 
\end{enumerate}
\end{lem}
We emphasize that these balls are not necessarily the same as those obtained by Proposition \ref{lem43}.

\subsubsection{Consequences on  the energy excess}
We next show how the energy excess $\E- \pi \N \lep$ controls various quantities, including the energy outside of the balls. 
\begin{coro}
 For any $t \le \te$ (where $\te$ is as in \eqref{bap}), letting $\{B_i\}_{i}$ (depending on $\ep, t$) be the collection of  balls given by Proposition \ref{lem43},  we have
 \begin{multline}\label{exces}
\hal \int_{\R^2 \backslash \cup_{i} B_i} |  \nab u_\ep -iu_\ep \N \v|^2 + \frac{(1-|u_\ep|^2)^2}{2\ep^2} \\
  \le \E(u_\ep(t),t) -\pi \N \lep +o(\N^2),
 \end{multline}
   \begin{equation}\label{contsumdi} \text{for $\ep$ small enough }\qquad 
 \N\le\sum_i |d_i|\le C \N,\end{equation}
 \begin{equation}\label{contlor} \|\nab u_\ep -iu_\ep \N \v\|_{L^{2,\infty}(\R^2)}\le C \N,\end{equation}
and for any nonnegative $\xi \in C^{0, \gamma}(\R^2)$,
\begin{multline}\label{exces2}
\hal \int_{\R^2}\xi \(\frac{1}{\lep}  |  \nab u_\ep -iu_\ep \N \v|^2 - \tilde \mu_\ep\) \\ \le  \|\xi\|_{L^\infty}\frac{1}{\lep} ( \E(u_\ep(t),t)-\pi \N \lep ) +o(\N)\|\xi\|_{C^{0,\gamma}} .
 \end{multline} \end{coro}
\begin{proof}
First, applying item 5 of Proposition \ref{lem43} with $\xi= \chi_R$ as in \eqref{defchir} and letting $R \to \infty$,  we must have $\i (2\pi \sum_i d_i \delta_{a_i} -\tilde \mu_\ep)= o_\ep(1)$. Comparing with  \eqref{intmun}, we deduce that $\sum_i d_i=   \N$ for $\ep $ small enough. Subtracting the result of item 3 of Proposition \ref{lem43} from $\E(u_\ep(t),t)$ and using Lemma \ref{contrmod}, we then obtain \eqref{exces}. The upper bound in \eqref{contsumdi} was proved in the course of the proof of Proposition \ref{lem43}, the lower bound is an obvious consequence of $\sum_i d_i=\N$.

The relation \eqref{contlor} is a direct consequence of item 4 of Proposition \ref{lem43}, \eqref{bap} and \eqref{contsumdi}, writing $\sum_i |d_i|^2 \le (\sum_i |d_i|)^2$.

Finally, for \eqref{exces2} we use instead the balls given by Lemma \ref{blocal}.    From item~2 of Lemma \ref{blocal} we may write  
\begin{align*}&\hal \int_{\R^2 \backslash \cup_{i} B_i} |\nab u_\ep -iu_\ep \N \v|^2 + \frac{(1-|u_\ep|^2)^2}{2\ep^2}   \\
&+\sum_{i}   \hal \int_{B_i}|\nab u_\ep -iu_\ep \N \v|^2 + \frac{(1-|u_\ep|^2)^2}{2\ep^2}  - \pi |d_i| \lep (1-o(1))
 \\ &\le \E(u_\ep(t),t) - \pi \sum_i |d_i|\lep (1+o(1)).\end{align*}Moreover,  from the same argument as before we have $\sum_i d_i=\N$  and $\sum_i |d_i|\le C\N$ for these balls, hence since the terms of the  above sum are all nonnegative, adding to both sides $ \pi \sum_i (|d_i|-d_i)\lep$, we obtain
 \begin{align*}&\hal \int_{\R^2 \backslash \cup_{i} B_i} |\nab u_\ep -iu_\ep \N \v|^2 + \frac{(1-|u_\ep|^2)^2}{2\ep^2}   \\
&+\sum_{i}  \left| \hal \int_{B_i}|\nab u_\ep -iu_\ep \N \v|^2 + \frac{(1-|u_\ep|^2)^2}{2\ep^2}  - \pi d_i \lep (1-o(1))\right|
 \\ &\le \E(u_\ep(t),t) - \pi \N \lep (1+o(1)).\end{align*}
Next,  separating the integral  between $\cup_{i} B_i$ and the complement,  and using item~1 of Lemma \ref{blocal}, we deduce that 
\begin{multline}\label{exces20}
\hal \int_{\R^2}\xi |  \nab u_\ep -iu_\ep \N \v|^2 + \xi \frac{(1-|u_\ep|^2)^2}{2\ep^2} 
\le \pi  \sum_i d_i \xi(a_i)  \lep(1+o(1))
\\+o(1)\|\xi\|_{C^{0,\gamma}} + \|\xi\|_{L^\infty} \( \E(u_\ep(t),t)-\pi \N \lep(1+o(1))\) .
 \end{multline}
Combining this with item 3 of Lemma \ref{blocal}, we deduce that  \eqref{exces2} holds.
\end{proof}

\subsubsection{Approximation of $\v$} 
We will need  an approximation of $\v(t)$ based on the balls constructed via Proposition \ref{lem43}.
\begin{lem} \label{lembl} Let $\v$ satisfy the assumptions of Theorem \ref{th2} and $u_\ep$ satisfy \eqref{bap}.  For each $t\le \te$, letting $\{B_i\}_{i}$ be the collection of   balls constructed in Proposition~\ref{lem43}, there exists a vector field $\bar \v: \R^2 \to \R^2$ (depending on $\ep$ and $t$)  such that 
\begin{enumerate}
\item $\bar\v$ is constant in each $B_i$,
\item for every $\gamma\in [0,1]$, $\| \bar\v-\v\|_{C^{0,\gamma}} \le C (\sum_i r_i)^{1-\gamma}\le C e^{-(1-\gamma )\sqrt{\N}}$ where $C$ depends only on $\v$ and $\gamma$,
\item $\bar \v-\v$ has compact support.
\end{enumerate}
\end{lem}

\begin{proof} It is an adaptation of Proposition 9.6 of \cite{livre}, which we can apply to each component of $\v$.  We note that since the collection of balls is finite,  we may replace $\bar \v$ by $\v+ \chi(\bar \v-\v)$ where $\chi$  is a smooth positive cut-off function which is equal to $1$ on a large enough ball containing $\cup_i B_i$ and vanishes outside of a large enough ball. This makes $\bar \v- \v$ compactly supported without affecting the other properties.  \end{proof}

By continuity of $u_\ep$ (for fixed $\ep$) and of $\v$, one may check that we may make $\bar \v(t)$ measurable in $t$. While we have a good control on $\nab \bar \v$, we have no  control on $\p_t \bar \v$, and this is what prevents us from applying this method in the regime $\N \le O(\lep)$ in the Schr\"odinger case.

\subsubsection{Product estimate}
Finally,  to control the velocity of the vortices, we also need the following result, whose proof is postponed to  Appendix \ref{app1}, and which is an $\ep$-quantitative version of  the ``product estimate" of \cite{ss6}.

We let $M_\ep$  be a quantity such that 
\begin{equation}\label{condM}\forall q>0, \  \lim_{\ep \to 0}  M_\ep \ep^q=0 , \ \lim_{\ep \to 0} \frac{\lep}{M_\ep^q}=0, \ \lim_{\ep \to 0} \frac{\log M_\ep }{\lep}=0.\end{equation}
For example $M_\ep= e^{\sqrt{\lep}}$ will do. 
In the statement below we do not aim at optimality, however we state the result in a way that would allow to  go beyond the range $\N \le O(\lep)$ that we are considering here (for example up to arbitrary powers of $\lep$ or quantities that satisfy the properties \eqref{condM}).  

 \begin{pro}\label{estprod}  Let $u_\ep: [0,\tau]\times \R^2 \to \mathbb{C}$.  Let $\v$ be a solution to \eqref{lime} or \eqref{limg}  and $\phi$ be as in \eqref{defphi}.  Let $\tilde V_\ep$ be as in \eqref{tV}. 
Let $X\in C^{0,1}([0, \tau]\times \R^2, \R^2)$ be a spatial vector field. 
Set  $$F_\ep = \int_0^\tau\( \i |\p_t u_\ep- iu_\ep \N \phi|^2 + \E(t)\,\) dt,$$ and assume $F_\ep \le M_\ep$. 
 Then for any $\Lambda \ge 1$, we have, as $\ep \to 0$,
\begin{multline}\label{V1}
\left|\int_0^\tau \i \tV \cdot X \right| \\ \le \frac{1+C\frac{\log M_\ep}{\lep}}{\lep} \( \frac{1}{\Lambda } \int_0^\tau \i |\p_t u_\ep- iu_\ep \N \phi|^2 +  \Lambda  \int_0^\tau\i |(\nab u_\ep -iu_\ep \N \v) \cdot X |^2 \)  \\ +   
    C \|X\|_{L^\infty}  \(\(\Lambda^3   (1+\|X\|_{C^{0,1}}) M_\ep^{-1/8}  +\ep \N\) \(F_{\ep} + \N^2 \)   +\sqrt{F_\ep \sup_{t\in [0, \tau] } \E(t)    }M_\ep^{-1/8} \)  ,
 \end{multline}where $C$ depends only on the bounds on $\phi$ and $\v$.
\end{pro} The second line in the right-hand side is $o(1)$, so is $\log M_\ep/\lep$.
One can see that optimizing over $\Lambda$ by taking 
$$\Lambda = \( \frac{\int_0^\tau \i |\p_t u_\ep- iu_\ep \N \phi|^2}{ \int_0^\tau \i|(\nab u_\ep -iu_\ep \N \v) \cdot X |^2 } \)^\hal$$
yields a right-hand side in the form of a product (plus error terms), hence the  name ``product estimate".

\subsection{Proof of Theorem \ref{th2}}
We now  present the main proof. 
\subsubsection{Evolution of the modulated energy}
In the  next result,  we take as before $\psi=-|\v|^2 $ in the definition \eqref{defER}, and  insert the equation solved by $\v$ and \eqref{divtt2} into   \eqref{eqlemdte} to obtain the crucial computation. 
\begin{lem}\label{dte22}Let $u_\ep$ solve \eqref{eq} and $\v$ solve \eqref{lime} or \eqref{limg}. Assume that \eqref{condN2} and \eqref{bap} hold. 
Then, for any $t \le \te$, we have 
\begin{equation}\label{5}
  \E(u_\ep(t),t)-\E(u_\ep^0,0) = I_S+I_V+I_E +I_D+ I_d+ I_m+I_v+o(1)
  \end{equation}
  where 
  \begin{align*}
& I_S = 2 \int_0^t\i \T :  \nab \bar\v^\perp\\
 & I_V=-\int_0^t \i \N   \tV \cdot \v  \\
 & I_E=  -\int_0^t \i 2 \N|\v|^2\tilde \mu_\ep \\
 &I_D= 
  -\int_0^t  \i   \frac{\N}{\lep}\left|\p_t u_\ep- iu_\ep \N \phi \right|^2 -
  \N  \la \p_t u_\ep-iu_\ep \N \phi , iu_\ep\ra ( \div \v   - \frac{\N}{\lep} \phi )
 \\
&\qquad \quad - 2 \frac{\N}{\lep} \la \p_t u_\ep- iu_\ep \N\phi, \nab u_\ep -iu_\ep \N \v \ra 
\cdot  \bar\v^\perp \\ & I_d= \int_0^t 
\i  2 \N\bar \v^\perp \cdot  (j_\ep -\N \v) 
 ( \frac{\N}{\lep} \phi- \div \v)
\\
  & I_m= \int_0^t  \i  \N^2  (1-|u_\ep|^2)\(- \p_t |\v|^2  +\v\cdot \nab \phi   + 2 \frac{\N}{\lep}\phi \v\cdot \bar \v^\perp\)  
 \\& I_v= \int_0^t \i 2 \N (j_\ep - \N \v)\cdot   (\v -\bar \v)  \curl \v + 2 \N (\bar \v - \v) \cdot \v \tilde\mu_\ep  . 
\end{align*}
\end{lem}
\begin{proof}
Again, we start from the result of Lemma \ref{lemdte} applied with $\psi= -|\v|^2$. The first step is to write 
\begin{multline}\label{pdtu2}
- \i \chi  \frac{\N }{\lep} |\p_t u_\ep|^2 \\= -\i   \chi \frac{\N}{\lep} \(|\p_t u_\ep -iu_\ep \N \phi|^2 - \N^2  |u_\ep|^2 \phi^2 +2  \N \phi \la  \p_t u_\ep,iu_\ep \ra \).\end{multline}
The second step is to use  \eqref{eqglob} to write \begin{multline}\label{3}
 \i \chi  \( \N^2 \v\cdot \p_t \v  -  \N j_\ep \cdot \p_t \v \)= 
\i \chi \N^2 (- 2 |\v|^2  \curl \v + \v \cdot \nab \phi )\\
 -\i \chi  \N j_\ep \cdot (-2 \v \curl \v+ \nab \phi).
\end{multline}
On the other hand,
integrating by parts and using \eqref{divji}, we have 
\begin{multline}\label{jf}
 \i \chi \N (\N \v-j_\ep) \cdot \nab \phi \\= - \i \N \nab \chi  \cdot (\N \v - j_\ep)  \phi -\i \N^2\chi \( \div \v - \frac{1}{\lep} \la \p_t u_\ep, iu_\ep\ra \)\phi.\end{multline}
 
 Next, we would like to transform the  linear term  
 $\i \chi (-2\N^2|\v|^2 \curl \v + 2\N j_\ep \cdot \v \curl \v)$ into a quadratic term plus error terms. For that, we would like to multiply the result of Lemma \ref{divst} by $2  \v^\perp$, which after integration by parts leads to terms in $\i  \T : \nab \v^\perp$. Using $\bar \v$ rather than $\v$ leads instead to integrals that live only {\it outside of the balls} since $\nab \bar \v=0$ there by item  1 of Lemma \ref{lembl}. These terms  will thus  be controlled by the energy outside of the balls, i.e.~the excess energy, as desired.  
This creates an additional  set of error terms in $\bar \v - \v$, which we will control thanks to item 2 in Lemma \ref{lembl}.

So as explained,  let us  multiply the result of Lemma \ref{divst} by $2 \chi \bar\v^\perp$, where $\bar \v$ is given, for each time $t$, by the result of Lemma \ref{lembl}.  This yields 
\begin{align*}
&\i  2\chi  \bar \v^\perp \cdot \div \T(u_\ep) \\ &
=  \i \chi     \left(   2\frac{\N}{\lep} \la \p_t u_\ep- iu_\ep \N\phi , \nab u_\ep-iu_\ep \N \v \ra   \right) \cdot   \bar\v^\perp
\\ &  +\i \chi\( 2 \N^2 |\v|^2 \curl \v  -2 \N j_\ep \cdot \v)  \curl \v
  -2\N|\v|^2 \mu_\ep \right)\\
  & + \i \chi \( 2 \N^2 (\bar \v -\v) \cdot \v \curl \v -2 \N j_\ep \cdot  (\bar \v - \v) \curl \v 
 \right. \\
  \nonumber & \qquad \left. - 2 \N (\bar \v - \v) \cdot \v \mu_\ep\)
  \\ & + 
  \i \chi \(2 \N^2   \v \cdot  \bar \v^\perp(\div \v - \frac{\N}{\lep}|u_\ep|^2 \phi )+ 2\N j_\ep\cdot  \bar \v^\perp (\frac{\N}{\lep} \phi-\div \v)\) .
\end{align*}
Inserting \eqref{tV}, \eqref{pdtu2},  \eqref{3}  and \eqref{jf} into the result of Lemma \ref{lemdte} applied with $\psi= -|\v|^2$,  taking advantage of the  cancellations and using one integration by parts,  we obtain 
\begin{align}\label{4}
  &\frac{d}{dt}
  \hat\E(u_\ep)
 =  - \i \chi   \frac{\N}{\lep} \(\left|\p_t u_\ep- iu_\ep \N \phi \right|^2 - \N^2  |u_\ep|^2|\phi|^2 \)\\
\nonumber &  -\i \nab \chi  \cdot ( \la \nab u_\ep-iu_\ep \N \v , \p_t u_\ep\ra  +\N(\N \v - j_\ep)  \phi  - 2 \T \bar \v^\perp )
\\   \nonumber &
+\i \chi \(\N  \la \p_t u_\ep, iu_\ep\ra \Big( \div \v   - \frac{\N}{\lep} \phi \Big)- \N^2 \phi \div \v \)
\\
\nonumber & +\i \chi  \(  \N ( - \tV - \N \v \p_t |u_\ep|^2 + \N \phi \nab |u_\ep|^2)  \cdot    \v    +2\T :  \nab \bar\v^\perp   \)    \\
 \nonumber &+ \i  \chi \( 2\frac{\N}{\lep} \la \p_t u_\ep- iu_\ep \N\phi, \nab u_\ep -iu_\ep \N \v \ra \cdot \bar\v^\perp       -2 \N|\v|^2 \mu_\ep  \right)  
\\
\nonumber   & + \i \chi \( 2 \N^2 (\bar \v -\v) \cdot \v \curl \v -2 \N j_\ep \cdot (\bar \v - \v) \curl \v - 2 \N (\bar \v - \v) \cdot \v \mu_\ep \)
 \\
\nonumber &
+\i \chi  2  \N \bar \v^\perp \cdot \( (j_\ep -\N \v) 
 ( \frac{\N}{\lep} \phi- \div \v) +  \frac{\N}{\lep}  \N(1-|u_\ep|^2) \phi  \v\)
 \\  \nonumber &- \i \chi  \N^2 \p_t ((1-|u_\ep|^2) |\v|^2 ) . 
\end{align}
Let us make three transformations to this expression. 
First, let us  single out  the terms
\begin{multline*}
\i \chi\N^2(- |\v|^2\p_t |u_\ep|^2  +  \phi \v\cdot \nab |u_\ep|^2 -\p_t( (1-|u_\ep|^2) |\v|^2)\\ = \i \chi \N^2   (1-|u_\ep|^2)  (-\p_t |\v|^2 +\nab \phi \cdot \v  +\phi \div \v )+\i   \N^2 \nab \chi \cdot \v  (1-|u_\ep|^2)\phi
\end{multline*} (where we have used an integration by parts).

Second, let us  replace $\la \p_t u_\ep, iu_\ep\ra$ by $\la \p_t u_\ep -iu_\ep \N \phi, iu_\ep \ra +\N |u_\ep|^2 \phi$ in  \eqref{4}, this leads to a cancellation of the terms in $(1-|u_\ep|^2) \phi \div \v$ and in $|u_\ep|^2\phi^2$. 

Third, owing to \eqref{mumu} let us replace $-2 \i \chi \N |\v|^2 \mu_\ep$ by $$-2\i \chi \N |\v|^2 \tilde \mu_\ep + O \(  (\|\nab \chi\|_{L^2} + \|\nab \v \|_{L^\infty} \| \v\|_{L^2})\N(\ep \sqrt{\E(u_\ep)}+  \ep^4 \N^2) \)$$
and the same for $-2\i \N (\bar \v-\v) \cdot \v \mu_\ep$. Both give rise to $o(1)$ error terms by \eqref{condN2} and \eqref{bap}.

After these substitutions, let us  integrate in time and take $\chi=\chi_R$. We may check via Lemmas \ref{apriori} and \ref{lemv}, and the fact that $\bar \v -\v$ is compactly supported,  that all integrands in factor of $\chi_R$ or $\nab \chi_R$ are in $L^1$.
Letting $R \to \infty$, we then get the conclusion.
\end{proof}
We next turn to studying these terms one by one. We will show that $I_d, I_m,I_v$ are all negligible terms, while $I_V, I_E$ and $ I_D$ recombine algebraically thanks to the product estimate, to give a term bounded by the energy outside the balls, as does $I_S$. 

\subsubsection{The negligible terms}
Let us start with $I_d$.   In the case $\N \ll \lep$, from \eqref{defphi} we have $\phi= \pp$ and we also have $\div \v=0$. Then, in view of \eqref{bap} and \eqref{contj} we may bound 
\begin{multline*}|I_d|\le  2\i\N \frac{\N}{\lep} |\bar \v| |j_\ep -\N \v| |\pp|\\ \le  C  \|\bar \v\|_{L^\infty} \|\pp\|_{L^2} \sqrt{\pi \N \lep + \N^2} \frac{\N^2}{\lep} =o(\N^2),\end{multline*}
where we used  the boundedness of $\v$ hence of $\bar \v$ and the $L^2$ character of $\pp$ (see Lemma \ref{lemv}).
In the case $\frac{\lep}{\N} \to \lambda>0$ finite (see \eqref{defl}), we have  $\phi = \lambda \div \v$ and  
\begin{equation}\label{divvv}
\frac{\N}{\lep} \phi - \div \v=   (\frac{\lambda \N}{\lep}-1) \div \v =o(1) |\div \v|.\end{equation}
  We again conclude easily that $|I_d|\le o(\N^2)$ in that regime too.

The term $I_m$ is easily seen to be $o(\N^2)$,  using the Cauchy-Schwarz inequality,  \eqref{bap} and Lemma \ref{contrmod} and the  integrability of   $\p_t \v$,  $\pp$,   $\nab \pp$,   $\div \v$ and $\nab \div \v$ provided by Lemma \ref{lemv}. 

To bound the first terms of $I_v$ we first use \eqref{pjj} and the fact that $\v\in C^{1,\gamma}$ to get
\begin{multline}\label{above}
\left|\i (j_\ep -\N \v) \cdot (\v -\bar \v) \curl \v\right| \le  C
 \|\v - \bar \v\|_{L^\infty} \Big(\i |\nab u_\ep -iu_\ep\N \v| |\curl \v|\\
 +  |1-|u_\ep|| |\nab u_\ep - iu_\ep \N \v| + \N^2  |1-|u_\ep|^2| |\curl \v|\Big)\end{multline}
 Using  the Cauchy-Schwarz inequality, \eqref{contmod},  \eqref{bap}, \eqref{condN2}, and the $L^2$ character of $\curl \v$ (by Lemma \ref{lemv}, $\curl \v \in L^1(\R^2)\cap L^\infty(\R^2)$), we find that the last two terms in the right-hand side integral give $o(1)$ terms. To bound the contribution of the first term,  we split the integral over $\R^2$ using the balls $B_i$ given by Proposition \ref{lem43} as follows~:
 \begin{multline*}
 \i |\nab u_\ep -iu_\ep\N \v| |\curl \v|\\ \le 
  \int_{\cup_i B_i} |\nab u_\ep -iu_\ep \N \v||\curl \v|+ \int_{\R^2 \backslash \cup_i B_i} |\nab u_\ep -iu_\ep \N \v||\curl \v|\\
\le  C\|\nab u_\ep - iu_\ep \N \v\|_{L^{2,\infty}}|\cup_i B_i|+
\(\int_{\R^2 \backslash \cup_i B_i} |\nab u_\ep -iu_\ep \N \v|^2\)^{\hal}  \|\curl \v\|_{L^2(\R^2)},\end{multline*}
where we used  \eqref{l2infini} and the boundedness of $\curl \v$. In view of \eqref{exces},  \eqref{contlor},  \eqref{bap} and item 1 of Proposition \ref{lem43},  we deduce that 
$$ \i |\nab u_\ep -iu_\ep\N \v| |\curl \v|\le C e^{-2\sqrt{\N}}\N + C \N$$
and inserting into \eqref{above} and using item 2 of Lemma \ref{lembl}, we deduce that 
$$\left|\i (j_\ep -\N \v) \cdot (\v -\bar \v) \curl \v\right| \le o(1).$$

For the second term of $I_v$ we apply item 3   of Lemma \ref{blocal} with $\xi=\v-\bar \v$ to obtain 
\begin{multline*}
\i (\v -\bar \v) \tilde \mu_\ep =2\pi  \sum_i d_i (\v-\bar \v)(a_i)+   o(1) \|\v-\bar \v\|_{C^{0,\gamma}}\\ \le  2\pi \sum_i |d_i| \|\v-\bar \v\|_{L^\infty}  +o(1)\|\v-\bar \v\|_{C^{0,\gamma}}  = o(1),\end{multline*}
where we used item 2 of Lemma \ref{lembl} and the fact that $\sum_i|d_i|\le C \N$ for these balls. 
We deduce that $I_v=o(1)$ and conclude that $I_d+I_m+I_v=o(\N^2)$.
\subsubsection{The dominant terms}
The term $I_S$ can easily  be treated with the help of \eqref{borneS}. Using that from Lemma \ref{lembl}, $\nab \bar \v$ vanishes outside of the balls $B_i$  given by Proposition \ref{lem43}  and is bounded otherwise by a constant depending on $\v$,
and using  \eqref{borneS},  \eqref{exces} and Lemma \ref{contrmod}, we may write for each $t \le \te$,
\begin{align*}
& \left|\i  2\T:  \nab \bar\v^\perp \right|\\ & \le 2 \|\nab \bar\v\|_{L^\infty} \int_{\R^2 \backslash \cup_i B_i}  |\nab u_\ep - iu_\ep \N\v|^2 + \frac{1}{4\ep^2}(1-|u_\ep|^2)^2 + \N^2|\v|^2 | 1-|u_\ep|^2|\\
& \le C(\E(u_\ep(t))- \pi \N \lep)  + o(\N^2).\end{align*}
We thus conclude that 
\begin{equation}\label{IS}
|I_S|\le C\(\int_0^t \E(u_\ep(s))- \pi \N \lep \) +o(\N^2).\end{equation}

For the term $I_D$ we replace $\bar \v^\perp $ by $\v^\perp + (\bar \v- \v)^\perp$, and using Young's inequality, we write
\begin{multline} \label{IC1}  
  I_D  \le   \frac{\N}{\lep}   \( -\hal  \int_0^t\i \left|\p_t u_\ep -  iu_\ep\N\phi \right|^2 +  2  \int_0^t \i \left| ( \nab u_\ep-iu_\ep \N \v) \cdot  \v^\perp  \right|^2\)\\
  \quad   +  \|\bar \v-\v\|_{L^\infty} \frac{\N}{\lep}\(  \int_0^t\i \left|\p_t u_\ep -  iu_\ep\N\phi \right|^2 +  \int_0^t \i \left|  \nab u_\ep-iu_\ep \N \v  \right|^2 \) \\
  +  \int_0^t\i  \N   | \p_t u_\ep-iu_\ep \N \phi | |\div \v   - \frac{\N}{\lep} \phi | .
\end{multline}
We claim that 
\begin{multline}\label{3ligne}
 \int_0^t\i  \N   | \p_t u_\ep-iu_\ep \N \phi | |\div \v   - \frac{\N}{\lep} \phi |\\ \le o\( \frac{\N}{\lep}\)\int_0^t \i \left|\p_t u_\ep-i u_\ep \N \phi \right|^2+o(\N^2).\end{multline}
Indeed, either $\N \ll \lep$, in which case $\div \v =0$ and $\phi=\pp$  and the result follows from the Cauchy-Schwarz inequality after inserting a factor $$\frac{1}{\sqrt{\N}} \( \frac{\N}{\lep}\)^{1/4} \sqrt{\N} \( \frac{\lep}{\N}\)^{1/4}$$ and the $L^2$ character of $\pp$; or $\lep/\N \to \lambda $ and 
we may use \eqref{divvv} and the $L^2$ character of $\div \v$ to conclude the same.
On the other hand, by \eqref{bap} and item~2 of Lemma \ref{lembl}, we have 
\begin{multline}\label{4ligne}
 \|\bar \v-\v\|_{L^\infty} \frac{\N}{\lep} \(   \int_0^t\i \left |\p_t u_\ep -  iu_\ep\N\phi \right|^2 +  \int_0^t \i \left|  \nab u_\ep-iu_\ep \N \v  \right|^2 \) \\ \le  C e^{-\sqrt{\N}}    \frac{\N}{\lep} \int_0^t\i \left |\p_t u_\ep -  iu_\ep\N\phi \right|^2+o(1).\end{multline}

\noindent
We next distinguish two cases:
\\
{\it Case 1:} the case where 
 $$\int_0^t \i |\p_t u_\ep-  i u_\ep\N \phi |^2\le 20 \int_0^t \i |(\nab u_\ep -iu_\ep \N \v) \cdot  \v|^2 .$$
By \eqref{bap} this implies  that  
\begin{equation}\label{bornintt}
\int_0^t \i |\p_t u_\ep-  i u_\ep \N \phi|^2\le 20 \|\v\|_{L^\infty} ( \pi \N \lep + \N^2) .\end{equation}
It follows,  together with   \eqref{bap} and  \eqref{condN2} that $F_\ep \le C\N\lep$ where $F_\ep $ is as  in Proposition \ref{estprod}. From Lemma \ref{lemv} we have that $\p_t \v$ is uniformly bounded while $\v \in C^{1,\gamma}$ in space, hence $\v$ is Lipschitz in space-time, so we may apply Proposition \ref{estprod}   with $M_\ep= e^{\sqrt{\lep}}$, $\tau=t$, $\Lambda=2$ and $X= \v $  
to obtain \begin{multline}\label{V2}
I_V\le  \frac{\N}{\lep}  \( \hal \int_0^t \i \left|\p_t u_\ep-i u_\ep \N \phi \right|^2 \right. \\
\left. +2\int_0^t \i |(\nab u_\ep -iu_\ep \N \v ) \cdot\v |^2   \)+o(\N^2).   \end{multline}
\medskip

\noindent
{\it Case 2:} the case where 
$$\int_0^t \i |\p_t u_\ep-  i u_\ep\N \phi |^2\ge 20 \int_0^t\i |(\nab u_\ep -iu_\ep \N \v) \cdot \v |^2 .$$ We may rewrite that condition as  \begin{multline}\label{V21}\frac{1}{4} \int_0^t \i |\p_t u_\ep  -  i u_\ep \N \phi |^2+ 4 \int_0^t \i |(\nab u_\ep -iu_\ep \N \v) \cdot \v |^2 \\ \le \(\frac{1}{4}+\frac{1}{10}\)   \int_0^t \i |\p_t u_\ep  - i u_\ep \N \phi|^2+ 2 \int_0^t \i |(\nab u_\ep -iu_\ep \N \v) \cdot \v |^2.\end{multline}
We note that in that situation, thanks to Lemma \ref{aprioridtu} and the $L^2$ character of $\phi$, we have 
$F_\ep \le  2\int_0^t \i |\p_t u_\ep-  i u_\ep \N \phi|^2\le C \N^3 \lep^3  $, where $F_\ep$ is as in Proposition~\ref{estprod}. Choosing $M_\ep= e^{\sqrt{\lep}}$ we may then apply that proposition  with $M_\ep= e^{\sqrt{\lep}}$, $\tau=t$, $\Lambda=4$ and $X=\v$, and combining the result with \eqref{V21}, we are led to 
\begin{multline}\label{V22}
I_V \le  \frac{\N}{\lep}   (\frac{1}{4}+\frac{1}{10}) \int_0^t \i \left|\p_t u_\ep-  i u_\ep \N\phi\right|^2\\
 +\frac{2\N}{\lep} \int_0^t \i |(\nab u_\ep -iu_\ep \N \v ) \cdot\v |^2    +o\(\frac{\N}{\lep}\) \(  \int_0^t \i |\p_t u_\ep  - i u_\ep \N\phi |^2  \) 
 +o(1).   \end{multline}
 This implies  that
 \begin{multline}\label{V3}
I_V\le  \frac{\N}{\lep}  \( \hal \int_0^t \i \left|\p_t u_\ep-i u_\ep \N \phi \right|^2 
+2\int_0^t \i |(\nab u_\ep -iu_\ep \N \v ) \cdot\v |^2   \) \\ -  \frac{1}{8}  \frac{\N}{\lep}\int_0^t \i \left|\p_t u_\ep-i u_\ep \N \phi \right|^2+o(1).   \end{multline}

Returning to the general situation, 
we may now combine in the first case \eqref{IC1}, \eqref{3ligne}, \eqref{4ligne}, \eqref{bornintt} and \eqref{V2}, and in the second case \eqref{IC1}, \eqref{3ligne}, \eqref{4ligne} and \eqref{V3}. Noticing an exact recombination of the terms
\begin{multline*} -\hal  \int_0^t\i \left|\p_t u_\ep -  iu_\ep\N\phi \right|^2 +  2  \int_0^t \i \left| ( \nab u_\ep-iu_\ep \N \v) \cdot  \v^\perp  \right|^2\\+ \hal \int_0^t\i \left|\p_t u_\ep -  iu_\ep\N\phi \right|^2+ 2 \int_0^t \i \left| ( \nab u_\ep-iu_\ep \N \v) \cdot  \v \right|^2,\end{multline*}
we obtain in both cases that
 \begin{equation}\label{ICIV}
I_D+I_V\le  \frac{2\N}{\lep}  \int_0^t \i |\nab u_\ep- iu_\ep \N \v|^2 |\v|^2+o(\N^2).
\end{equation}
On the other hand, from \eqref{exces2} applied with $\xi=|\v|^2$, we have
\begin{multline*}2\N \i  \frac{1}{\lep}|\nab u_\ep- iu_\ep \N \v|^2|\v|^2  -|\v|^2\tilde \mu_\ep\\ \le( \|\v\|_{L^\infty}^2+1) \frac{\N}{\lep} (\E(u_\ep)-\pi \N \lep)+o(\N^2),\end{multline*} so using that $\N \le O(\lep)$ and combining with \eqref{ICIV}, we obtain that 
\begin{equation}\label{IdIv}
I_V+I_D+I_E\le    C \int_0^t \(\E(u_\ep(s))- \pi \N \lep\)\, ds +o(\N^2).\end{equation}
Let us point out that  this is the only place in the proof where we really are limited to the situation where $\N \le O(\lep)$. 
\subsubsection{The Gronwall argument}
Combining \eqref{IdIv} with \eqref{IS} and the result on the negligible terms, we are led to 
$$\E(u_\ep(t))- \E(u_\ep^0) \le C \int_0^t \(\E(u_\ep(s))- \pi \N \lep\)\, ds +o(\N^2)$$
and this holds for any $t \le \te$.

In view of the assumption on the initial data,  Gronwall's lemma immediately  yields that 
\begin{equation}\label{et}
\E(u_\ep(t),t)\le \pi \N \lep +o(\N^2)\end{equation} for all $t\le \te$. Thus we must have $\te=\min(1, \t)$ and we may extend the argument up to time $\t$ to obtain that \eqref{et} holds until $\t$. This proves the first assertion of Theorem \ref{th2}.
 
 \subsubsection{The convergence result}
To conclude the proof of Theorem \ref{th2}, there remains to check that this implies that $\frac{\la \nab u_\ep, iu_\ep\ra }{\N} \to \v$ in $L^p_{loc}(\R^2)$ for $p<2$.
In view of \eqref{exces}, \eqref{et} implies that for every $t \le \t$, 
\begin{equation}\label{bextb}
\int_{\R^2 \backslash \cup_{i} B_i} |\nab u_\ep - iu_\ep \N \v|^2 \le o(\N^2),\end{equation}
hence for any ball   $B_R$ centered at the origin and any $1\le p<2$, by H\"older's inequality, 
\begin{equation}\label{intbr}
\int_{B_R \backslash \cup_{i} B_i} |\nab u_\ep-iu_\ep \N \v|^p \le o(\N^p).\end{equation} 
Meanwhile, by \eqref{contlor} 
and using the embedding of $L^{2,\infty}(B_R)$ into $L^q(B_R)$ for any $q<2$,  by H\"older's inequality and  item 1 of Proposition \ref{lem43}, we deduce that for any $p<q<2$, 
$$\int_{\cup_{i}B_i} |\nab u_\ep -iu_\ep\N \v|^p \le o(1).$$
Combining this with \eqref{intbr}, we conclude 
 that  $\frac{1}{\N} (\nab u_\ep-iu_\ep\N \v) \to 0$ in $L^p_{loc}(\R^2)$.

In the case $\frac{\lep}{\N}\to \lambda$, with  \eqref{contmod} and the upper bound \eqref{et}, we have in addition  that  $\frac{1}{\N} (\nab u_\ep - iu_\ep \N \v)$ is bounded in $L^2(\R^2)$. It thus has a weak  limit $f$, up to extraction of a subsequence $\ep_k$.  As in \cite[p. 151]{livre}, letting $\Omega_\ep$ denote  $\cup_i B_i$ for each given $\ep$,  since $\sum_i r_i \le e^{-\sqrt{\N}}$, we may extract a further subsequence such that  $\mathcal A_n:= \cup_{k\ge n} \Omega_{\ep_k}$ has Lebesgue measure tending to $0$ as $n \to \infty$. For any fixed $n$, by weak convergence we have 
$$\liminf_{k\to \infty} \int_{\R^2 \backslash \Omega_{\ep_k} }\frac{|\nab u_\ep -iu_\ep \N \v|^2 }{\N^2}\ge  \liminf_{k\to \infty} \int_{\R^2 \backslash \mathcal A_n}\frac{ |\nab u_\ep -iu_\ep \N \v|^2 }{\N^2}\ge \int_{\R^2 \backslash \mathcal A_n}|f|^2,$$
but the left-hand side is equal to $0$ by \eqref{bextb}, so letting $n \to \infty$,  we deduce that $f$ must be $0$. Since this is true for any subsequence, we conclude that $\frac{1}{\N} (\nab u_\ep - iu_\ep \N \v)$  converges weakly in $L^2$ to $0$. 

The appropriate convergence of $j_\ep/\N$ is then deduced by \eqref{contjl1}  \eqref{contj} and \eqref{pjj},  and this concludes the proof of Theorem \ref{th2}.
\begin{remark}\label{remark410}
In order to treat the mixed flow or complex Ginzburg-Landau case, the computations are very similar and  shown in Appendix \ref{gauge} below in the gauge case. One should multiply the result of Lemma \ref{divst} by $2\b\v+ 2\a\v^\perp$ instead of $\v^\perp$, and use $\phi = \pp $  or  $\frac{\lambda}{\alpha} \div \v$ (respectively), and $\psi=\beta  \phi- |\v|^2$. 
A supplementary error term in $O(\beta \i \tilde V_\ep \cdot (\v-\bar\v))$ appears, which can be controlled only by an  estimate on $\i |\tilde V_\ep|$ and leads to the extra condition  $\N\gg \llep$.  \end{remark}

\appendix

  \section{Proof of Proposition \ref{estprod}}\label{app1}
As already mentioned, the result is a quantitative  version of the  ``product estimate" of \cite{ss6}. It also needs to be adapted to the case of an infinite domain, which we do by a localization procedure based on a partition of unity.

As in \cite{ss6} we  view things in three  dimensions where the first dimension is time and the last  two are spatial dimensions.
By analogy with a gauge, we introduce the vector-field in three-space
\begin{equation}\label{defAA}
A_\ep= \N(\phi, \v),\end{equation}
whose first  coordinate is $\N \phi$ and whose last two coordinates are those of $\N \v$.
 Equivalently, we can identify $A_\ep$ with a 1-form.
 We also note that \begin{equation}\label{A1bis}
 \frac{1}{\N} \curl A_\ep= (\curl \v, \p_t \v_1 - \p_1 \phi ,\p_t \v_2-\p_2 \phi).\end{equation} 

  We then define  the 2-form  \begin{equation}\label{J2}
J_\ep= d \( \la d u_\ep -iu_\ep A_\ep,iu_\ep \ra +  A_\ep \) ,\end{equation}where $d$ corresponds to the differential in three-space.
\begin{lem} Identifying a spatial vector-field $X$  with the 2-form $X_2 dt\wedge dx_1 + X_1 dt\wedge dx_2$, we have
\begin{equation}\label{Jep} J_\ep =  \tilde \mu_\ep dx_1 \wedge dx_2 + \tilde V_\ep + (1-|u_\ep|^2) \N (\p_t \v-\nab \phi) ,\end{equation}  where $\tilde \mu_\ep $ is as in \eqref{tildemu} and $\tilde V_\ep$ as in \eqref{tV}.
\end{lem}
\begin{proof} By definition,  $J_\ep = J_\ep^t dx_1 \wedge dx_2+ J_\ep^2 dt \wedge dx_1 + J_\ep^1 dt \wedge dx_2$, where for $k=1,2$,
\begin{equation*}
\left\{ \begin{array}{l}
J_\ep^t = \curl (\la \nab u_\ep-iu_\ep \N \v, iu_\ep\ra + \N \v)= \tilde \mu_\ep \\
J_\ep^k=\p_t ( \la \p_k u_\ep -iu_\ep \N \v_k, iu_\ep\ra + \N \v_k)- \p_k (\la \p_t u_\ep -iu_\ep\N \phi, iu_\ep \ra +\N \phi) 
\end{array}\right.\end{equation*}
To obtain the expression of $J_\ep$, it thus suffices to compute 
\begin{align*}
& \p_t (\la \nab u_\ep -iu_\ep \N \v, iu_\ep\ra + \N \v)- \nab(  \la  \p_t u_\ep-iu_\ep\N \phi , iu_\ep\ra + \N  \phi)\\
& = \la \nab u_\ep -iu_\ep \N \v, i\p_t u_\ep \ra + \N \p_t \v + \la \p_t (\nab u_\ep -iu_\ep \N \v) , iu_\ep\ra  -\la \nab \p_t u_\ep, iu_\ep\ra\\ & -\la \p_t u_\ep, i\nab u_\ep\ra 
+ \nab ((|u_\ep|^2-1)\N\phi)
\\ & = 2\la \nab u_\ep -iu_\ep \N \v, i\p_t u_\ep \ra 
+ \N \v \la u_\ep, \p_t u_\ep\ra +(1-|u_\ep|^2) \N \p_t \v- \N \v \la \p_t u_\ep, u_\ep\ra\\ &+ \N \nab ((|u_\ep|^2-1) \phi)
\\& =
\tilde V_\ep-  2 \N\phi\la   \nab u_\ep, u_\ep\ra  +(1-|u_\ep|^2) \N \p_t \v+  \N\nab ((|u_\ep|^2-1) \phi)
\end{align*} where we used \eqref{tV}, and this yields the result. 
\end{proof}

 We work in the space-time slab $[0,\tau ]\times\R^2$. 
We consider $X$ (here a spatial vector field,  depending on time)  and $Y$ (here $Y=e_t$ the unit  vector of the time coordinate)  two  vector fields  on $[0,\tau]\times\R^2$.  
In order to reduce ourselves to the situation where $X$ is locally constant, we use a partition of unity at a small scale: 
let $M_\ep$ be as in \eqref{condM} and let us consider a covering of $[0, \tau]\times \R^2$ by balls of radius $2 M_\ep^{-1/4}$ centered at points of $M_\ep^{-1/4}\mathbb{Z}^3$, and let $\{D_k\}_{k \in \mathbb{N}}$ be an indexation of this sequence of balls and  $\{\chi_k\}_{k\in \mathbb{N}}$  a partition of unity associated to this covering (which we observe has bounded overlap) such that $\sum_{k\in \mathbb{N}} \chi_k= 1$ and  $\|\nab \chi_k\|_{L^\infty} \le M_\ep^{1/4}$. 
For each $k\in \mathbb{N}$, let then $X_k$ be the average of $X$ in $D_k$.    Then, working only in $D_k$, without loss of generality, we can assume that $X_k$ is aligned with the first space coordinate vector $e_1$, with  $(e_t, e_1, e_2)$ forming an orthonormal frame and the coordinates in that frame  being denoted by  $(t, w,\sigma)$. We will assume first that $X_k\neq 0$. 
  Let us define for each $k,\sigma$ the set
 $$\Omega_{k,\sigma}= \{(t,w)|(t, w,\sigma) \in  D_k\},$$  which is a slice of $D_k$ (hence a two-dimensional ball). 
 Let us  write 
 $J_{\ep, k,\sigma}$ for $J_\ep(e_1, e_t)$ restricted to $\Omega_{k,\sigma}$.  In other words, by \eqref{J2}, if $\xi $ is a smooth test-function on $\Omega_{\sigma,k}$, we have
  \begin{equation}\label{Jeps} \int_{\Omega_{k,\sigma}} \xi \wedge J_{\ep, k, \sigma}= - \int_{\Omega_{k,\sigma}}  d\xi \wedge\(  \la d u_\ep -iu_\ep A_\ep, iu_\ep\ra +A_\ep\) \end{equation} 
where $d$ denotes the  differential in the slice $\Omega_{k,\sigma}$.
 
 We  let  $g_k$ be  the constant metric on $\Omega_{k,\sigma}$ defined by  $g_k(e_1,e_1)=\Lambda/|X_k| $,  $ g_k(e_t, e_t )=1/\Lambda$ and $g_k(e_1, e_t )=0$  with $\Lambda \ge 1$ given.
 
We then apply the ball construction  method in each set $\Omega_{k,\sigma}$.  Instead of constructing balls for the flat metric, we construct geodesic balls for the metric associated to $g_k$, i.e.~here, ellipses.  
\begin{lem}\label{52}Let $\Omega_{k,\sigma}\subset \R^2$ be as above and  denote $\Omega^\ep_{k,\sigma}=\{x\in \Omega_{k,\sigma} | \dist (x, \p \Omega_{k,\sigma})>\ep\}$. Assume that 
\begin{multline*}
F_{\ep,k, \sigma}:=\hal \int_{\Omega_{k,\sigma} }\Big( |\p_t u_\ep-iu_\ep\N \phi|^2 +|\p_1 u_\ep -iu_\ep \N \v_1|^2+ \frac{1}{2\ep^2} (1-|u_\ep|^2)^2\\
+\N^2(|\nab\phi|^2 + |\p_t \v|^2)\Big) 
\le M_\ep\end{multline*} with $M_\ep$ as in \eqref{condM}.
 Then if $\ep$ is small enough,  there exists a finite collection of disjoint closed  balls $\{B_i\}$ for the metric $g_k$ of centers $a_i$ and radii $r_i$ such that 
\begin{enumerate}
\item $ \sum_i r_i \le \Lambda M_\ep^{-1}$
\item $ \cup_i B_i$ covers $\{| |u_\ep(x)|-1|\ge \hal \} \cap \Omega_{k,\sigma}^\ep$.
\item Writing $d_i= \deg(u_\ep, \p B_i)$ if $B_i \subset \Omega^\ep_{k,\sigma}$ and $d_i=0$ otherwise,  we have  for each $i$,
\begin{multline}\label{lbmetric}
\hal \int_{B_i}\frac{1}{|X_k|}\Big( \Lambda |X_k|^2 |\p_1  u_\ep -iu_\ep \N \v_1 |^2   +   
\frac{1}{\Lambda} |\p_t u_\ep-iu_\ep \N\phi|^2   \\
 + 2 M_\ep^{-2} \N^2(\Lambda |X_k|^2 |\nab \phi|^2 + \frac{1}{\Lambda} |\p_t \v|^2 )  \Big) \\ \ge \pi   |d_i|\ \( \lep - C\log M_\ep\) .\end{multline}
\item Letting $\mu_{\ep,k,\sigma}= 2\pi \sum_i  d_i \delta_{a_i}$, we have  for any $0 < \gamma\le 1$,  and any $\xi \in C^{0, \gamma}_c(\Omega_{k,\sigma })$, 
$$\left|\int \xi  \wedge J_{\ep,k,\sigma}- \xi  \mu_{\ep,k,\sigma}\right| \le C \|\xi\|_{C^{0,\gamma}}(\Lambda M_\ep^{-2})^\gamma F_{\ep,k, \sigma}  .$$
\end{enumerate}
\end{lem}
\begin{proof}
The first 3 items are a rewriting of \cite[Proposition IV. 2]{ss6}, itself based on the ball construction,   that needs to be adapted to the case of the nonstandard metric.
As in \cite[Proposition IV. 2]{ss6} we start by noting, via the co-area formula, that  there exists $m_\ep$ with $\hal M_\ep^{-1} \le m_\ep\le M_\ep^{-1}$ such that  setting  $\omega:=\{|u_\ep|\le 1- m_\ep\}$ has perimeter for the $g_k$ metric bounded by $C \ep M_\ep^{3}$. We may then  apply \cite[Proposition 4.3]{livre}  outside that set to 
 $v_\ep(t,w)=\frac{u_\ep}{|u_\ep|}(\sqrt{\Lambda } |X_k| w e_1+  \frac{1}{\sqrt{\Lambda}} t e_t )$ and $\tilde A_\ep(t,w)= A_\ep( \sqrt{\Lambda} |X_k| w e_1 + \frac{1}{\sqrt{\Lambda}} t e_t)$ restricted to the slice, with initial radius $r_0= C \ep M_\ep^3$ and final radius $r_1=M_\ep^{-1}$. 
This yields  a collection of disjoint closed balls $\bar B_i$ with sum of radii bounded by $ M_\ep^{-1}$ and such that 
$$\hal\int_{\cup_i \bar B_i\backslash \omega} |\nab v_\ep|^2 + M_\ep^{-2} |\curl \tilde A_\ep|^2
\ge \pi \sum_i |d_i| (\lep- C \log M_\ep).$$
Making the change of variables $x= \sqrt{\Lambda} |X_k| w$ and $ s= \frac{t}{\sqrt{\Lambda}}$, we obtain balls $B_i$, the images of the $\bar B_i$'s by the change of variable, which are geodesic balls for the metric   $g_k$ and whose sum of the radii is bounded by $\Lambda M_\ep^{-1}$ (since $\Lambda \ge 1$);  and inserting \eqref{A1bis} and using \eqref{condM}
we obtain \eqref{lbmetric}.
 We  note that the fact that the domain size also depends on $\ep$ does not create any problem in applying that proof. 
 
 Item 4 is    a consequence of Lemma \ref{lemjac}  adapted to the present setting with differential forms, replacing $\nab u_\ep -iu_\ep A_\ep$ by $du_\ep -iu_\ep A_\ep$ and using again \eqref{A1bis}.
 
 \end{proof}
 
We now proceed as in the proof of \cite{ss6}.  We set $\nu_{\ep,k,\sigma}$ to be the $\mu_{\ep, k ,\sigma}$ of  Lemma \ref{52} (item 4) if the assumption $F_{\ep,k,\sigma}\le M_\ep$ is verified, and $0$ if not.
We note that \begin{equation}\label{proxi}
\|J_{\ep,k,\sigma}- \nu_{\ep,k,\sigma}\|_{(C^{0,1}_c(\Omega_{k,\sigma}))'}\le C \Lambda M_\ep^{-1/2} F_{\ep,k,\sigma} \end{equation} is true in all cases. Indeed, either $F_{\ep, k,\sigma} \le M_\ep$ in which case  the result is true by item 4 of Lemma \ref{52} since $M_\ep^{-1/2} \ge M_\ep^{-2}$, or $\nu_{\ep, k,\sigma}=0$  in which case, for any $\xi \in C^{0,1}_c(\Omega_{k,\sigma})$, starting from \eqref{Jeps} and writing $|\la \nab u_\ep-iA_\ep u_\ep, iu_\ep\ra|\le |\nab u_\ep-iA_\ep u_\ep|+||u_\ep|-1||\nab u_\ep -iA_\ep u_\ep|$, we obtain   with the Cauchy-Schwarz inequality,  using the boundedness of $\Omega_{k,\sigma}$, 
\begin{align*}
\left|\int \xi \wedge J_{\ep,k, \sigma}\right|&  \le C \|\nab \xi\|_{L^\infty}  \int_{\Omega_{k,\sigma}} |\nab u_\ep-iu_\ep A_\ep| + |A_\ep|+|1-|u_\ep|^2| \\& \le C \|\nab \xi\|_{L^\infty}  \sqrt{F_{\ep,k, \sigma}}+ \ep F_{\ep, k,\sigma} \le M_\ep^{-1/2}F_{\ep, k,\sigma} ,\end{align*} for $\ep $ small enough.
But since $F_{\ep, k,\sigma} \ge M_\ep$,  we have $\sqrt{F_{\ep,k, \sigma}} +\ep F_{\ep, k,\sigma}  \le 2 M_\ep ^{-1/2} F_{\ep,k, \sigma} $ and thus we find that  \eqref{proxi} holds as well. 
By \eqref{lbmetric}, we also have that 
\begin{equation}\label{bornnu}
\int |\nu_{\ep, k,\sigma}| \le \frac{C}{\lep} \max(\frac{1}{\Lambda |X_k|}, \Lambda |X_k|)F_{\ep, k,\sigma} .\end{equation}
Next, we choose $\eta$ a function depending only on time, vanishing at $0$ and $\tau$, such that $\eta=1$ in $[M_\ep^{-1/4}, \tau- M_\ep^{-1/4}] $ if that interval is not empty, and affine otherwise. By construction \begin{equation}\label{cheat}
\|\eta\|_{C^{0,1}} \le CM_\ep^{1/4}, \quad \|\eta \chi_k\|_{C^{0,1}}\le C M_\ep^{1/4}.\end{equation}

We may now write that 
\begin{multline}\label{binfl}
\int_{\Omega_{k,\sigma}} \frac{\eta \chi_k }{|X_k|}\Big( \Lambda |X_k|^2 |\p_1  u_\ep - iu_\ep \N \v_1|^2  +   
\frac{1}{\Lambda} |\p_t u_\ep-  iu_\ep\N\phi|^2   \\+ 2 M_\ep^{-2} \N^2(\Lambda |X_k|^2 |\nab \phi|^2 + \frac{1}{\Lambda} |\p_t \v|^2 )   \Big) \\ 
\ge   (\lep -C \log M_\ep) \int_{\Omega_{k,\sigma}} (\eta\chi_k -C \Lambda^2 M_\ep^{1/4-1}) \nu_{\ep,k,\sigma} .\end{multline}
Indeed, if we are in a slice where $\nu_{\ep,k,\sigma}=0$, this is trivially true. If not, we apply \eqref{lbmetric} and obtain
\begin{multline*}
\frac12 \int_{\Omega_{k,\sigma}^\ep \cap (\cup_iB_i)} \frac{\eta \chi_k }{|X_k|}\Big( \Lambda |X_k|^2 |\p_1  u_\ep - iu_\ep \N \v_1|^2  +   
\frac{1}{\Lambda} |\p_t u_\ep-  iu_\ep\N\phi|^2  \\+ 2 M_\ep^{-2} \N^2(\Lambda |X_k|^2 |\nab \phi|^2 + \frac{1}{\Lambda} |\p_t \v|^2 )  \Big) \\
 \geq 2\pi  \sum_i |d_i|\min_{B_i} (\eta \chi_k) (\lep - C \log M_\ep)).\end{multline*}
Inserting then $\min_{B_i} \eta\chi_k\ge (\eta\chi_k)    (a_i)  -C \Lambda r_i \|\eta \chi_k\|_{C^{0,1}}$ and  \eqref{cheat}, and using item 1 of Lemma  \ref{52}, yields \eqref{binfl}.

Next, we  integrate \eqref{binfl} with respect to $\sigma$, and combine with \eqref{proxi} and \eqref{bornnu} to obtain
\begin{align}
\nonumber   & \int_{D_k } \frac{\eta\chi_k}{|X_k|}\Big( \Lambda |X_k|^2 |\p_1  u_\ep - iu_\ep \N \v_1|^2   +   
\frac{1}{\Lambda} |\p_t u_\ep-iu_\ep\N\phi |^2 
\\ \nonumber &  + 2 M_\ep^{-2} \N^2(\Lambda |X_k|^2 |\nab \phi|^2 + \frac{1}{\Lambda} |\p_t \v|^2 )   \Big) 
\\ 
\nonumber  & \ge (\lep- C \log M_\ep)  \int_{\R^2 \times [0,\tau]} \eta\chi_k  J_{\ep,k,\sigma}   \\ \label{final} &-
C \max(\frac{1}{\Lambda |X_k|}, \Lambda |X_k|)\Lambda^2 M_\ep^{1/4-1} F_{\ep, k} 
+ C M_\ep^{1/4}\lep \Lambda M_\ep^{-1/2} F_{\ep,k},\end{align} 
with $=F_{\ep, k}:= \int F_{\ep, k,\sigma} d\sigma.$
  Moreover, by \eqref{Jep},  (since we assumed $X_k$ is along the direction $e_1$) we have
 \begin{equation*}J_{\ep,k,\sigma}= J_\ep (e_1, e_t)=  \frac{1}{|X_k|}\(\tilde V_\ep \cdot X_k + \N(1-|u_\ep|^2)(\p_t \v -\nab \phi) \cdot X_k   \) 
 \end{equation*} 
 so we may bound
\begin{multline*}
\int_{D_k} \eta\chi_k \N(1-|u_\ep|^2) (  \p_t \v-\nab \phi)  \cdot X_k \\ \le C \|X\|_{L^\infty} \N\|1-|u_\ep|^2\|_{L^2(D_k)} \|\p_t \v - \nab \phi \|_{L^2(D_k)}\\
 \le C \|X\|_{L^\infty}   \N \ep F_{\ep,k}.\end{multline*}
 Inserting into \eqref{final} and multiplying by $|X_k|$, we may write 
 \begin{align*}
& \int_{D_k} \eta\chi_k \Big( \Lambda |X_k|^2 |\p_x  u_\ep - iu_\ep \N \v|^2  +   
\frac{1}{\Lambda} \left|\p_t u_\ep - iu_\ep\N\phi \right|^2 
\\  \nonumber & + 2 M_\ep^{-2} \N^2(\Lambda |X_k|^2 |\nab \phi|^2 + \frac{1}{\Lambda} |\p_t \v|^2 )   \Big) 
\\
&  \ge( \lep- C \log M_\ep)  \int_{D_k} \eta\chi_k \tilde V_\ep \cdot X_k 
\\
& - \max(\frac{1}{\Lambda }, \Lambda |X_k|^2)  \Lambda^2 M_\ep^{-3/4} +(\Lambda M_\ep^{-1/4} \lep +\ep \N) \|X\|_{L^\infty}F_{\ep, k},\end{align*}
and we note that this holds as well if $X_k=0$. 
We may next replace $X_k$ by $X$ in the left-hand side and the $\int \tilde V_\ep \cdot X_k$ term, and using that $|X-X_k|\le  C M_\ep^{-1/4} \|X\|_{C^{0,1}}$ in $D_k$, the error thus created is bounded above by
$$ M_\ep^{-1/4}  \Lambda \lep \|X\|_{L^\infty} (1+\|X\|_{C^{0,1}} )  F_{\ep, k} 
$$
where we have used that by definition of $\tilde V_\ep$, we have $\int_{D_k} |\tilde V_\ep|\le F_{\ep, k}$.
We may thus absorb this error into the others and write (since $\Lambda \ge 1$)
\begin{align*}
& \int_{D_k} \eta\chi_k \Big( \Lambda |X|^2 |\p_x  u_\ep - iu_\ep \N \v|^2  +   
\frac{1}{\Lambda} \left|\p_t u_\ep - iu_\ep\N\phi \right|^2 \Big)\\  \nonumber & + 2 M_\ep^{-2} \N^2(\Lambda |X|^2 |\nab \phi|^2 + \frac{1}{\Lambda} |\p_t \v|^2 )   \Big) \\
&  \ge (\lep- C \log M_\ep)  \int_{D_k} \eta\chi_k \tilde V_\ep \cdot X
\\ & - C  \|X\|_{L^\infty}  \(\Lambda^3   (1+\|X\|_{C^{0,1}}) M_\ep^{-1/8}  +\ep \N\) F_{\ep, k}  .\end{align*}
Summing over $k$,  using that $\sum_k \chi_k=1$ in $[0,\tau ]\times \R^2$, the finite overlap of the covering, the fact that $\p_t \v $ and $\nab \phi \in L^2([0, \tau] , L^2 (\R^2))$ by Lemma \ref{lemv} and \eqref{contrmod}, we are led to 
\begin{multline*}
 \int_0^\tau \i\eta \( \Lambda |X|^2 |\p_x  u_\ep - iu_\ep \N \v|^2  +   
\frac{1}{\Lambda} \left|\p_t u_\ep - iu_\ep\phi \right|^2 \) 
 \\ \ge( \lep- C \log M_\ep)  \int_0^\tau \i \eta \tilde V_\ep \cdot X
\\  - C \|X\|_{L^\infty}  \(\Lambda^3   (1+\|X\|_{C^{0,1}}) M_\ep^{-1/8}  +\ep \N\) \(F_{\ep} + \N^2 \)
  .\end{multline*}
Moreover, by choice of $\eta$ and definition of $\tilde V_\ep$,   we have 
\begin{align*}
\lefteqn{
\left|\int_0^\tau \i  (1-\eta) X\cdot \tV \right|
} \qquad & \\
& \le 2 \|X\|_{L^\infty}  \int_{ [0,M_\ep^{-1/4}]\cup [  \tau-M_\ep^{-1/4},\tau]}\i|\p_t u_\ep-iu_\ep \N \phi | |\nab u_\ep -iu_\ep \N \v| \\ 
& \le 
C \|X\|_{L^\infty} \sqrt{F_\ep\Big(\sup_{t\in [0,\tau]} \E(t)+o(1)\Big) }  M_\ep^{-1/8}   .
\end{align*}
where we have used the Cauchy-Schwarz inequality and  \eqref{contmod}. Combining the last two relations, we deduce the desired result.

\section{
Existence and uniqueness for \eqref{limgm}}\label{app2}
We recall that the definitions of H\"older spaces that we use are at the beginning  of Section \ref{sec2}.
For the sake of generality we study \eqref{limgm} with arbitrary $\a>0$ and $\b\ge 0$ such that $\a^2+\b^2=1$, and $\lambda>0$.
We denote by $\omega=\curl \v$ and $\d= \div \v$. We note that  if $\v$ solves \eqref{limg} then $(\om, \d)$ solve
\begin{equation}\label{syst}
\left\{
\begin{array}{ll}
\p_t \om=\displaystyle 2 \div (\beta \v \om+ \a \v^\perp \om)\\ [2mm]
\p_t \d = \displaystyle \frac{\l}{\a} \Delta \d+ 2\div (\beta\v^\perp \om-\a\v\om).
\end{array}\right. 
\end{equation}

\begin{theo} \label{th-exi}
Assume $\lambda >0$, $\a>0$ and $\a^2+\b^2=1$.
Assume  $\v(0)$ is such that  $\frac{1}{2\pi}\om(0)$ is a    probability measure which also belongs to $C^{\g}(\R^2)$, and   $\d(0) \in C^\g \cap L^p(\R^2)$, for some $0<\g\le 1$ and some $1\le p <2$.   Then there exists a  unique local in time solution to \eqref{limg} on  some interval $[0,T]$, $T>0$, which is such that $\v \in L^\infty([0,T], C^{1+\g'} )$ for any $\g'<\g$. Moreover, we have $\v(t)-\v(0)\in L^\infty([0,T],L^2(\R^2))$, $\p_t \v \in L^2([0,T], L^2(\R^2))$, $\d\in L^\infty([0,T], L^p(\R^2))$ and $\frac{1}{2\pi}\om(t) $ is a probability measure for every $t \in[0,T]$.
\end{theo}
We start with some preliminary results.
\begin{lem}\label{lem1}
Let $v$ be a vector field in $L^\infty([0,T],C^{1,\gamma}(\R^2))$,  $\om_0 \in C^{\g}$, $0<\g\le1$, $f \in L^\infty(\R, C^{\g})$  then the equation 
\begin{equation}\label{transport}
\left\{   \begin{array}{l}
\p_t \omega= \div (v \omega) +f \\
\omega(0)=\omega_0\end{array}\right.\end{equation} has a unique solution, and it holds that for some $C>0, C_0>1$,
\begin{equation}\label{estimeel1}
\|\om(t) \|_{L^1(\R^2) }\le \|\om_0\|_{L^1(\R^2)}+       \exp\( C \int_0^t\|\nab v(s)\|_{L^\infty} \, ds\) \int_0^t \|f(s)\|_{L^\infty}\, ds
\end{equation}
 and  for any $-1<\sigma\le \g$,
\begin{multline}\label{esttransport}
\|\om(t) \|_{C^{\sigma} }\le C_0 \(\| \om_0 \|_{ C^{\sigma}} + \int_0^t \|f(s)\|_{C^\sigma}\, ds    + \|\om_0\|_{L^\infty} \int_0^t \|\div v(s)\|_{C^\sigma}\, ds \)\\ \times  \exp\( C \int_0^t\|\nab v(s)\|_{L^\infty} \, ds\) .\end{multline} \end{lem}
\begin{proof}One may rewrite the equation as 
$$\p_t \om= v\cdot \nab \om + \om \div v+f.$$
Then, by  propagation along characteristics, we obtain  first
$$\|\om(t)\|_{L^1} \le \|\om_0\|_{L^1} + \( \int_0^t \|f(s)\|_{L^\infty} \, ds \) \exp\int_0^t \|\div v(s)\|_{L^\infty} \, ds,$$ second, 
$$\|\om(t)\|_{L^\infty} \le \( \|\om_0\|_{L^\infty} + \int_0^t \|f(s)\|_{L^\infty} \, ds\) \exp\int_0^t \|\div v(s)\|_{L^\infty} \, ds,$$
and third  \eqref{esttransport}  for $\sigma=\gamma$ follows by a Gronwall argument.
For general $\sigma\le \g$, one can proceed as in  \cite[Chap. 3]{bcd}.

\end{proof}

The next lemma  about the regularizing effect of the heat equation can be found in \cite[Proposition 2.1]{chemin2} (applied with $p=\infty$, $\ro=\infty$ and noting that $B^{s}_\infty$ is the same as $C^s$).
\begin{lem}\label{lem2}
If $g\in L^\infty([0,T], C^{-1+\g}\cap H^{-1}(\R^2))$ and $u_0\in C^{\g}(\R^2)$,    then the equation 
\begin{equation}\label{heat}
\left\{   \begin{array}{l}
\p_t u= \nu \Delta u +g \\
u(0)=u_0\end{array}\right.\end{equation}
has a unique solution which is in $L^\infty([0,T], C^{\g}\cap L^2(\R^2))$  and $L^2([0,T], H^1(\R^2))$ and for any $\sigma \le \g$, 
\begin{equation}\label{estheat}
\|u\|_{L^\infty([0,T],C^{\sigma} )} \le C_0\(   \|u_0\|_{C^\sigma} +  \sqrt{\frac{T }{\nu} }  \|g\|_{L^\infty([0,T],C^{-1+\sigma}) }  \).\end{equation}
\end{lem}

We note that the fact that $u\in L^2([0,T], H^1(\R^2))$ comes from the fact that $g\in L^2([0,T], H^{-1}(\R^2))$ and the regularizing effect of the heat equation.

\begin{lem}\label{lemheqt2}Assume $u$ solves in $[0,T]$
\begin{equation}\label{heat}
\left\{   \begin{array}{l}
\p_t u=  \Delta u +\div f \\
u(0)=u_0\end{array}\right.\end{equation}
then, if $q\le p$ and $ \frac1p-\frac1q+\frac12\ge 0$, we have for any $t\in [0,T]$,
$$\|u(t)\|_{ L^p(\R^2)}\le \|u_0\|_{L^p}+ C_{p,q}t^{\frac1p-\frac1q+\frac12}\|f\|_{L^\infty([0,T],L^q(\R^2))}.$$
\end{lem}
\begin{proof}We follow \cite[Lemma 2.3]{du}.
Using Duhamel's formula, we may write
$$u(t,\cdot )= G(t,\cdot)*u_0 + w( t,\cdot)$$ 
with $$w( t,\cdot)= \int_0^t\int \Gamma_{t-s}(x-y) f(s,y)\, ds \, dy
 $$ where $G(t,x)= \frac{e^{\frac{-|x|^2}{4t} }}{4\pi t}$
 and  $\Gamma_t(x)=-\p_x G(t,x)=  \frac{x}{8\pi t^2}e^{\frac{-|x|^2}{4t}}$.
 By Young's inequality for convolutions, we have \begin{multline*}
 \|u(t)\|_{L^p(\R^2)} \le  \|u_0\|_{L^p(\R^2)}\|G(t,\cdot)\|_{L^1(\R^2)} + \|w(t)\|_{L^p(\R^2)}\\ \le \|u_0\|_{L^p(\R^2)} + \|w(t)\|_{L^p(\R^2)} .\end{multline*}
   We turn to studying  $w$.
   We may write with H\"older's inequality, 
   $$|w(t,x)|\le \int_0^t\( \int \Gamma_{t-s}( x-y)^{\frac{q'}{2}}\, dy\)^{\frac{1}{q'}}\( \int \Gamma_{t-s}(x-y)^{\frac{q}{2}}|f(s,y)|^q\, dy\)^{\frac1q}\, ds$$
   with $q'$ such that $1/q+1/q'=1$, and hence with H\"older's inequality again, if $q\le p$,
   \begin{multline*}
   \|w(t)\|_{L^p(\R^2)}\\
   \le \int_0^t\( \int \Gamma_{t-s}^{\frac{q'}{2}}\)^{\frac{1}{q'}} \( \int \(\int  \Gamma_{t-s}(x-y)^{\frac{q}{2}}|f(s,y)|^q\, dy\)^{\frac{p}{q}} dx\)^{\frac1p} ds\\
  \le \int_0^t ( \int \Gamma_{t-s}^{\frac{q'}{2}}\, dy)^{\frac{1}{q'}} 
   \(\int ( \int \Gamma_{t-s}( x-y)^{\frac{p}{2}} |f(s,y)|^q\, dy )
   ( \int |f(s,y)|^q \, dy )^{\frac{p}{q}-1} dx\)^{\frac1p} ds
   \\ \le \int_0^t \( \int \Gamma_{t-s}^{\frac{q'}{2}}\)^{\frac{1}{q'}} 
   \( \int |f(s,y)|^q \, dy \)^{\frac{1}{q}-\frac1p} \( \int \Gamma_{t-s}^{\frac{p}{2}}  \int  |f(s,y)|^q\, dy \)
  ^{\frac1p}\, ds
\\= \int_0^t \( \int \Gamma_{t-s}^{\frac{q'}{2}}\)^{\frac{1}{q'}} 
   \( \int |f(s,y)|^q \, dy \)^{\frac{1}{q}}\( \int \Gamma_{t-s}^{\frac{p}{2}}  \)^{\frac1p}
  \, ds
     \end{multline*}   where for the passage from the second to third line we used Young's inequality for convolutions.
   We may thus write
  \begin{equation*}
   \|w(t)\|_{L^p(\R^2)}\le \|f\|_{L^\infty([0,T], L^q(\R^2))} \int_0^t \|\Gamma_s\|_{L^{q'/2}(\R^2)}^{1/2} \|\Gamma_s\|_{L^{p/2}(\R^2)}^{\hal}\, ds.\end{equation*}
  But computing explicitly, we find $\|\Gamma_s\|_{L^r(\R^2)}=  C_r s^{-3/2 +1/r}$ and so we deduce that \begin{equation*}
   \|w(t)\|_{L^p(\R^2)}\le C_{p,q} \|f\|_{L^\infty([0,T], L^q(\R^2))} t^{\frac12+\frac1p-\frac1q},\end{equation*}
   and the result follows.
   \end{proof}
   
  Finally, we will need the following potential estimates:
  \begin{lem}\label{lempt}
  Let  $u\in C^{\gamma} \cap L^p(\R^2)$ with $0<\gamma\le 1$ and $1\le p<2$. Then $\nab \Delta^{-1} u$, where $\Delta^{-1}$ is meant as the convolution with $-\frac{1}{2\pi} \log$, is well defined and 
\begin{equation}\label{lemb3} \|\nab \Delta^{-1} u\|_{C^{1,\gamma}(\R^2)} \le C_1(\|u\|_{C^{\gamma}(\R^2)} + \|u\|_{L^p(\R^2)}).\end{equation}
  \end{lem}
  \begin{proof}Setting  $v= \nab \Delta^{-1} u$ we may write
  $$v(x)= -\frac{1}{2\pi}\int \frac{x-y}{|x-y|^2} u(y)\, dy$$ and one may check that this integral is convergent with 
  \begin{multline}\label{els}
  |v(x)|\le  C\|u\|_{L^\infty(\R^2)}  \int_{|y-x|\le 1} \frac{1}{|x-y|} \, dy + \|u\|_{L^p} \( \int_{|y-x|\ge 1} \frac{1}{|x-y|^{p'}} \)^{1/p'} \\
   \le C( \|u\|_{L^\infty(\R^2)} + \|u\|_{L^p(\R^2)} )\end{multline}
   with $1/p+1/p'=1$.
Let then $w$ be such that $v=\nab w$, and thus $\Delta w=u$.
  For any $x\in \R^2$, by Schauder estimates for elliptic equations and \eqref{els}, we have  \begin{align*} \|\nab w\|_{C^{1,\gamma}(B(x,1))} & \le  C (\| \nab w\|_{L^\infty(B(x,2))} + \|u\|_{C^{\gamma}(B(x,2))})\\
  &   \le  C( \|u\|_{L^p(\R^2)} +\|u\|_{C^{\gamma}(\R^2)}) .
\end{align*}
 Taking the sup over $x$ yields the result.
  \end{proof}
  %

\smallskip 

We now turn to the proof of Theorem \ref{th-exi} for which we set up an iterative scheme.

 Let $\v_0=\v(0)$ and $\om_0= \curl \v_0, \d_0=\div \v_0$.
Given $\v_n$ and $\om_n= \curl \v_n, \d_n= \div \v_n$, we want to solve
\begin{equation}\label{iteration}\left\{\begin{array}{l}
\p_t \om_{n+1} = 2 \div ((\b  \v_n+ \a\v_n^\perp) \om_{n+1} )\\ [2mm]
\p_t \d_{n+1} = \frac{\l}{\a} \Delta \d_{n+1} +2 \div ((\b \v_n^\perp -\a\v_n) \om_n)\\
\omega_{n+1}(0)= \om(0)\\
\d_{n+1}(0)= \d(0).\end{array}\right.\end{equation}
Let $$
t_n=\sup\{ t, K_n(t)   \le  2 C_1 C_0 K_0\}
$$  where 
$$K_n(t) = \|\om_{n}(t) \|_{C^{\g}} + \| \om_n(t) \|_{L^1} + \|\d_n(t)\|_{C^\g}+ \|\d_n(t)\|_{L^p} + \|\v_n(t)\|_{C^{1,\g} } $$
$$K_0=  \|\omega(0)\|_{L^1(\R^2)} + \|\omega(0)\|_{C^{\g} } + \|\d(0)\|_{L^p}+\|\d(0)\|_{C^\g},$$ and $C_0$ is the maximum  of the constants in \eqref{esttransport}, \eqref{estheat}, and $C_1$ the maximum of the constant in  \eqref{lemb3} and $1$.

Let us show by induction that \eqref{iteration} is solvable  and there exists $ T_0>0$ independent of $n$ such that $t_n>T_0$ for all $n$.

For $n=0$, this statement is true by assumption. Assume it is true for $n$.  
Then in view of Lemma \ref{lem1}, we have that the equation \eqref{iteration} for $n+1$ has a solution $\om_{n+1}\in L^\infty([0,T],C^{\g}\cap L^1)$,  with, if $T \le T_0< t_n$,
\begin{align}\label{estsolsyst}
\lefteqn{
\|\om_{n+1}  \|_{L^\infty([0,T],C^{\g}) }
} \qquad & \\
& \le
 C_0\Big(\| \om(0) \|_{ C^\g } + \|\om(0)\|_{L^\infty}  \int_0^T  \|\v_n\|_{C^{1,\g}} (s)\, ds   \Big)   \exp\(  C\int_0^T   \|\v_n\|_{C^{1}} (s)\, ds\) \notag
\\
& \le 
 C_0\(\| \om(0) \|_{ C^\g } + \|\om(0)\|_{L^\infty}  2 C_1 C_0T K_0   \) \exp \(2 C_1C_0 C T  K_0\). \notag
\end{align}
 Also it is straightforward to see  by integrating the equation and since $\omega_{n+1} \in L^1$ by \eqref{estimeel1}  that $\frac{1}{2\pi}\om_{n+1}$  remains a  probability measure.

Similarly, Lemma \ref{lem2} yields the existence of a solution  $\d_{n+1} \in L^\infty([0,T], C^\g)$ with, if $T\le T_0<t_n,$
\begin{multline}\label{estsolsyst2}
\|\d_{n+1}\|_{L^\infty([0,T] ,C^\g) }\le C_0 ( \|\d(0)\|_{C^\g} +C \sqrt{T} \| \v_n\om_n\|_{C^\g}   )\\ \le C_0(\|d(0)\|_{C^\g} + C\sqrt{T} C_1 C_0^3 K_0^2 ) ,
\end{multline}where $C$ depends only on $\alpha$ and $\lambda$; and by Lemma \ref{lemheqt2} applied with $q=1$,
\begin{multline}\label{B10}\|\d_{n+1}\|_{L^\infty([0,T],L^p)} \le \|\d(0)\|_{L^p} + CT^{\frac{1}{p}-\frac{1}{2}}\|\v_{n} \omega_{n+1}   \|_{L^\infty([0,T], L^1)}\\
\le \|\d(0)\|_{L^p}+CT^{\frac{1}{p}-\frac{1}{2}} C_0 C_1 K_0.
\end{multline}
We then let \begin{equation}\label{defvnp1}\v_{n+1}= \nab \Delta^{-1}\d_{n+1}- \np \Delta^{-1}\om_{n+1}.\end{equation} By Lemma \ref{lempt}, this is well-defined and
 \begin{equation}\label{regul}
\|\v_{n+1}\|_{C^{1,\g}} \le   C_1 \( \|\d_{n+1}\|_{C^\g} +\|d_{n+1}\|_{L^p}+\|\om_{n+1}\|_{C^{\g}}+\|\om_{n+1}\|_{L^1}\).
\end{equation}In view of  \eqref{estsolsyst}, \eqref{estsolsyst2}, \eqref{B10}, and \eqref{regul}  we then deduce that if $T_0$ is chosen small enough (depending on  $K_0$ and the various constants), then we will find that $t_{n+1} \ge T_0$. 
The desired result is thus proved by induction.

Let us now show that $\{\om_n\}$ and $\{\d_n\}$ are Cauchy sequences in $C^{-1+\g}$. 
By subtracting the equations for $n$ and $n+1$ we have 
\begin{equation}\label{iteration2}\left\{\begin{array}{l}
\p_t(\om_{n+1} -\om_n)= 2 \div \Big((\b\v_{n}+\a\v_{n}^\perp)(\om_{n+1}-\om_n) \\
\quad + (\b(\v_{n}-\v_{n-1})+\a(\v_{n}-\v_{n-1})^\perp) \om_n\Big)  \\ [3mm]
\p_t (\d_{n+1}-\d_n) = \frac{\l}{\a} \Delta (\d_{n+1}-\d_n) \\+ 2\div \( (\b\v_{n}^\perp-\a\v_{n})(\om_{n}-\om_{n-1}) + (\b(\v_{n}-\v_{n-1})^\perp-\a(\v_{n}-\v_{n-1})) \om_n    \).\end{array}\right.\end{equation} 
Since $\om_n(0)= \om(0)$ for all $n$, applying \eqref{esttransport} with $\sigma= -1+ \g$,  and the result of the previous step,  we have
\begin{align*}
& \|\om_{n+1}-\om_n\|_{L^\infty([0,T],C^{\sigma})} \\
& \le e^{C\int_0^T \|\nab\v_n\|_{L^\infty} \, ds}  \int_0^T \|\div ((\beta (\v_n-\v_{n-1})+ \a (\v_n-\v_{n-1})^\perp\om_n)\|_{C^\sigma} \, ds \\
&
\le  Ce^{CT}T \|(\v_n-\v_{n-1})\om_n\|_{L^\infty([0,T],C^{1+\sigma})}\\
& \le  C e^{CT}T\( \|\om_n-\om_{n-1}\|_{L^\infty([0,T], C^\sigma)} +\|\d_n-\d_{n-1}\|_{L^\infty([0,T],C^\sigma)}\).
 \end{align*} 
Similarly, in view of \eqref{estheat},
\begin{multline*}
 \|\d_{n+1}-\d_n\|_{L^\infty([0,T],C^{\sigma}) }\le   C\sqrt{T}   \Big(\|\v_n (\om_{n}-\om_{n-1})\|_{L^\infty([0,T], C^\sigma)}\\ + \|(\v_n-\v_{n-1}) \om_n\|_{L^\infty([0,T], C^\sigma)} \Big)
\\ \le  C\sqrt{T}  \(  \|\om_n-\om_{n-1}\|_{L^\infty([0,T],C^\sigma)} + \|\d_n-\d_{n-1}\|_{L^\infty([0,T], C^\sigma)}\),
 \end{multline*} 
 if $\sigma <-1+\g$.
We deduce by standard arguments that  $\{\om_n\}$
 and $\{\d_n\}$ form a Cauchy sequence   in $L^\infty([0,T], C^\sigma)$ if $T$ is taken small enough.
 By interpolation $\{\v_n\}$ is a  Cauchy sequence in $C^{1,\g'}$, for $\g'<\g$ (resp. $\{\om_n\}$ in $C^{\g'} $ and $\{\d_n \}$ in $C^{\g'}$). The limits $\v$, $\om, \d$ will obviously solve \eqref{syst}, $\omega$ will be in $L^\infty([0,T],L^1(\R^2))$, $\d $ in $L^\infty([0,T],L^p(\R^2))\cap L^2([0,T],H^1(\R^2))$  by Lemmas \ref{lem2} and \ref{lemheqt2}, and (from \eqref{defvnp1})
 \begin{equation}\label{vv}
 \v=\nab \Delta^{-1} \d  - \np \Delta^{-1} \om ,\end{equation}
  with  $\v \in L^\infty([0,T], C^{1,\g'})$. 
    Taking the time derivative of \eqref{vv} and using \eqref{syst},  it is standard to deduce that $\v$ must solve \eqref{limgm}. By integrating the equation, using that $\v$ is bounded and $\omega \in L^1(\R^2)$,  we also have that $\frac{1}{2\pi}\omega $ remains a probability density. 
  \\
 
 Next, we prove that $\v-\v(0)$ remains in $L^2(\R^2)$. 
 Using \eqref{limgm} and integration by parts, for $\zeta(x)= e^{-\eta |x|}$ with $0<\eta<1$ we compute
 \begin{multline*}\frac{d}{dt}\i \zeta |\v(t)-\v(0)|^2 
 = 2\i \zeta (\v-\v(0))\cdot \( \frac{\lambda}{\alpha} \nab \div \v+ 2\beta \v^\perp\curl \v - 2\alpha \v \curl \v\)
 \\
 = - \frac{2\lambda}{\alpha} \i \zeta (\div \v)^2  + \frac{2\lambda}{\alpha } \i \zeta (\div \v)(\div \v(0)) -\frac{2\lambda}{\alpha } \i \nab \zeta \cdot (\v-\v(0)) \div \v \\- 4\alpha \i\zeta  |\v|^2 \omega +\zeta \v\cdot \v(0) \omega -4 \beta \i \zeta \v^\perp \cdot \v(0) \omega\end{multline*}
 and  using Young's inequality, the fact that $|\nab \zeta|\le |\zeta|$ and that $\div \v\in L^2$, we find 
 $$\frac{d}{dt}\i \zeta |\v-\v(0)|^2 
\le C\( \i\zeta |\div \v(0)|^2 +  \i \zeta |\v(0)|^2 \omega +\i \zeta |\v-\v(0)|^2\)$$ with a constant $C$ depending on $\alpha, \beta $ and $\lambda$.
Using Gronwall's lemma, since $\i|\omega|$ is uniformly bounded and $\v(0)$ is bounded, we deduce that $\i \zeta |\v(t)-\v(0)|^2 
\le C_t $, with $C_t$ independent from $\eta$. Letting then $\eta \to 0$,  the desired claim  follows.
 \\
 
 To prove uniqueness, it suffices to apply the same idea of weak-strong uniqueness as used in the proof of the main results: if $\v_1$ and $\v_2$ are two  solutions of \eqref{limgm}, we  introduce the ``relative" energy
 $E(t)= \hal \i |\v_1(t)-\v_2(t)|^2$, and the relative stress-energy tensor 
 $$T=(\v_1-\v_2)\otimes (\v_1-\v_2) -\hal |\v_1-\v_2|^2 I.$$
By direct computation we have 
\begin{multline}\label{nrj}
\frac{d}{dt}E= \i (\v_1-\v_2)\Big(\frac{\lambda}{\a}\nab (\div \v_1-\div \v_2) +2\b (\v_1^\perp\curl \v_1 -\v_2^\perp \curl \v_2 )\\  -2\a (\v_1\curl\v_1-\v_2\curl \v_2)\Big)\\
= \i -\frac{\lambda}{\a} |\div (\v_1-\v_2)|^2  + 2\b (\v_1-\v_2)\cdot \v_2^\perp  \curl (\v_1-\v_2)\\ +
\i -2\a |\v_1-\v_2|^2 \curl \v_1 -2\a(\v_1-\v_2)\cdot \v_2 \curl (\v_1-\v_2).
\end{multline}
and 
\begin{equation}\label{ddt}
\div T=(\v_1-\v_2)\div \v_1+(\v_1-\v_2)^\perp \curl (\v_1-\v_2).\end{equation}
Multiplying \eqref{ddt} by $2\b \v_2+2\a\v_2^\perp$ and rearranging terms, we find
\begin{multline*}
\i (2\b\v_2+2\a\v_2^\perp) \cdot \div T=\i (2\b \v_2+2\a\v_2^\perp) \cdot  (\v_1-\v_2) \div (\v_1-\v_2)\\+\i (2\a \v_2-2\b \v_2^\perp) \cdot  (\v_1-\v_2) \curl (\v_1-\v_2),\end{multline*}
and inserting into \eqref{nrj} and using one integration by parts,  we obtain
\begin{multline*}
\frac{d}{dt} E\le  -\i \frac{\lambda}{\a} |\div (\v_1-\v_2)|^2 +\i  T : \nab  (2\b\v_2+2\a \v_2^\perp)\\
+ \i (2\b \v_2+2\a \v_2^\perp) \cdot  (\v_1-\v_2) \div (\v_1-\v_2).\end{multline*}
We may next use the boundedness of $\v_2$ and $\nab \v_2$, bound $|T|$ by $2E$, and use the Cauchy-Schwarz inequality to absorb the last term into the first negative term plus  a constant times $E$, to finally obtain a differential inequality of the form $\frac{d}{dt} E \le C E$. The integrations by parts can easily be justified by using cut-off functions like $\chi_R$ and taking the limit, using the fact that $\v_i-\v(0)\in L^2(\R^2)$ and $\div (\v_1-\v_2) \in L^2(\R^2)$.  We then conclude to uniqueness by Gronwall's lemma.

\section{The gauge case}\label{gauge}
In this appendix, we present in a formal manner the quantities that should be introduced and the computations which need to be followed to obtain the limiting dynamical laws in the two-dimensional case with gauge mentioned in Section \ref{sec:gauge} of the introduction. We use the notation of that section, in particular the starting point is \eqref{glg}.
 In the rest of this appendix,  we will not write down negligible terms, such as terms involving $1-|u_\ep|^2$.   We will however present the computations in the general case of the mixed flow \eqref{glg}, which can thus be used as a model for studying the mixed flow in the case without gauge. 

 We start by introducing some notation: we set    $$\p_\Phi:= \p_t - i \Phi_\ep$$ and 
 \begin{equation}\label{C1}\phi:=\begin{cases}\pp & \text{in cases leading to } \eqref{le1} \\
 \frac{\lambda}{\a} \div \v & \text{in cases leading to } \ \eqref{le2}.\end{cases}\end{equation} 
 We will also denote $\Phi'= \Phi_\ep+\N\phi$ and $\p_{\Phi'}=\p_t -i(\Phi_\ep-\N\phi)$, $A'= A_\ep+\N \v$ (dropping the $\ep$) and $h_\ep= \curl A_\ep$.
  We define  the velocity as  in (2.43) and Lemma 2.12 in \cite{tice}:
\begin{multline}\label{defVg}
V_\ep := - \nab \la iu_\ep,  \p_{\Phi_\ep} u_\ep\ra + \p_t \la iu_\ep,  \nab_{A_\ep} u_\ep\ra - E_\ep \\
=  2\la i \p_{\Phi_\ep} u_\ep,  \nab_{A_\ep} u_\ep \ra +(  |u_\ep|^2-1  ) E_\ep 
\end{multline} 
thus 
\begin{equation}\label{dtjg}
\p_t j_\ep= \nab \la iu_\ep, \p_{\Phi_\ep} u_\ep \ra + V_\ep +E_\ep.\end{equation} The modulated velocity is then defined as 
\begin{equation}\label{tvvv}
\tilde V_\ep=   2\la i \p_{\Phi'} u_\ep,  \nab_{A'} u_\ep\ra +(  |u_\ep|^2-1  ) E_\ep .
\end{equation}
 The modulated energy is defined as 
 \begin{equation}
 \label{nev}
 \E(u_\ep, A_\ep)= \hal\i |\nab_{A_\ep} u_\ep -iu_\ep \N \v|^2 + |\curl A_\ep - \N \h|^2 + \frac{(1-|u_\ep|^2)^2}{2\ep^2}+  (1-|u_\ep|^2) \psi
 \end{equation}
 with \begin{equation}\label{choix}
 \psi=\beta  \phi- |\v|^2 .\end{equation}
 We will also use the fact that for $u_\ep$ solution of \eqref{eq} we have the relation 
\begin{equation}\label{divjig}
\div j_\ep= \N \la (\frac{\a}{\lep}+i\b)\p_\Phi u_\ep, iu_\ep\ra.\end{equation}
The stress energy tensor is now
\begin{multline}\label{defT2}
(S_\ep(u,A))_{kl}:=\\
 \la (\p_k u-iA_k u, \p_l u -iA_l u \ra - \hal \( |\nab_A u|^2 + \frac{1}{2\ep^2} (1-|u|^2)^2- |\curl A|^2 \)\delta_{kl}.\end{multline}
A direct computation shows that if $u$ is sufficiently regular, we have
\begin{multline*}\div S_\ep(u,A) :=\sum_l \p_l (S_\ep(u,A))_{kl}\\
= \la \nab_A u, \nab_A^2 u + \frac{1}{\ep^2} u (1-|u|^2)\ra - \curl A (  \np \curl A + \la iu, \nab_A u\ra  )^\perp \end{multline*} so
 if $u_\ep$ solves \eqref{eq}, we have 
\begin{equation}\label{divTg}
\div S_\ep(u_\ep,A_\ep)=\N \la   (\frac{\a}{\lep} + i \b) \p_\Phi u_\ep, \nab_{A_\ep} u_\ep \rangle  +\sigma  E_\ep^\perp\curl A_\ep    .\end{equation}
For simplicity, we will denote it by $S_\ep$.
The modulated stress tensor is now defined by 
\begin{multline}
(\T (u,A))_{kl}= \la \p_k u -iu A_k-iu  \N \v_k, \p_l u -iu A_l u- iu \N \v_l\ra  + \N^2 (1-|u|^2 ) \v_k \v_l 
\\- \hal \( |\nab_A u-iu\N\v |^2 + (1-|u|^2) \N^2|\v|^2+ \frac{1}{2\ep^2} (1-|u|^2)^2- |\curl A-\N \h|^2\)\delta_{kl},\end{multline}
and by  direct computation
\begin{multline}\label{reldivtg}
\div \T = \div S_\ep + \N^2 \v \div \v + \N^2\v^\perp\curl \v -\N j_\ep  \div \v\\
-\N j_\ep^\perp \curl \v
-\N \v  \div j_\ep- \N \v^\perp \curl j_\ep +  \nab(\hal \N^2 \h^2- \N \h h_\ep) .
\end{multline}
Combining \eqref{divTg} and \eqref{reldivtg} and  we obtain
\begin{multline} \label{divttg}
\div \T=  \N \la   (\frac{\a}{\lep} + i \b) \p_\Phi u_\ep, \nab_{A_\ep} u_\ep \rangle  + \sigma h_\ep E_\ep^\perp + \N^2 \v \div \v + \N^2\v^\perp\curl \v\\ -j_\ep \N \div \v-\N j_\ep^\perp \curl \v
- \N \v  \div j_\ep- \N \v^\perp \curl j_\ep +  \nab(\hal \N^2 \h^2- \N \h h_\ep) .
\end{multline}

Writing $\p_\Phi u_\ep = \p_{\Phi'} u_\ep+ iu_\ep \N\phi 
 $ and  $\nab_{A}u_\ep= \nab_{A'} u_\ep+ iu_\ep \N \v$, and inserting \eqref{defVg} and 
\eqref{divjig} into \eqref{divttg},  we are led to 
\begin{multline}\label{divtt2g}
\div \T = \frac{\N\a}{\lep}  \la \p_{\Phi'} u_\ep  , \nab_{A'} u_\ep\ra + \sigma h_\ep E_\ep^\perp  +  \nab(\hal \N^2 \h^2- \N \h h_\ep)\\ +\frac{\b}{2}\N V_\ep -\frac{ \beta}{2} \N^2 \v\p_t |u_\ep|^2  + \N^2\v^\perp\curl \v -\N j_\ep^\perp \curl \v
- \N \v^\perp \curl j_\ep+o(1).\end{multline}
Multiplying relation \eqref{divtt2g}  by $2\b\v+2\a\v^\perp$  yields 
\begin{align*}
& \i   ( 2\b\v +2\a \v^\perp) \cdot \div \T 
\\ &=  \i     \left(   \frac{\N\a}{\lep} \la \p_{\Phi'} u_\ep , \nab_{A'}  u_\ep\ra +
 \frac{\b}{2}\N V_\ep -\frac{\b}{2} \N^2 \v \p_t |u_\ep|^2 \right) \cdot (2\b \v + 2\a \v^\perp)  
 \\ & + \i\   2\a \N^2 |\v|^2 \curl \v + \N j_\ep \cdot (2\b\v^\perp -2\a \v)  \curl \v
  -2\a \N|\v|^2 \curl j_\ep   \\ & + \i \nab(\hal \N^2 \h^2- \N \h h_\ep)\cdot (2\b \v + 2\a \v^\perp)-\sigma h_\ep E_\ep (2\beta \v^\perp - 2\a\v)+o(1).
\end{align*}This allows to express $\i \N \b^2 V_\ep \cdot \v $ so that 
splitting  $ V_\ep \cdot \v $ as $\a^2 V_\ep \cdot \v+ \b^2 V_\ep \cdot \v$ (since $\a^2+\beta^2=1$)  and inserting the previous relation we may obtain
\begin{align}\label{neVg}
& -\i  \N V_\ep \cdot \v  =\i    \N   V_\ep \cdot (-\a^2\v +\a \b \v^\perp) -\N^2 \b^2|\v|^2 \p_t |u_\ep|^2 \\ \nonumber &
 +\i ( 2\b\v +2\a \v^\perp) \cdot \( - \div \T(u_\ep)+ \frac{\N\a}{\lep} \la \p_{\Phi'} u_\ep, \nab_{A'} u_\ep\ra 
\) \\ \nonumber &
+\i   2\a \N^2 |\v|^2 \curl \v 
+ \N j_\ep \cdot (2\b\v^\perp -2\a \v)  \curl \v
  -2\a \N|\v|^2 \curl j_\ep   \\ \nonumber
 & +\i  \nab(\hal \N^2 \h^2- \N \h \curl A)\cdot (2\b \v + 2\a \v^\perp)-  \sigma h_\ep E_\ep(2\beta \v^\perp- 2\a \v)+o(1) .
\end{align}
The analogue of Lemma \ref{lemdte} combined with \eqref{le1} or \eqref{le2} is 
\begin{align*}
& \frac{d}{dt}\E(u_\ep) = -\i   \frac{\N\a}{\lep} |\p_\Phi u_\ep|^2 + \sigma  |E_\ep|^2 +\i \N\la iu_\ep, \p_\Phi u_\ep \ra \div \v   \\ &  +\i \N^2 \v\cdot \p_t \v
-  \N j_\ep  \cdot \p_t \v+ \N ( -V_\ep-E_\ep)  \cdot  \v +\N \p_t \h ( \N \h -h_\ep) - \N \h \p_t h_\ep\\ & + \i \hal \N^2\p_t\((1- |u_\ep|^2)  (\psi-|\v|^2)\).
\end{align*} Using \eqref{le1} or \eqref{le2}, we may write 
\begin{multline*}
\i \N^2 \v\cdot \p_t \v  - \N j_\ep  \cdot \p_t \v \\
= \i \N  (\N \v-j_\ep) \cdot \(\mathrm{E}+ (-2\a \v+ 2\b\v^\perp)(\curl \v + \h) + \nab \phi\).
\end{multline*}
Next, we insert  again $\p_\Phi u_\ep=\p_{\Phi'}u_\ep+ iu_\ep\N \phi$ and write that 
\begin{equation}\label{nvj}
\N \v - j_\ep =  - \np (\N \h - h_\ep) - \sigma (\N \mathrm{E}-E_\ep)\end{equation} which holds  by subtracting the equations \eqref{glg} and \eqref{le1}--\eqref{le2},
and  we insert \eqref{neVg} and use an integration by parts to obtain
\begin{align*}
&\frac{d}{dt} \E= \i \T : \nab (2\beta \v + 2\a \v^\perp) + \N  V_\ep \cdot (-\a^2 \v + \a \b \v^\perp)  -\N^2 \b^2|\v|^2 \p_t |u_\ep|^2\\
&  -   \i  2\a \N |\v|^2 \curl j_\ep +  \N \phi \div (\N \v-j_\ep)
\\
&  -\i   \frac{\N\a}{\lep} |\p_{\Phi'} u|^2 - \frac{\N\a}{\lep} \la \p_{\Phi'} u, \nab_{A'} u\ra \cdot (2\beta \v+ 2\a \v^\perp) \\
& +\i\N ( - \np (\N \h - h_\ep) - \sigma (\N \mathrm{E}-E_\ep))\cdot ((-2\a \v + 2\b \v^\perp) \h+\mathrm{E}) \\& + \i  - \sigma |E_\ep|^2   - \N E_\ep \cdot \v - \N \curl \mathrm{E} (\N \h - h_\ep)+\N \h \curl E_\ep
   \end{align*}
\begin{align*}
 & +\i  \nab(\hal \N^2 \h^2- \N \h h_\ep)\cdot (2\b \v + 2\a \v^\perp)- \sigma h_\ep E_\ep (2\beta \v^\perp - 2\a\v)
  \\ & + \i \hal \N^2\p_t\((1- |u_\ep|^2)  (\psi-|\v|^2)\)
    +o(1).\end{align*}
    Next, we replace $V_\ep$ by $\tilde V_\ep+\N \v\p_t |u_\ep|^2 - \N \phi\nab |u_\ep|^2$ obtained by comparing \eqref{defVg} and \eqref{tvvv},  and use \eqref{divjig} to reexpress $\div j_\ep$.  By choice \eqref{choix}, the terms in factor of $\p_t (1-|u_\ep|^2)$ cancel, and the terms in factor of $\phi, \div \v$ and $\la \p_{\Phi} u_\ep, iu_\ep\ra$ formally cancel (up to small errors) as in the parabolic case. 
    Then, we insert 
    \begin{align*}& \sigma h_\ep E_\ep \cdot (2\beta \v^\perp -2 \a \v)\\ &= 
\sigma h_\ep (E_\ep - \N \mathrm{E})  \cdot (2\beta\v^\perp-2\a\v)+
\N h_\ep( -\v -\np \h ) \cdot (2\beta \v^\perp-2 \a \v) 
\\ & = 
\sigma h_\ep (E_\ep - \N \mathrm{E})  \cdot (2\beta\v^\perp-2\a\v) + 2\a\N  |\v|^2 h_\ep  - \N h_\ep \np \h \cdot ( 2\beta \v^\perp-2\a \v).
\end{align*}
The term  $2\a\N |\v|^2 h_\ep$ gets grouped with $2\a \N |\v|^2 \curl j_\ep $ to form $-2\a |\v|^2 \mu_\ep$.
Combining all these elements, there remains
\begin{align*}
& \frac{d}{dt} \E= \i \T : \nab (2\beta \v + 2\a \v^\perp) + 
\N \tilde V_\ep \cdot (-\a^2 \v + \a \b \v^\perp)   -  2\a \N |\v|^2 \mu_\ep \\
 & \i  -  \frac{\N\a}{\lep} |\p_{\Phi'} u|^2 + \frac{\N\a}{\lep} \la \p_{\Phi'} u, \nab_{A'} u\ra \cdot (2\beta \v+ 2\a \v^\perp) \\ &
 +\i - \N \np (\N \h - h_\ep) \cdot ((-2\a \v + 2\b \v^\perp)  \h  + \mathrm{E})\\
 & + \i  \sigma     ( 2\a \v - 2\b \v^\perp) (\N\mathrm{E}-E_\ep)(\N \h-h_\ep) - \N \curl \mathrm{E} (\N \h - h_\ep)\ \\
& + \i  -\sigma \N ( \N \mathrm{E}-E_\ep) \N\mathrm{E} - \sigma |E_\ep|^2  +\N \h \curl E_\ep - \N E_\ep \cdot \v \\ &
  +\i  \nab(\hal \N^2 \h^2- \N \h h_\ep) \cdot (2\b \v + 2\a \v^\perp)+  \N h_\ep \nab \h \cdot (2\beta \v+2\a \v^\perp)+o(1).\end{align*}
Keeping only the last three lines, after some simplifications using again \eqref{le1}, and some integration by parts, these three lines can be rewritten as
\begin{align*} &\i -  \N \h\nab (\N \h - h_\ep) \cdot (2\a \v^\perp + 2\b \v)   + \sigma (\N\mathrm{E}-E_\ep) ( 2\a \v - 2\b \v^\perp) (\N \h-h_\ep) \\
& \i  - \sigma |E_\ep|^2   - \N E_\ep \cdot \v+\N \h \curl E_\ep
  \\ &+\i\N \h  \nab(\N\h -  h_\ep)\cdot (2\b \v + 2\a \v^\perp) -\sigma\N (\N\mathrm{E}-E_\ep) \mathrm{E}+o(1)\\
 &  = \i \sigma (\N\mathrm{E}-E_\ep) ( 2\a \v - 2\b \v^\perp) (\N \h-h_\ep) 
+  \i  - \sigma |E_\ep-\N\mathrm{E}|^2 +o(1)  .
\end{align*}
We thus finally obtain 
\begin{align}\label{infg}
& \frac{d}{dt} \E= \i \T : \nab (2\beta \v + 2\a \v^\perp) + 
\i \N \tilde V_\ep \cdot (-\a^2 \v + \a \b \v^\perp)   - \i 2\a \N |\v|^2 \mu_\ep \\
\nonumber &  \i  -  \frac{\N\a}{\lep} |\p_{\Phi'} u_\ep|^2 + \frac{\N\a}{\lep} \la \p_{\Phi'} u_\ep, \nab_{A'} u_\ep\ra \cdot (2\beta \v+ 2\a \v^\perp) \\
 \nonumber & \i \sigma (\N\mathrm{E}-E_\ep) ( 2\a \v - 2\b \v^\perp) (\N \h-h_\ep) 
+  \i  - \sigma |E_\ep-\N\mathrm{E}|^2   +o(1).
\end{align}
For the terms involving $E$ we write with Young's inequality
$$ \i \sigma (\N\mathrm{E}-E_\ep) ( 2\a \v - 2\b \v^\perp) (\N \h-h_\ep) 
\le \sigma \i |E_\ep - \N \mathrm{E}|^2+ C \i |\N \h - h_\ep|^2$$
and we deduce that the last line in \eqref{infg} is bounded above by $C \i |\N \h - h_\ep|^2$, which we claim it itself bounded by a constant time the energy excess  $\E-\pi \N \lep$. Indeed, in writing down the analogue of Proposition \ref{lem43} and \ref{exces}, one may replace $\i |\curl A_\ep- \N \h|^2$ in \eqref{nev} by half of itself and still obtain the same optimal lower bounds for the energy in the balls, so we deduce that $\i|\curl A_\ep- \N \h|^2$ must be bounded by the order of the energy excess. 

To finish, we use Young's inequality again  to bound \begin{multline*}
\i \frac{\N\a}{\lep} \la \p_{\Phi'} u_\ep, \nab_{A'} u_\ep\ra \cdot (2\beta \v+ 2\a \v^\perp)\\ \le \frac{\N\a}{\lep}\(\hal  \i |\p_{\Phi'} u_\ep|^2 + \hal \i |\nab_{A'}u_\ep\cdot (\b \v+\a\v^\perp)|^2\)\end{multline*}
and the product estimate to formally bound 
\begin{multline*}
\i \N \tilde V_\ep \cdot (-\a^2 \v + \a \b \v^\perp)  \\
\le \frac{\N\a}{\lep} \(\hal  \i |\p_{\Phi'} u_\ep|^2 + 2 \i |\nab_{A'}u_\ep\cdot (-\a \v+\b\v^\perp)|^2\)
\end{multline*}
 and using that $\a\v-\b\v^\perp =(\b\v+\a\v^\perp)^\perp$ and $|\b\v+\a\v^\perp|^2=|\v|^2$, we see that these two relations add up to a left-hand side bounded by 
\begin{multline*}
 \frac{\N\a}{\lep} \( \i |\p_{\Phi'} u_\ep|^2 + 2 \i |\nab_{A'}u_\ep|^2 |\v|^2\)\end{multline*} which will recombine with $ - \i  \frac{\N\a}{\lep} |\p_{\Phi'} u_\ep|^2 +2\a \N |\v|^2 \mu_\ep$ into  a term bounded by $C(\E-\pi \N \lep)$. 
 Inserting into \eqref{infg} and replacing the use of $\v$ by that of $\bar \v$ in the case with dissipation, we may then obtain a Gronwall relation $\frac{d}{dt}\E\le C(\E-\pi \N \lep)+o(\N^2)$ as in the case without gauge, and conclude in the same way.


\begin{thebibliography}{99}
%
%
\bibitem[AGS]{ags} L. Ambrosio, N. Gigli, G. Savar\'e, {\em Gradient flows in metric spaces and in the Wasserstein space of probability measures}, Birkh\"auser, (2005).

\bibitem[AS]{as} L. Ambrosio, S. Serfaty, A gradient flow approach to an evolution problem arising in
superconductivity,  {\it  Comm. Pure
Appl. Math.}    {\bf 61}  (2008),  no. 11, 1495--1539.

\bibitem[BBH]{bbh}F. Bethuel, H. Brezis and F. H\'elein, {\em
Ginzburg-Landau
  Vortices}, Birkh\"auser, (1994).
  \bibitem[AK]{ak}I. S. Aranson, L. Kramer, The world of the complex Ginzburg-Landau equation, {\it Rev. Mod. Phys. } {\bf 74} (2002), 99-143.
  
  \bibitem[BCD]{bcd} H. Bahouri, J.-Y. Chemin, R. Danchin, {\em Fourier analysis and nonlinear partial differential equations}, Springer, 2011.
  
  \bibitem[BJS]{bjs} F. Bethuel, R. L. Jerrard, D. Smets, On the NLS dynamics for infinite energy vortex configurations on the plane, {\it 
Rev. Mat. Iberoam.} {\bf  24} (2008), no. 2, 671--702. 

\bibitem[BOS1]{bos1}F. Bethuel, G. Orlandi, D.  Smets, Collisions and phase-vortex interactions in dissipative Ginzburg-Landau dynamics. {\it  Duke Math. J.} {\bf  130} (2005),  no. 3, 523--614.
\bibitem[BOS2]{bos2}F. Bethuel, G. Orlandi, D.  Smets, Quantization and motion law for Ginzburg-Landau vortices. {\it Arch. Ration. Mech. Anal.} {\bf 183} (2007),  no. 2, 315--370.
\bibitem[BOS3]{bos3}F. Bethuel, G. Orlandi, D.  Smets, Dynamics of multiple degree Ginzburg-Landau vortices. {\it Comm. Math. Phys.} {\bf  272}  (2007),  no. 1, 229--261.


\bibitem[BR1]{br}F. Bethuel, T. Rivi\`ere, Vortices for a Variational Problem
Related to Superconductivity, {\em Annales IHP, Anal. non
lin.}  {\bf 12} (1995), 243-303.
 \bibitem[BDGSS]{bdgs} F. Bethuel, R. Danchin, P. Gravejat, J.C. Saut, D. Smets, Les \'equations d'Euler, des ondes et de Korteweg-de Vries comme limites asymptotiques de l'\'equation de Gross-Pitaevskii, s\'eminaire E. D. P (2008-2009), Exp. I, {\it  \'Ecole Polytechnique}, 2010.
   \bibitem[BS]{bs} F. Bethuel, D. Smets, A remark on the Cauchy problem for the 2D Gross-Pitaevskii equation with nonzero degree at infinity. {\it Diff. Int. Eq.} {\bf 20} (2007), no. 3, 325--338.
  
 \bibitem[Br]{brenier} Y. Brenier, Convergence of the Vlasov-Poisson system to the incompressible Euler equations, {\it Comm. PDE} {\bf 25}, (2000), 737--754.
 
 
 \bibitem[CDS]{danchin} R. Carles, R. Danchin, J.-C. Saut, Madelung, Gross-Pitaevskii and Korteweg, {\it Nonlinearity} {\bf 25} (2012), no. 10, 2843--2873. 
\bibitem[CRS]{crs}S. J. Chapman, J. Rubinstein,   M. Schatzman, A
  Mean-field Model of Superconducting Vortices, {\it Eur. J. Appl. Math.} {\bf  7}, No. 2, (1996),
   97--111.
  \bibitem[Ch1]{chemin} J.-Y. Chemin, {\em Perfect incompressible fluids}.  Oxford Lecture Series in Mathematics and its Applications, 14. T Oxford University Press, 1998.
  \bibitem[Ch2]{chemin2} J.-Y. Chemin, Th\'eor\`emes d'unicit\'e pour le syst\`eme de Navier-Stokes tridimensionnel, {\it J. Anal. Math.} {\bf 77} (1999), 77, no. 1, 27--50.
     \bibitem[CJ1]{cj1}J. Colliander, R. L. Jerrard, Vortex dynamics for the
Ginzburg-Landau-Schr\"odinger equation, {\it Inter. Math. Res.
Notices} {\bf 7} (1998), 333--358.
\bibitem[CJ2]{cj2}J. Colliander, R. L. Jerrard, Ginzburg-Landau
vortices: weak stability and Schr\"odinger equations dynamics,
{\it J. Anal. Math.} { \bf 77}, (1999), 129-205.
\bibitem[De]{delort} J.-M. Delort, Existence de nappes de tourbillon en dimension deux, {\it JAMS} {\bf 4} (1991), 553--586.
\bibitem[Do]{dorsey} A. Dorsey, Vortex motion and the Hall effect in type II superconductors: a time-dependent Ginzburg-Landau approach. {\it Phys. Rev. B}  {\bf 46} (1992), 8376--8392.
  \bibitem[DZ]{duzhang}Q. Du, P. Zhang, Existence of weak solutions to some vortex density models.
{\it SIAM J. Math. Anal.} {\bf 34} (2003), 1279--1299.

\bibitem[Du]{du} M. Duerinckx, Well-posedness for mean-field evolutions arising in superconductivity, forthcoming.


   \bibitem[E1]{E}{W. E, Dynamics of vortices in Ginzburg-Landau
  theories with applications to superconductivity, {\em Phys. D,
  77,}(1994), 383-404.}
 \bibitem[E2]{E2}W. E,  Dynamics of vortex liquids in Ginzburg-Landau theories with applications to superconductivity, {\em Phys. Rev. B},  {\bf 50} (1994), No. 2, 1126--1135.  
\bibitem[Ga]{gallo} C. Gallo, The Cauchy problem for defocusing nonlinear Schr\"odinger equations with non-vanishing initial data at infinity. {\it Comm. PDE} {\bf 33} (2008), no. 4-6, 729--771.
\bibitem[Ge]{gerard} P. G\'erard, The Cauchy problem for the Gross-Pitaevskii equation, 
{\it Ann. Inst. H. Poincar\'e Anal. Non Lin.} {\bf  23} (2006), no. 5, 765--779. 
\bibitem[GE]{ge} L. P. Gor'kov, G. M.  Eliashberg,  Generalization of the Ginzburg-Landau equations for nonstationary
problems in the case of alloys with paramagnetic impurities. {\it Sov. Phys. JETP} {\bf 27}
(1968), 328--334.

\bibitem[HL]{hanli}Z.-C. Han, Y.-Y. Li, Degenerate elliptic systems and applications to Ginzburg-Landau type equations, part I, {\it CVPDE}, {\bf 4} (1996), No.  2,  171--202.

\bibitem[J1]{j}R. L. Jerrard, Lower Bounds for Generalized Ginzburg-Landau
  Functionals, {\em SIAM Journal Math. Anal.} {\bf 30},  (1999), No. 4,  721--746.
  \bibitem[J2]{jerrard3d} R. L. Jerrard, Vortex filament dynamics for Gross-Pitaevsky type equations, {\em Ann. Scuola Norm. Sup. Pisa Cl. Sci  (5)} {\bf 1} (2002),  733--768.
\bibitem[JS1]{js}R. L.  Jerrard, H.M.
Soner, Dynamics of Ginzburg-Landau vortices, {\em  Arch. Rat.
Mech. Anal.} {\bf 142} (1998), No. 2, 99--125.
\bibitem[JS2]{js2}R. L.
Jerrard, H.M. Soner, The Jacobian and the Ginzburg-Landau
functional, {\em Calc. Var. PDE. 14, No 2}, (2002), 151-191.
\bibitem[JSp1]{jsp1} R. L. Jerrard, D.  Spirn, Refined Jacobian estimates and Gross-Pitaevsky vortex dynamics.  {\it Arch. Rat. Mech. Anal.} {\bf 190} (2008), no. 3, 425--475.
\bibitem[JSp2]{jsp2} R. L.  Jerrard, D.  Spirn,  Hydrodynamic limit of the Gross-Pitaevskii equation. {\it Comm. PDE} {\bf 40} (2015), no. 2, 135--190.
\bibitem[KIK]{kik} N. B. 
Kopnin, B. I.  Ivlev, V. A. Kalatsky,  The flux-flow Hall effect in type II superconductors.
An explanation of the sign reversal. {\it J. Low Temp. Phys.} {\bf 90} (1993),
1--13.
\bibitem[KS1]{ks0} M.  Kurzke, D.  Spirn,  $\Gamma$-stability and vortex motion in type II superconductors. {\it Comm. PDE} {\bf 36} (2011), no. 2, 256--292.
  \bibitem[KS2]{ks} M. Kurzke, D. Spirn, Vortex liquids and the Ginzburg-Landau equation. {\it Forum Math. Sigma} {\bf 2} (2014), e11, 63 pp.
 \bibitem[KMMS]{kmms} M. Kurzke, C.  Melcher, R.  Moser, D. Spirn,
Dynamics for Ginzburg-Landau vortices under a mixed flow. {\it Indiana Univ. Math. J.} {\bf 58} (2009), no. 6, 2597--2621. 
\bibitem[Li1]{li}F.H. Lin, Some Dynamical Properties of Ginzburg-Landau
Vortices, {\it Comm. Pure Appl. Math.} {\bf 49}, (1996), 323-359.
\bibitem[Li2]{li2}F.H. Lin, Vortex Dynamics  for the Nonlinear Wave Equation,
{\em Comm. Pure Appl. Math., 52}, (1999), 737-761.
\bibitem[LX1]{lx}F. H. Lin, J. X. Xin, 
On the dynamical law of the Ginzburg-Landau vortices on the plane, {\it Comm. Pure Appl. Math.} {\bf 52} (1999), no. 10, 1189--1212.
\bibitem[LX2]{lx2}F. H. Lin, J. X. Xin, 
On the incompressible fluid limit and the vortex motion law of the nonlinear Schr\"odinger equation, {\it Comm. Math. Phys.} {\bf  200} (1999), no. 2, 249--274. 
\bibitem[LZ1]{lz} F. H. Lin,  P. Zhang,  {On the hydrodynamic limit of
Ginzburg-Landau vortices,} {\it  Disc. Cont. Dyn. Systems}  {\bf 6}
(2000), 121--142.
\bibitem[LZ2]{lz2} F. H. Lin, P. Zhang, Semi-classical Limit of the Gross-Pitaevskii Equation in an Exterior Domain, {\it Arch. Rat. Mech. Anal.} {\bf 179} (2005), 79--107. 
\bibitem[MB]{bertom} A. Majda, A. Bertozzi, {\em Vorticity and Incompressible Flow}, Cambridge texts in applied mathematics, Cambridge University Press, 2002.
\bibitem[MZ]{mz} N. Masmoudi, P. Zhang, Global solutions to vortex density equations arising from sup-conductivity, {\it Ann. Inst. H. Poincar\'e  Anal. non lin\'eaire} {\bf 22}, (2005), 441--458.
\bibitem[Mi]{miot} E. Miot, Dynamics of vortices for the complex Ginzburg-Landau equation. {\it Anal. PDE} {\bf 2}  (2009), no. 2, 159--186.
\bibitem[O]{otto} F. Otto, The geometry of dissipative evolution
equations: the porous medium equation, {\it Comm. PDE} {\bf 26},
(2001), 101--174.
\bibitem[RuSt]{rust} J.  Rubinstein, P.
Sternberg,  On the slow motion of vortices in the Ginzburg-Landau
heat flow, {\em SIAM J. Math. Anal. 26, no 6} (1995), 1452--1466.
\bibitem[PR]{pr} L. Peres, J. Rubinstein, Vortex Dynamics in $U(1)$ Ginzburg-Landau models, {\it Phys. D} {\bf 64} (1993) 299-309.
\bibitem[SR]{laure} L. Saint Raymond, {\em Hydrodynamic limits of the Boltzmann equation}. Lecture Notes in Mathematics, 1971, Springer, (2009).
\bibitem[Sa]{sa}E.
Sandier, Lower Bounds for the Energy of Unit Vector
  Fields and Applications, {\em J. Funct. Anal.} {\bf 
 152}  (1998), No. 2,  379--403.
 \bibitem[Sch]{schmid}
 A. Schmid,  A time dependent Ginzburg-Landau equation and its application to the problem of resistivity in the mixed state, {\it Phys. Kondens. Materie} {\bf 5}  (1966), 302--317.
\bibitem[SS1]{ss1} E. Sandier, S. Serfaty, Global Minimizers for the Ginzburg-Landau Functional below the First Critical Magnetic Field, {\em Annales IHP, Analyse non lin\'eaire.} {\bf 17} (2000), No. 1,  119--145.
\bibitem[SS2]{ss4}E. Sandier,  S. Serfaty, Limiting Vorticities
for the Ginzburg-Landau Equations, {\em Duke Math. J.} {\bf 117} (2003), No. 3, 403--446.

\bibitem[SS3]{ss6}{E. Sandier, S. Serfaty, A product estimate
for Ginzburg-Landau and corollaries,}  {\it J. Func. Anal.} {\bf 211} (2004),  No. 1, 219--244. 
\bibitem[SS4]{ss}
E. Sandier, S. Serfaty, Gamma-convergence of gradient flows with applications to Ginzburg-Landau, {\it  Comm. Pure Appl. Math.} {\bf 57}, No 12, (2004), 1627--1672.

\bibitem[SS5]{livre} E. Sandier, S. Serfaty,  {\it Vortices in the Magnetic Ginzburg-Landau Model}, Progress in Nonlinear Differential Equations and their Applications, vol 70, Birkh\"auser, (2007).  

\bibitem[Sch]{scho}S. Schochet, The point vortex method  for periodic weak solutions of the 2D Euler equations, {\it Comm. Pure Appl. Math.} {\bf 49} (1996), 911--965.

    \bibitem[Se1]{s2}S. Serfaty,
     Vortex collisions and energy-dissipation rates in the Ginzburg-Landau heat flow, part I: Study of the perturbed Ginzburg-Landau equation, {\it JEMS} {\bf  9}, (2007),  No 2, 177--217. part II: The dynamics, {\it JEMS} {\bf  9} (2007), No. 3, 383--426.
     \bibitem[Se2]{serfaty} S. Serfaty, Gamma-convergence of gradient flows on Hilbert and metric spaces and applications, {\it Disc. Cont. Dyn. Systems, A} {\bf 31}  (2011), No. 4, 1427--1451. 
\bibitem[ST1]{stice1} S. Serfaty, I. Tice, Lorentz Space Estimates for the Ginzburg-Landau Energy, {\it J. Func. Anal.} {\bf  254} (2008), No 3, 773--825.
\bibitem[ST2]{stice} S. Serfaty, I. Tice, Ginzburg-Landau vortex dynamics with pinning and strong applied currents, {\it  Arch. Rat. Mech. Anal.} {\bf  201}  (2011), No. 2,  413--464. 
\bibitem[SV]{sv} S. Serfaty, J. L. Vazquez, A Mean Field Equation as Limit of Nonlinear Diffusions with Fractional Laplacian Operators, {\it Calc Var. PDE} {\bf  49} (2014), no. 3-4, 1091--1120. 
\bibitem[Sp1]{sp1}D.
Spirn, Vortex dynamics for the full time-dependent Ginzburg-Landau
equations, {\it Comm. Pure Appl. Math.} {\bf 55} (2002), No. 5, 537--581.
\bibitem[Sp2]{sp2}D. Spirn, Vortex motion law for the Schr\"odinger-Ginzburg-Landau equations, {\it  SIAM J. Math. Anal.} {\bf  34} (2003), no. 6, 1435--1476.
\bibitem[Ti]{tice} I. Tice, Ginzburg-Landau vortex dynamics driven by an applied boundary current, {\it Comm. Pure Appl. Math. } {\bf 63} (2010), no. 12, 1622--1676. 
\bibitem[TT]{tt}J. 
Tilley,  D. Tilley, 
  {\it Superfluidity and superconductivity} Second edition.
   Adam Hilger, 1986.


\bibitem[T]{t}M. Tinkham, {\em Introduction to Superconductivity, 2d edition}, McGraw-Hill, (1996).
\bibitem[Yau]{yau} H. T. Yau, Relative entropy and hydrodynamics of Ginzburg-Landau models. {\it Lett. Math.
Phys. } {\bf 22} (1991), 63--80.
%
\end{thebibliography}
\end{document}